\documentclass[11pt,reqno,a4paper]{amsart}

\usepackage[utf8]{inputenc}
\usepackage[T1]{fontenc}
\usepackage{lmodern}
\usepackage{microtype}
\usepackage[margin=2.4cm]{geometry}
\usepackage{enumitem}
\usepackage{xcolor}

\usepackage{amsmath,amssymb,mathtools}
\usepackage{amsthm}
\usepackage{mathdots}
\usepackage{easybmat}

\usepackage{tikz}
\usetikzlibrary{decorations.pathreplacing}
\usepackage{dynkin-diagrams}
\usepackage{rotating}
\usepackage{wasysym}
\usepackage{subcaption}
\usepackage{multirow}
\usepackage{float}

\usepackage[usenames,dvipsnames,svgnames,table]{xcolor}
\definecolor{todocolor}{HTML}{D7E1E5}
\usepackage[textsize=tiny,
            backgroundcolor=todocolor,
            bordercolor=todocolor,
            linecolor=todocolor,
            textwidth=3.9cm]{todonotes}

\usepackage{xspace}
\usepackage{comment}
\usepackage{tabularx}
\usepackage{xltabular}
\usepackage{cite}
\usepackage{bibentry}

\usepackage{xr-hyper}
\usepackage[
pdftex,
bookmarks=true,
colorlinks=true,
debug=true,
pdfnewwindow=true]{hyperref}
\hypersetup{bookmarksdepth = 3}

\makeatletter
\let\@wraptoccontribs\wraptoccontribs
\makeatother

\theoremstyle{plain}
\newtheorem{Thm}[equation]{Theorem}

\newtheorem{Prop}[equation]{Proposition}

\newtheorem{Def}[equation]{Definition}

\newtheorem*{Thm-non}{Theorem}

\theoremstyle{remark}

\newtheorem{Rem}[equation]{Remark}

\numberwithin{equation}{section}

\makeatletter
\g@addto@macro\bfseries{\boldmath}
\makeatother

\makeatletter
\newcases{rrcases}{\quad}{%
  \hfil$\m@th\displaystyle{##}$}{$\m@th\displaystyle{##}$\hfil}{\lbrace}{.}
\makeatother

\newcommand{\ol}{\overline}
\newcommand{\bbar}[1]{\bar{\mathbf{#1}}}

\newcommand{\iso}{\,\vphantom{j^{X^2}}\smash{\overset{\sim}{\vphantom{\rule{0pt}{0.20em}}\smash{\longrightarrow}}}\,}

\newcommand{\wt}{\widetilde}
\newcommand{\sfrac}[2]{{\textstyle\frac{#1}{#2}}}

\newcommand{\BZ}{\mathbb{Z}}
\newcommand{\BC}{\mathbb{C}}

\newcommand{\BK}{\mathbb{K}}

\newcommand{\fg}{\mathfrak{g}}
\newcommand{\fh}{\mathfrak{h}}

\newcommand{\gl}{\mathfrak{gl}}

\newcommand{\fosp}{\mathfrak{osp}}

\newcommand{\fR}{\mathfrak{R}}

\newcommand{\CA}{\mathcal{A}}

\newcommand{\CO}{\mathcal{O}}
\newcommand{\CR}{\mathcal{R}}

\newcommand{\CT}{\mathcal{T}}

\newcommand{\CY}{\mathcal{Y}}

\newcommand{\sfh}{{\mathsf{h}}}

\newcommand{\sfs}{{\mathsf{s}}}
\newcommand{\sft}{{\mathsf{t}}}

\newcommand{\sT}{\mathsf{T}}

\newcommand{\DrJ}{\mathrm{DJ}}

\newcommand{\RLL}{\textrm{RLL}}

\newcommand{\uqV}{U_{q}(\fosp(V))}
\newcommand{\UdqV}{U_{q}(\widehat{\fosp}(V))}

\newcommand{\aff}{\mathrm{aff}}

\newcommand{\End}{\mathrm{End}}

\newcommand{\opp}{\mathrm{op}}

\newcommand{\ID}{\mathrm{I}}
\newcommand{\rmu}{\mathrm{u}}
\newcommand{\rms}{\mathrm{s}}

\DeclareMathOperator{\Id}{Id}
\DeclareMathOperator{\sVect}{sVect}

\DeclareMathOperator{\Span}{Span}
\DeclareMathOperator{\st}{st}

\newmuskip\pFqmuskip
\newcommand*\pFq[6][8]{%
    \begingroup
    \pFqmuskip=#1mu\relax
    \mathcode`\,=\string"8000
    \begingroup\lccode`\~=`\,
    \lowercase{\endgroup\let~}\pFqcomma
    {}_{#2}F_{#3}{\left(\genfrac..{0pt}{}{#4}{#5};#6\right)}%
    \endgroup}
\newcommand{\pFqcomma}{\mskip\pFqmuskip}

\def\be{\begin{eqnarray}}
\def\ee{\end{eqnarray}}

\begin{document}

\title[Rank reduction in orthosymplectic types]
{\large{\textbf{Rank reduction in orthosymplectic types}}}

\author{Kyungtak Hong}
\address{K.H.: Purdue University, Department of Mathematics, West Lafayette, IN 47907, USA}
\email{hong420@purdue.edu}

\author{Alexander Tsymbaliuk}
\address{A.T.: Purdue University, Department of Mathematics, West Lafayette, IN 47907, USA}
\email{sashikts@gmail.com}

\begin{abstract}
This note establishes a rank reduction for RLL-realization of quantum affine supergroups of orthosymplectic type.
In the rational limit, this provides a more conceptual proof of the analogous result for orthosymplectic superYangians from~\cite{ft}.
We also determine the exact scalar prefactor appearing in the evaluation of the universal $R$-matrix, refining the main result of~\cite{ht1}.
\end{abstract}

\maketitle


\smallskip

\section{Introduction}
\label{sec:intro}


\subsection{Summary}
\label{ssec:summary}
\

For simple Lie algebras $\fg$ of classical type, the quantum groups $U^{\RLL}_q(\fg)$ first implicitly appeared in the work
of Faddeev's school on the \emph{quantum inverse scattering method}, see e.g.~\cite{frt}. In this \emph{$\RLL$-realization},
the algebra generators are encoded by two square matrices $L^\pm$ subject to the famous \emph{$\RLL$-relations}
$RL^\pm_1L^\pm_2=L^\pm_2L^\pm_1R$, $RL^+_1L^-_2=L^-_2L^+_1R$ (along with central reduction), where $R$ is a solution of the
\emph{quantum Yang--Baxter equation} $R_{12}R_{13}R_{23} = R_{23}R_{13}R_{12}$. The Hopf algebra structure on $U^{\RLL}_q(\fg)$
is naturally given via $\Delta(L^\pm)=L^\pm\otimes L^\pm$, $S(L^\pm)=(L^\pm)^{-1}$, $\epsilon(L^\pm)=\ID$.

This was further generalized to the affine setup in~\cite{rs}, giving rise to the algebra $U^{\RLL}_q(\widehat{\fg})$ whose
generators are encoded by two formal power series matrices $L^\pm(u)$ subject to similar \emph{$\RLL$-relations}
\begin{align*}
  R_{12}(u/v)L^{\pm}_{1}(u)L^{\pm}_{2}(v) &= L^{\pm}_{2}(v)L^{\pm}_{1}(u) R_{12}(u/v) \,, \\
  R_{12}(q^{c}u/v)L^{+}_{1}(u)L^{-}_{2}(v) &= L^{-}_{2}(v)L^{+}_{1}(u) R_{12}(q^{-c}u/v) \,,
\end{align*}
where $q^{\pm c}$ are central and $R(z)$ is a solution of the parameter-dependent \emph{quantum Yang--Baxter equation}
\begin{equation*}
  R_{12}(u/v) R_{13}(u/w) R_{23}(v/w) = R_{23}(v/w) R_{13}(u/w) R_{12}(u/v) \,.
\end{equation*}

The uniform definition of quantum groups $U^{\DrJ}_q(\fg)$ and $U^{\DrJ}_q(\widehat{\fg})$ was provided by Drinfeld and Jimbo,
called the \emph{Drinfeld--Jimbo realization}. In this presentation, the generators $e_i,f_i,k^{\pm 1}_i$ are labeled by simple
roots $\alpha_i$ of $\fg$ or $\widehat{\fg}$, while the Hopf algebra structure is given formally by the assignment on generators.
When $\fg$ is of finite type and $q$ is not a root of unity, the representation theory of $U^{\DrJ}_q(\fg)$ is completely parallel
to that of $\fg$. On the other hand, to develop the representation theory of $U^{\DrJ}_q(\widehat{\fg})$ in affine types one needs
an alternative \emph{new Drinfeld} (aka \emph{loop}) realization
$$
  U^{\mathrm{Dr}}_q(\widehat{\fg}) \simeq U^{\DrJ}_q(\widehat{\fg}) \,.
$$

To relate the above algebras $U^{\RLL}_q(\widehat{\fg})$ and $U^{\mathrm{Dr}}_q(\widehat{\fg})$, one considers
the corresponding affine $R$-matrix $R(z)$ arising through the evaluation of the universal $R$-matrix $\fR_{\aff}$
(corresponding to the canonical tensor with respect to the Drinfeld double realization of $U^{\DrJ}_q(\widehat{\fg})$)
on the tensor product of two evaluation modules $V(u)\otimes V(v)$, where $z=u/v$. In type $A$, the corresponding isomorphism
\begin{equation*}
  U^{\mathrm{Dr}}_q(\widehat{\gl}_n) \iso U^{\RLL}_q(\widehat{\gl}_n)
\end{equation*}
was first constructed in~\cite{df} through the Gauss decomposition of the generating matrices $L^\pm(u)$.
The explicit generalization to other classical $BCD$ types appeared much more recently in~\cite{jlm2, jlm3}.
In the above setup, we shall denote $U^{\RLL}_q(\widehat{\fg})$ rather by $U(R(z))$.

The theory of quantum supergroups associated with Lie superalgebras is still far from a full development, though their
Drinfeld--Jimbo realizations (for both finite and affine types) appeared more than two decades ago in~\cite{y}.
One of the major deficiencies lies in the absence of a uniform (new) Drinfeld realization of such algebras akin to
$U^{\mathrm{Dr}}_q(\widehat{\fg})$, except for super $A$ and $D(2,1;\alpha)$ types. Another new feature of the super setup is the
presence of additional \emph{higher-order Serre relations}, which crucially depend on the choice of a Dynkin diagram
of the underlying Lie superalgebra. For $\fg=\fosp(V)$, the so-called \emph{orthosymplectic types} (aka super $BCD$ types),
a conjectured description of the algebra $U^{\mathrm{Dr}}_q(\widehat{\fg})$ appeared in~\cite{bfk}, while its
shuffle realization was established recently in~\cite{nt}.

This note is devoted to establishing the key \emph{rank reduction} structure required to carry out the super-type generalization
of~\cite{jlm2,jlm3}, establishing the new Drinfeld realization of quantum affine orthosymplectic supergroups. Indeed, in analogy
with the superYangian counterpart (aka rational version) carried out in~\cite{ft}, performing the Gauss decomposition on the
generating matrices $L^\pm(u)$, the quadratic relations in the sought-after Drinfeld presentation are derived from the super
$A$ type analogue as well as rank $\leq 2$ cases (the latter ones are handled by brute force).

As noticed in~\cite{m}, the classical argument of~\cite{jlm1,jlm2,jlm3} does not always go through due to certain resonance
conditions of the corresponding $R$-matrix. In~\cite{m} (see also~\cite{ft}) this was bypassed via a tedious linear algebraic
argument. The main objective of this note is to show that keeping track of residues at possible singularities of the $R$-matrix still
allows one to apply the standard argument, see Theorem~\ref{thm:rank-reduction}. The latter is crucially based on a $10$-term matrix
equality~\eqref{eq:fused-RLL-check}, encoding an important \emph{fused structure}. While the same argument can be applied verbatim
to the rational setup, we rather re-derive the corresponding results of~\cite{ft,m} through the trigonometric-to-rational~limit.

We also refine the main result of~\cite{ht1} by explicitly computing the scalar prefactor entering the evaluation $R(z)$ of the
universal orthosymplectic affine $R$-matrix $\fR_{\aff}$. In the non-super case, such factors were derived in~\cite{jlm2,jlm3}
through the fundamental \emph{crossing symmetry}, the detailed proof of which was presented in~\cite[Appendix C]{mt}. In the case
of super $B$ type with the standard Dynkin diagram, this prefactor appeared earlier in~\cite{wlz}.


\subsection{Outline}
\label{ssec:outline}
\

The structure of the present paper is the following:
\begin{itemize}[leftmargin=*]

\item[$\bullet$]
In Section~\ref{sec:preliminaries}, we recall the basic conventions on superalgebras and the normalized
affine orthosymplectic $R$-matrix of~\cite{ht1}.

\item[$\bullet$]
In Section~\ref{sec:RLL-algebra}, we recall the RLL realization $U(\ol{R}(z))$ of the (extended) quantum affine
orthosymplectic supergroup, see Definition~\ref{def:U(R(z))}. We also relate it to the twisted version $\wt{U}(\ol{R}(z))$
from Definition~\ref{def:U-tilde} through the $2$-cocycle twist in Proposition~\ref{prop:affine-twisted-rtt}.

\item[$\bullet$]
In Section~\ref{sec:rank-reduction}, we state the main result of this note, see Theorem~\ref{thm:rank-reduction}.
We also provide a graphical proof of the $10$-term matrix equality~\eqref{eq:fused-RLL-check}, the key ingredient in the proof
of the above theorem, and interpret~\eqref{eq:fused-RLL-check} via the \emph{fused structure} in Remark~\ref{rem:fusion-interpretation}.

\item[$\bullet$]
In Section~\ref{sec:proof-rank-reduction}, we present the proof of Theorem~\ref{thm:rank-reduction}.
In Theorem~\ref{thm:Y-rank-reduction}, we derive a similar result for orthosymplectic superYangians
through the standard trigonometric-to-rational limit.

\item[$\bullet$]
In Appendix~\ref{sec:prefactor}, we refine the main result of~\cite{ht1} by computing explicitly the scalar prefactor
entering the evaluation of the universal affine $R$-matrix, see Proposition~\ref{prop:R-univ-explicit}.

\end{itemize}


\subsection{Acknowledgement}\label{ssec:acknowl}
\

A.T.\ gratefully acknowledges the support by NSF Grant DMS-2302661 and Simons Fellowship MPS-SFM-00020373.


\smallskip

\section{Preliminaries and the affine orthosymplectic $R$-matrix}\label{sec:preliminaries}


\subsection{Notations and setups}\label{ssec:notations-setups}
\

Our notations and conventions regarding superspaces and superalgebras closely follow~\cite[\S2.1]{ht2}.
The category $\sVect_{\BK}$ of superspaces $V = V_{\bbar{0}} \oplus V_{\bbar{1}}$ over a field $\BK$
(with degree-preserving linear maps as morphisms) is a symmetric tensor category with the braiding given
by the graded permutation:
\begin{equation*}
  \tau_{VW} \colon V \otimes W \to W \otimes V\,, \qquad  v \otimes w \mapsto (-1)^{|v||w|} w \otimes v
\end{equation*}
for $\BZ_{2}$-homogeneous elements $v \in V, w \in W$ of respective $\BZ_{2}$-degrees $|v|, |w| \in \BZ_{2}$,
where we set $(-1)^{\bbar{0}} = 1$ and $(-1)^{\bbar{1}} = -1$.

For any superalgebras $A, B$ in $\sVect_{\BK}$, their
tensor product $A \otimes B$ is a superalgebra with multiplication given by
$$
  (a \otimes b)(a' \otimes b') = (-1)^{|b||a'|} (aa') \otimes (bb')
$$
for $\BZ_{2}$-homogeneous elements $a, a' \in A$ and $b, b' \in B$. Analogously, the action of $A \otimes B$ on
the tensor product $V \otimes W$ of representations $V$ of $A$ and $W$ of $B$ (in particular, the action of
$\End_{\BK}(V) \otimes \End_{\BK}(W)$ on $V \otimes W$) follows the same sign convention:
\begin{equation*}
  (a \otimes b)(v \otimes w) = (-1)^{|b||v|} (av) \otimes (bw)
\end{equation*}
for $\BZ_{2}$-homogeneous elements $a \in A$, $b \in B$, $v \in V$, and $w \in W$.

For any superalgebra $A$ in $\sVect_{\BK}$, its opposite superalgebra $A^{\opp}$ is the superalgebra with the same underlying
superspace as $A$ and multiplication given by $a \cdot_{\opp} a' \coloneqq (-1)^{|a||a'|} a'a$ for any $\BZ_{2}$-homogeneous
elements $a, a' \in A$. A superalgebra anti-homomorphism between superalgebras $A$ and $B$ is a superalgebra homomorphism
$A \to B^{\opp}$, or equivalently $A^{\opp} \to B$.

\medskip
Henceforth, we fix a superspace $V$ over $\BC$ of dimension $(m,n)$ with $n$ even, and set $N \coloneqq m+n$ and
$s \coloneqq \lfloor \frac{N}{2} \rfloor$. We assume $N \geq 3$ and fix an ordered $\BZ_{2}$-homogeneous basis
$\{v_{1}, \ldots, v_{N}\}$ of $V$. For each index $i \in \mathbb{I} \coloneqq \{1, \ldots, N\}$, we define its
\emph{parity} by $\ol{i} \coloneqq |v_{i}| \in \BZ_{2}$. Consider the involution on the index set $\mathbb{I}$ defined by
\begin{equation*}
  i \mapsto i' \coloneqq N+1-i \,.
\end{equation*}
We assume the symmetry $\ol{i} = \ol{i'}$ on the index set $\mathbb{I}$ (in particular, $\ol{s+1} = |v_{s+1}| = \bbar{0}$
when $N$ is odd), so that the grading of the ordered basis is recorded by the \emph{parity sequence}
$$
  \gamma_{V} \coloneqq (\ol{1}, \ldots, \ol{s}) \in \{\bbar{0}, \bbar{1}\}^{s} \,.
$$
This choice of basis identifies $\End_{\BC}(V)$ with the superalgebra of $N \times N$ matrices spanned by the elementary
matrices $\{E_{ij}\}_{i,j=1}^{N}$ of $\BZ_{2}$-degree $|E_{ij}| = \ol{i}+\ol{j}$. We also fix a sign sequence
$\vartheta_{V} \coloneqq (\vartheta_{1}, \ldots, \vartheta_{N}) \in \{\pm 1\}^{N}$ satisfying $\vartheta_{i} = 1$
if $\ol{i} = \bbar{0}$ and $\vartheta_{i} = -\vartheta_{i'}$ if $\ol{i} = \bbar{1}$, so that
\begin{equation*}
  \vartheta_{i}^{2} = 1 \qquad \text{and} \qquad
  \vartheta_{i}\vartheta_{i'} = (-1)^{\ol{i}} \qquad \text{for all} \quad i \in \mathbb{I} \,.
\end{equation*}

The \emph{orthosymplectic Lie superalgebra} $\fosp(V)$ is the Lie subalgebra of $\gl(V) = \End_{\BC}(V)$ preserving the
bilinear form with matrix $G = \sum_{i=1}^{N} \vartheta_{i} E_{i'i}$ (see~\cite[\S4.1]{ht2} for details), and is spanned~by
\begin{equation}\label{eq:X-ij-def}
  X_{ij} \coloneqq E_{ij} - (-1)^{\ol{i}(\ol{i}+\ol{j})} \vartheta_{i}\vartheta_{j} E_{j'i'}
  \qquad (1 \leq i,j \leq N) \,.
\end{equation}
Let $\fh$ be the Cartan subalgebra of $\fosp(V)$ consisting of diagonal matrices, with basis $\{X_{ii}\}_{i=1}^{s}$, and let
$\{\varepsilon_{i}\}_{i=1}^{s}$ be the dual basis of $\fh^{*}$. Setting $\varepsilon_{i'} \coloneqq -\varepsilon_{i}$ for
$1 \leq i \leq s$ (and $\varepsilon_{s+1} \coloneqq 0$ for odd $N$), the non-degenerate supertrace form on $\fosp(V)$
induces a symmetric bilinear form $(\cdot,\cdot)$ on $\fh^{*}$ given by
\begin{equation}\label{eq:epsilon-form}
  (\varepsilon_{i}, \varepsilon_{j}) = (\delta_{ij} - \delta_{ij'}) (-1)^{\ol{i}}
  \qquad \text{for all} \quad 1 \leq i,j \leq N \,.
\end{equation}

With respect to the choice of upper-triangular Borel subalgebra of $\fosp(V)$, the positive root system
$\Phi^{+} = \Phi_{\bbar{0}}^{+} \cup \Phi_{\bbar{1}}^{+}$ is given by
\begin{equation*}
  \Phi^{+} = \big\{\varepsilon_{a} - \varepsilon_{b} \,\big|\, 1 \leq a < b < a' \leq N \big\}
  \cup \big\{2\varepsilon_{a} \,\big|\, 1 \leq a \leq s\,,\ \ol{a} = \bbar{1} \big\} \,,
\end{equation*}
where $\Phi_{\bbar{0}}^{+}$ and $\Phi_{\bbar{1}}^{+}$ denote the subsets of even and odd positive roots, respectively.
We also define the \emph{Weyl vector} $\rho \in \fh^{*}$ via
\begin{equation}\label{eq:weyl-vector}
  \rho \coloneqq
  \frac{1}{2} {\sum}_{\alpha \in \Phi_{\bbar{0}}^{+}} \alpha
  - \frac{1}{2} {\sum}_{\alpha \in \Phi_{\bbar{1}}^{+}} \alpha \,.
\end{equation}


\subsection{Normalized affine $R$-matrix $\ol{R}(z)$}\label{ssec:normalized-R-matrix}
\

Let $V_{q} \coloneqq \BC(q) \otimes_{\BC} V$. The \emph{(normalized) affine orthosymplectic $R$-matrix}
$\ol{R}(z) \in \BC(q)(z) \otimes_{\BC(q)} \End(V_{q})^{\otimes 2}$ is given by the following explicit formula
(cf.~\cite[Theorem 6.3]{ht1}):
{\allowdisplaybreaks
\begin{equation}\label{eq:R-affine-BCD}
\begin{split}
  \frac{zq - q^{-1}}{z-1} \ol{R}(z)
  = &\, \sum_{i,j = 1}^{N} q^{(\varepsilon_{i}, \varepsilon_{j})} E_{ii} \otimes E_{jj}
     + (q-q^{-1}) \sum_{i>j} (-1)^{\ol{j}} E_{ij} \otimes E_{ji} \\
  & - (q-q^{-1}) \sum_{i>j} (-1)^{\ol{i}\,\ol{j}} \vartheta_{i}\vartheta_{j} q^{(\rho,\varepsilon_{i}-\varepsilon_{j})}
    q^{-(\varepsilon_{i},\varepsilon_{i})/2}q^{(\varepsilon_{j},\varepsilon_{j})/2} E_{ij} \otimes E_{i'j'} \\
  & + \frac{q-q^{-1}}{z-1} \sum_{i,j = 1}^{N} (-1)^{\ol{j}} E_{ij} \otimes E_{ji} \\
  & - \frac{(q-q^{-1})\xi}{z-\xi} \sum_{i,j = 1}^{N} (-1)^{\ol{i}\,\ol{j}} \vartheta_{i}\vartheta_{j}
    q^{(\rho,\varepsilon_{i}-\varepsilon_{j})}q^{-(\varepsilon_{i},\varepsilon_{i})/2}q^{(\varepsilon_{j},\varepsilon_{j})/2}
    E_{ij} \otimes E_{i'j'}
\end{split}
\end{equation}}
where we set $\xi \coloneqq q^{-m+n+2}$. Here, the normalization is chosen so that $\ol{R}(z)$ is \emph{unitary}, i.e.
\begin{equation}\label{eq:R-affine-unitary}
  \ol{R}(z^{-1})\ol{R}_{21}(z) = \ID = \ol{R}_{21}(z)\ol{R}(z^{-1}) \,,
\end{equation}
where $\ol{R}_{21}(z) \coloneqq \tau_{12}(\ol{R}(z))$ (see~\cite[\S3.2]{ht2} for details on the leg-numbering notation).

\begin{Rem}\label{rem:R-comparison}
The right-hand side of~\eqref{eq:R-affine-BCD} is obtained from $R(z)$ of~\cite[Theorem~6.3]{ht1} via
$$
  R(z) \mapsto (z-\xi)^{-1} \wt{f}^{-1/2} R(z) \wt{f}^{1/2}
$$
with $\wt{f}^{1/2} \coloneqq \sum_{i,j=1}^{N} q^{-(\varepsilon_{i},\varepsilon_{j})/2} E_{ii} \otimes E_{jj}$,
where the conjugation by $\wt{f}^{1/2}$ accounts for passing from the coproduct~\cite[(2.23)]{ht1} to~\eqref{eq:DJ-Hopf},
similarly to~\cite[Remark~4.7]{ht1}.
\end{Rem}

Consider the operator
\begin{equation}\label{eq:R-check}
  \check{R}(z) \coloneqq \tau_{V_q V_q} \circ \ol{R}(z) \,.
\end{equation}
Then, we note that the unitarity condition~\eqref{eq:R-affine-unitary} is equivalent to
\begin{equation}\label{eq:R-check-unitary}
  \check{R}(z^{-1})\check{R}(z) = \ID = \check{R}(z)\check{R}(z^{-1}) \,.
\end{equation}
Notably, $\ol{R}(z)$ satisfies the parameter-dependent \emph{quantum Yang--Baxter equation}:
\begin{equation*}
  \ol{R}_{12}(u/v)\ol{R}_{13}(u/w)\ol{R}_{23}(v/w) = \ol{R}_{23}(v/w)\ol{R}_{13}(u/w)\ol{R}_{12}(u/v) \,,
\end{equation*}
or equivalently, $\check{R}(z)$ satisfies the parameter-dependent \emph{braid relation}:
\begin{equation}\label{eq:braid-rel}
  \check{R}_{12}(u/v)\check{R}_{23}(u/w)\check{R}_{12}(v/w) = \check{R}_{23}(v/w)\check{R}_{12}(u/w)\check{R}_{23}(u/v) \,.
\end{equation}

\begin{Rem}\label{rem:R-unitary}
The verification of~\eqref{eq:R-affine-unitary}, or equivalently~\eqref{eq:R-check-unitary}, is straightforward by direct
computation. More conceptually, in view of~\eqref{eq:R-univ-eval}, the operator $\check{R}(u/v)$ is a $\UdqV$-intertwiner
$V_{q}[u,u^{-1}] \otimes V_{q}[v,v^{-1}] \to V_{q}[v,v^{-1}] \otimes V_{q}[u,u^{-1}]$ between the tensor product of evaluation
modules from Appendix~\ref{ssec:eval-modules}. Since the tensor product of evaluation modules is generically irreducible, the
composition $\check{R}(z^{-1})\check{R}(z)$ (with $z = u/v$) is a scalar multiple of the identity, so it suffices to verify
that $\check{R}(z^{-1})\check{R}(z)$ fixes the vector $v_{1} \otimes v_{1}$, which is a direct check.
\end{Rem}

Finally, for the computations in Section~\ref{sec:rank-reduction}, we define the scalar function
\begin{equation*}
  \phi(z) \coloneqq \frac{zq - q^{-1}}{z-1}
\end{equation*}
and the rescaled operator
\begin{equation}\label{eq:calR-hat-def}
  \hat{\CR}(z) \coloneqq \phi(z) \check{R}(z) = \frac{zq - q^{-1}}{z-1} \check{R}(z) \,.
\end{equation}
Explicitly, $\hat{\CR}(z)$ is given by:
\begin{equation}\label{eq:calR-hat-explicit}
\begin{split}
  \hat{\CR}(z)
  &= \sum_{i \neq j,j'} (-1)^{\ol{j}} E_{ij} \otimes E_{ji}
   + \sum_{i \neq i'} (-1)^{\ol{i}} \cdot
     \frac{zq^{(\varepsilon_{i},\varepsilon_{i})} - q^{-(\varepsilon_{i},\varepsilon_{i})}}{z-1} E_{ii} \otimes E_{ii} \\
  &\hspace{15pt} + \sum_{i \neq i'} (-1)^{\ol{i}} \cdot
   \frac{zq^{-(\varepsilon_{i},\varepsilon_{i})} - \xi q^{(\varepsilon_{i},\varepsilon_{i})}}{z-\xi} E_{ii'} \otimes E_{i'i} \\
  &\hspace{15pt} + \sum_{i = i'} (-1)^{\ol{i}} \left(1 + (q-q^{-1})\frac{(1-\xi)z}{(z-1)(z-\xi)}\right) E_{ii} \otimes E_{ii} \\
  &\hspace{15pt} + (q-q^{-1})\frac{z}{z-1} \sum_{i<j} E_{ii} \otimes E_{jj}
   + (q-q^{-1})\frac{1}{z-1} \sum_{i>j} E_{ii} \otimes E_{jj} \\
  &\hspace{15pt} - (q-q^{-1})\frac{z}{z-\xi} \sum_{i>j} (-1)^{\ol{i}(\ol{i}+\ol{j})} \vartheta_{i}\vartheta_{j}
    q^{(\rho,\varepsilon_{i}-\varepsilon_{j})}q^{-(\varepsilon_{i},\varepsilon_{i})/2}q^{(\varepsilon_{j},\varepsilon_{j})/2}
    E_{i'j} \otimes E_{ij'} \\
  &\hspace{15pt} - (q-q^{-1})\frac{\xi}{z-\xi} \sum_{i<j} (-1)^{\ol{i}(\ol{i}+\ol{j})} \vartheta_{i}\vartheta_{j}
   q^{(\rho,\varepsilon_{i}-\varepsilon_{j})}q^{-(\varepsilon_{i},\varepsilon_{i})/2}q^{(\varepsilon_{j},\varepsilon_{j})/2}
   E_{i'j} \otimes E_{ij'} \,.
\end{split}
\end{equation}
Note that $\hat{\CR}(z)$ has simple poles only at $z = 1$ and $z = \xi$.


\medskip

\section{The superalgebras $U(\ol{R}(z))$ and $\wt{U}(\ol{R}(z))$}\label{sec:RLL-algebra}



\subsection{The superalgebra $U(\ol{R}(z))$}\label{ssec:def-U-R}
\

We employ the standard leg-numbering notation (for details, see~\cite[\S3.2]{ht2}).
Consider the free $\BC(q)$-superalgebra
\begin{equation*}
  \CT \coloneqq \BC(q) \big\langle l^{\pm}_{ij}[\mp k] \,,\, q^{\pm c} \,\big|\, 1 \leq i,j \leq N\,, \; k \geq 0 \big\rangle
\end{equation*}
equipped with the $\BZ_{2}$-grading
\begin{equation*}
  \big|l^{\pm}_{ij}[\mp k]\big| = \ol{i} + \ol{j} = |E_{ij}| \,, \qquad \big|q^{\pm c}\big| = \bbar{0} \,.
\end{equation*}
We organize these generators into generating series
\begin{equation*}
  l^{\pm}_{ij}(u) \coloneqq \sum_{k \geq 0} l^{\pm}_{ij}[\mp k] u^{\pm k} \in \CT[\![u^{\pm 1}]\!]
\end{equation*}
and matrices
\begin{equation*}
  L^{\pm}(u) \coloneqq \sum_{i,j=1}^{N} l^{\pm}_{ij}(u) \otimes E_{ij}
  \in \CT[\![u^{\pm 1}]\!] \otimes_{\BC(q)} \End(V_{q}) \,.
\end{equation*}

\begin{Def}\label{def:U(R(z))}
The $\BC(q)$-superalgebra $U(\ol{R}(z))$ is the quotient of $\CT$ by the following relations:
\begin{align}
  & q^{\pm c} q^{\mp c} = 1\,, \qquad q^{\pm c} - \textrm{central}\,, \label{eq:U(R(Z))-rel-central}\\
  & l^{+}_{ij}[0] = l^{-}_{ji}[0] = 0 \quad (1 \leq j < i \leq N)\,, \qquad
    l^{\pm}_{ii}[0] l^{\mp}_{ii}[0] = 1 \quad (1 \leq i \leq N)\,, \label{eq:U(R(Z))-rel-0}
\end{align}
together with all matrix coefficients in $\CT$ of the following matrix relations:
\begin{align}
  \ol{R}_{12}(u/v)L^{\pm}_{1}(u)L^{\pm}_{2}(v) &= L^{\pm}_{2}(v)L^{\pm}_{1}(u)\ol{R}_{12}(u/v)\,,
    \label{eq:U(R(Z))-rel-1} \\
  \ol{R}_{12}(q^{c}u/v)L^{+}_{1}(u)L^{-}_{2}(v) &= L^{-}_{2}(v)L^{+}_{1}(u)\ol{R}_{12}(q^{-c}u/v)\,.
    \label{eq:U(R(Z))-rel-2}
\end{align}
\end{Def}

We note that $\ol{R}_{12}(u/v)$ in~\eqref{eq:U(R(Z))-rel-1} can be expanded as a formal power series in either $u/v$ or $v/u$
(yielding equivalent sets of relations), whereas $\ol{R}_{12}(q^{\pm c}u/v)$ in~\eqref{eq:U(R(Z))-rel-2} are expanded as
formal power series in $u/v$.

\begin{Rem}\label{rem:L-invertible}
Relations~\eqref{eq:U(R(Z))-rel-0} imply that the matrices $L^{\pm}[0] \coloneqq \sum_{i,j=1}^{N} l^{\pm}_{ij}[0] \otimes E_{ij}$
are upper and lower triangular, respectively, with invertible diagonal entries. Consequently, $L^{\pm}(u)$ are invertible in
$U(\ol{R}(z))[\![u^{\pm 1}]\!] \otimes \End(V_{q})$.
\end{Rem}


\subsection{Sign conventions and the twisted superalgebra $\wt{U}(\ol{R}(z))$}\label{ssec:twisted-U-R}
\

For a $\BC(q)$-superalgebra $\CA$ (such as $\CT$ or its quotients) and $\nu,\eta \in \{\pm\}$, in the standard tensor product
superalgebra $\CA[\![u^{\nu}, v^{\eta}]\!] \otimes \End(V_{q})^{\otimes 2}$ (where $z^{\pm} \coloneqq z^{\pm 1}$), multiplying
$L^{\nu}_{1}(u)$ and $L^{\eta}_{2}(v)$ introduces Koszul signs from permuting the matrix units past the generating series:
\begin{equation}\label{eq:standard-L1L2}
  L^{\nu}_{1}(u)L^{\eta}_{2}(v)
  = \sum_{i,j,k,l=1}^{N} (-1)^{(\ol{i}+\ol{j})(\ol{k}+\ol{l})}
    l^{\nu}_{ij}(u)l^{\eta}_{kl}(v) \otimes E_{ij} \otimes E_{kl} \,.
\end{equation}
Consider $G \coloneqq \BZ_{2} \times \BZ_{2}$. As established in~\cite[\S3.5]{ht2}, these Koszul signs between the coefficient
superalgebra and the auxiliary matrix factors can be decoupled by viewing $\CA$ and $\End(V_{q})^{\otimes k}$ as $G$-graded spaces
concentrated in degrees $\BZ_{2} \times \{\bbar{0}\}$ and $\{\bbar{0}\} \times \BZ_{2}$, respectively, and using the twisted
braiding induced by the bicharacter
\begin{equation*}
  \beta' \colon G \times G \to \BC^{\times} \,, \qquad
  \beta'\big((g_{1},g_{2}), (h_{1},h_{2})\big) \coloneqq (-1)^{g_{1}h_{1} + g_{2}h_{2}} \,.
\end{equation*}
In the resulting \emph{twisted} tensor product $\CA \otimes^{\beta'} \End(V_{q})^{\otimes k}$, the multiplication is given by
$$
  (a \otimes X)(a' \otimes X') = (aa') \otimes (XX') \qquad \forall\, a,a' \in \CA \,,\, X,X' \in \End(V_{q})^{\otimes k}
$$
where $aa'$ and $XX'$ denote the products in the superalgebras $\CA$ and $\End(V_{q})^{\otimes k}$, respectively. Thus,
in contrast to~\eqref{eq:standard-L1L2}, in $\CA[\![u^{\nu}, v^{\eta}]\!] \otimes^{\beta'} \End(V_{q})^{\otimes 2}$ we have
\begin{equation}\label{eq:twisted-L1L2}
  L^{\nu}_{1}(u)L^{\eta}_{2}(v)
  = \sum_{i,j,k,l=1}^{N} l^{\nu}_{ij}(u)l^{\eta}_{kl}(v) \otimes E_{ij} \otimes E_{kl} \,.
\end{equation}
Moreover, the product in $\CA[\![z^{\nu}]\!] \otimes^{\beta'} \End(V_{q})$ obeys the usual sign-free matrix multiplication rule.

\begin{Def}\label{def:U-tilde}
The \emph{twisted superalgebra} $\wt{U}(\ol{R}(z))$ is the quotient of $\CT$ by the
relations~\eqref{eq:U(R(Z))-rel-central}--\eqref{eq:U(R(Z))-rel-2}, where the matrix products in~\eqref{eq:U(R(Z))-rel-1}
and~\eqref{eq:U(R(Z))-rel-2} are evaluated via~\eqref{eq:twisted-L1L2} in the twisted tensor product $\otimes^{\beta'}$
instead of the standard tensor product $\otimes$.
\end{Def}

Recall from~\cite[\S3.5]{ht2} that for any superalgebra $A$ and the normalized $2$-cocycle
$\zeta \colon \BZ_{2} \times \BZ_{2} \to \{\pm 1\}$ given by $\zeta(\bbar{i},\bbar{j}) \coloneqq (-1)^{\bbar{i}\,\bbar{j}}$,
the \emph{$2$-cocycle twist} $A^{\zeta}$ is the superalgebra with the same underlying superspace as $A$ and twisted multiplication
$$
  a \cdot_{\zeta} a' \coloneqq (-1)^{|a||a'|} aa' \qquad \mathrm{for\ homogeneous} \quad a,a' \in A \,.
$$
The exact same proof as in~\cite[Proposition~3.16]{ht2} yields:

\begin{Prop}\label{prop:affine-twisted-rtt}
(a) The assignment $l^{\pm}_{ij}[\mp k] \mapsto l^{\pm}_{ij}[\mp k]$ and $q^{\pm c} \mapsto q^{\pm c}$ gives rise to
a superalgebra isomorphism
\begin{equation*}
  U(\ol{R}(z)) \iso \wt{U}(\ol{R}(z))^{\zeta} \,.
\end{equation*}

\smallskip
\noindent
(b) The assignment $l^{\pm}_{ij}[\mp k] \mapsto l^{\pm}_{ij}[\mp k]$ and $q^{\pm c} \mapsto q^{\pm c}$ gives rise to
a superalgebra isomorphism
\begin{equation*}
  \wt{U}(\ol{R}(z)) \iso U(\ol{R}(z))^{\zeta} \,.
\end{equation*}
\end{Prop}

\begin{Rem}\label{rem:twist-functoriality}
By functoriality of the $2$-cocycle twist, any superalgebra homomorphism $\psi \colon A \to B$ defines a superalgebra
homomorphism $\psi^{\zeta} \colon A^{\zeta} \to B^{\zeta}$ given by the same underlying set-theoretic map
(and vice versa, since $A$ is obtained by untwisting via $A = (A^{\zeta})^{\zeta}$).
In particular, part~(a) of Proposition~\ref{prop:affine-twisted-rtt} follows immediately from part~(b).
\end{Rem}

Although our main results are stated for both $U(\ol{R}(z))$ and $\wt{U}(\ol{R}(z))$, the twisted superalgebra $\wt{U}(\ol{R}(z))$
is better suited for explicit computations; henceforth, we shall work within $\wt{U}(\ol{R}(z))$ unless stated otherwise.
Let $\{v^{*}_{1}, \ldots, v^{*}_{N}\}$ be the dual basis of $V_{q}^{*}$, and denote
\begin{equation*}
  |i_{1}, \ldots, i_{r}\rangle \coloneqq v_{i_{1}} \otimes \cdots \otimes v_{i_{r}} \in V_{q}^{\otimes r} \,, \qquad
  \langle j_{1}, \ldots, j_{r}| \coloneqq v^{*}_{j_{1}} \otimes \cdots \otimes v^{*}_{j_{r}} \in (V_{q}^{*})^{\otimes r} \,.
\end{equation*}
The canonical pairing between $(V_{q}^{*})^{\otimes r}$ and $V_{q}^{\otimes r}$ in $\sVect_{\BC(q)}$ is given by
\begin{equation}\label{eq:canonical-r-pairing}
  \langle j_{1}, \ldots, j_{r} | i_{1}, \ldots, i_{r} \rangle
  = (-1)^{\sum_{1 \leq a < b \leq r} \ol{i_{a}}\,\ol{i_{b}}} \delta_{j_{1}i_{1}} \cdots \delta_{j_{r}i_{r}} \,.
\end{equation}
In particular, for $r=2$, we have
\begin{equation*}
  \langle k,l | E_{ab} \otimes E_{cd} | i,j \rangle
  = \delta_{ka}\delta_{bi}\delta_{lc}\delta_{dj} (-1)^{\ol{k}\,\ol{l}} (-1)^{\ol{i}\,\ol{j}} (-1)^{\ol{i}\,\ol{l}} \,,
\end{equation*}
which extends to $\CA[\![u^{\nu}, v^{\eta}]\!] \otimes^{\beta'} \End(V_{q})^{\otimes 2}$ via~\eqref{eq:twisted-L1L2} to yield:
\begin{equation}\label{eq:twisted-matrix-element}
  \langle k,l | L^{\nu}_{1}(u) L^{\eta}_{2}(v) | i,j \rangle
  = (-1)^{\ol{k}\,\ol{l}} (-1)^{\ol{i}\,\ol{j}} (-1)^{\ol{i}\,\ol{l}} l^{\nu}_{ki}(u) l^{\eta}_{lj}(v) \,.
\end{equation}
Analogously, for $r=4$ and any $X, Y \in \CA \otimes^{\beta'} \End(V_{q})^{\otimes 2}$, we have
\begin{equation}\label{eq:twisted-4-matrix-element}
\begin{split}
  \langle k_{1},k_{2},l_{1},l_{2} &| X_{12} Y_{34} | i_{1},i_{2},j_{1},j_{2} \rangle \\
  & = (-1)^{(\ol{k_{1}}+\ol{k_{2}})(\ol{l_{1}}+\ol{l_{2}}) + (\ol{i_{1}}+\ol{i_{2}})(\ol{j_{1}}+\ol{j_{2}}) +
            (\ol{i_{1}}+\ol{i_{2}})(\ol{l_{1}}+\ol{l_{2}})}
  \langle k_{1},k_{2} | X | i_{1},i_{2} \rangle \langle l_{1},l_{2} | Y | j_{1},j_{2} \rangle \,.
\end{split}
\end{equation}


\medskip

\section{Rank reduction}\label{sec:rank-reduction}

In this section, we state our main result (Theorem~\ref{thm:rank-reduction}): the rank reduction homomorphisms.


\subsection{Quasideterminants}\label{ssec:quasideterminant}
\

For $1 < i,j < 1'$, we define the \emph{quasideterminant} $s^{\pm}_{ij}(u) = \sum_{k \geq 0} s^{\pm}_{ij}[\mp k] u^{\pm k}$ via
\begin{equation}\label{eq:quasidet-untwisted}
  s^{\pm}_{ij}(u) \coloneqq l^{\pm}_{ij}(u) - (-1)^{(\ol{1}+\ol{i})(\ol{1}+\ol{j})}
    l^{\pm}_{i1}(u) (l^{\pm}_{11}(u))^{-1} l^{\pm}_{1j}(u) \quad\in\quad U(\ol{R}(z))[\![u^{\pm 1}]\!] \,,
\end{equation}
and the corresponding \emph{twisted quasideterminant} $\wt{s}^{\pm}_{ij}(u) = \sum_{k \geq 0} \wt{s}^{\pm}_{ij}[\mp k] u^{\pm k}$
via
\begin{equation}\label{eq:quasidet-1}
  \wt{s}^{\pm}_{ij}(u) \coloneqq l^{\pm}_{ij}(u) -
    l^{\pm}_{i1}(u) (l^{\pm}_{11}(u))^{-1} l^{\pm}_{1j}(u) \quad\in\quad \wt{U}(\ol{R}(z))[\![u^{\pm 1}]\!] \,.
\end{equation}
We organize these series into the $(N-2)\times (N-2)$ matrices
\begin{equation*}
  S^{\pm}(u) \coloneqq \sum_{1<i,j<1'} s^{\pm}_{ij}(u) \otimes E_{ij} \,, \qquad
  \wt{S}^{\pm}(u) \coloneqq \sum_{1<i,j<1'} \wt{s}^{\pm}_{ij}(u) \otimes E_{ij}\,.
\end{equation*}

\begin{Rem}\label{rem:quasidet-twist}
Since $\big|l^{\pm}_{11}[\mp k]\big| = \bbar{0}$ and $\big|l^{\pm}_{ab}[\mp k]\big| = \ol{a} + \ol{b}$, the isomorphism
$\wt{U}(\ol{R}(z)) \iso U(\ol{R}(z))^{\zeta}$ of Proposition~\ref{prop:affine-twisted-rtt}(b) maps $\wt{s}^{\pm}_{ij}(u)$
to $s^{\pm}_{ij}(u)$. Moreover, since $l^{+}_{i1}[0] = l^{-}_{1j}[0] = 0$ for $1 < i,j < 1'$, we have
$\wt{s}^{\pm}_{ij}[0] = s^{\pm}_{ij}[0] = l^{\pm}_{ij}[0]$, and hence by~\eqref{eq:U(R(Z))-rel-0}:
\begin{equation}\label{eq:s-zero-mode}
  \wt{s}^{+}_{ij}[0] = \wt{s}^{-}_{ji}[0] = 0 \quad (1 < j < i < 1') \qquad \text{and} \qquad
  \wt{s}^{\pm}_{ii}[0] \wt{s}^{\mp}_{ii}[0] = 1 \quad (1 < i < 1') \,,
\end{equation}
with the same relations holding for $s^{\pm}_{ij}[0]$ in $U(\ol{R}(z))$.
\end{Rem}

Working in $\wt{U}(\ol{R}(z))$, comparing the $E_{i1} \otimes E_{11}$ components of both sides of~\eqref{eq:U(R(Z))-rel-1}
(after multiplying by $q\frac{u}{v} - q^{-1}$ to clear the common denominator), we obtain:
\begin{equation*}
  \left(\frac{u}{v} - 1\right) l_{i1}^{\pm}(u)l_{11}^{\pm}(v)
    + (-1)^{\ol{1}}(q-q^{-1})\frac{u}{v} l_{11}^{\pm}(u)l_{i1}^{\pm}(v)
  = \left(q^{(\varepsilon_{1},\varepsilon_{1})}\frac{u}{v} - q^{-(\varepsilon_{1},\varepsilon_{1})}\right)
    l_{11}^{\pm}(v)l_{i1}^{\pm}(u) \,.
\end{equation*}
The right-hand side vanishes under the specialization $u = \kappa_{0} v$, where
\begin{equation}\label{eq:kappa-o}
  \kappa_{0} \coloneqq q^{-2(\varepsilon_{1},\varepsilon_{1})} \,,
\end{equation}
so that
\begin{equation*}
  (\kappa_{0} - 1) l_{i1}^{\pm}(\kappa_{0} v)l_{11}^{\pm}(v)
  + (-1)^{\ol{1}}(q-q^{-1})\kappa_{0} l_{11}^{\pm}(\kappa_{0} v)l_{i1}^{\pm}(v) = 0 \,.
\end{equation*}
The above equality simplifies to
\begin{equation}\label{eq:l11-li1-comm}
  l_{11}^{\pm}(\kappa_{0} v)l_{i1}^{\pm}(v)
  = q^{(\varepsilon_{1},\varepsilon_{1})} l_{i1}^{\pm}(\kappa_{0} v)l_{11}^{\pm}(v) \,.
\end{equation}
Replacing $v$ by $u$ in~\eqref{eq:l11-li1-comm} and substituting the resulting relation back into~\eqref{eq:quasidet-1} yields
\begin{equation}\label{eq:quasidet-2}
  \wt{s}^{\pm}_{ij}(u) = l^{\pm}_{ij}(u) -
  q^{(\varepsilon_{1},\varepsilon_{1})} (l^{\pm}_{11}(\kappa_{0} u))^{-1} l^{\pm}_{i1}(\kappa_{0} u) l^{\pm}_{1j}(u) \,.
\end{equation}

To establish the relations among these $\wt{s}^{\pm}_{ij}(u)$, we first rewrite the defining
relations~\eqref{eq:U(R(Z))-rel-1}--\eqref{eq:U(R(Z))-rel-2} in terms of $\check{R}(z)$ from~\eqref{eq:R-check}:
\begin{align}
  \check{R}_{12}(u/v) L^{\pm}_{1}(u) L^{\pm}_{2}(v)
  &= L^{\pm}_{1}(v) L^{\pm}_{2}(u) \check{R}_{12}(u/v) \,, \label{eq:braided-RLL-1}
  \\
  \check{R}_{12}(q^{c}u/v) L^{+}_{1}(u) L^{-}_{2}(v)
  &= L^{-}_{1}(v) L^{+}_{2}(u) \check{R}_{12}(q^{-c}u/v) \,, \label{eq:braided-RLL-2}
\end{align}
or equivalently, in terms of the rescaled operator $\hat{\CR}(z) = \phi(z)\check{R}(z)$ from~\eqref{eq:calR-hat-def}:
\begin{align}
  \hat{\CR}_{12}(u/v) L^{\pm}_{1}(u) L^{\pm}_{2}(v)
  &= L^{\pm}_{1}(v) L^{\pm}_{2}(u) \hat{\CR}_{12}(u/v) \,, \label{eq:rescaled-braided-RLL-1} \\
  \phi(q^{c}u/v)^{-1} \hat{\CR}_{12}(q^{c}u/v) L^{+}_{1}(u) L^{-}_{2}(v)
  &= \phi(q^{-c}u/v)^{-1} L^{-}_{1}(v) L^{+}_{2}(u) \hat{\CR}_{12}(q^{-c}u/v) \,. \label{eq:rescaled-braided-RLL-2}
\end{align}
Using the explicit formula~\eqref{eq:calR-hat-explicit} for $\hat{\CR}(z)$ and~\eqref{eq:twisted-matrix-element},
we compute $\langle 1,i|\hat{\CR}_{12}(u/v)L^{\pm}_{1}(u)L^{\pm}_{2}(v)|1,j \rangle$ (which is equal to
$\langle 1,i|L^{\pm}_{1}(v)L^{\pm}_{2}(u)\hat{\CR}_{12}(u/v)|1,j \rangle$ by~\eqref{eq:rescaled-braided-RLL-1}):
\begin{align*}
  \langle 1,i|\hat{\CR}_{12}(u/v)L^{\pm}_{1}(u)L^{\pm}_{2}(v)|1,j \rangle
  &= (-1)^{\ol{1}\,\ol{i}} \langle i,1|L^{\pm}_{1}(u)L^{\pm}_{2}(v)|1,j \rangle
   + \frac{(q-q^{-1})\frac{u}{v}}{\frac{u}{v}-1} \langle 1,i|L^{\pm}_{1}(u)L^{\pm}_{2}(v)|1,j \rangle \\
  &= (-1)^{\ol{1}(\ol{1}+\ol{j})} l^{\pm}_{i1}(u)l^{\pm}_{1j}(v)
   + (-1)^{\ol{1}\,\ol{j}} \frac{(q-q^{-1})\frac{u}{v}}{\frac{u}{v}-1} l^{\pm}_{11}(u)l^{\pm}_{ij}(v) \,.
\end{align*}
Combining this with~\eqref{eq:quasidet-2} yields
\begin{equation}\label{eq:quasidet-1i1j}
  \lim_{\kappa \to \kappa_{0}} \langle 1,i|\hat{\CR}_{12}(\kappa)L^{\pm}_{1}(\kappa u)L^{\pm}_{2}(u)|1,j \rangle
  = -(-1)^{\ol{1}(\ol{1}+\ol{j})} q^{-(\varepsilon_{1},\varepsilon_{1})} \cdot l^{\pm}_{11}(\kappa_{0}u) \wt{s}^{\pm}_{ij}(u) \,.
\end{equation}
We note that this limit is well-defined as $\kappa_{0} \neq 1$, so the denominator appearing in the above formula is non-zero
at $u/v = \kappa_{0}$. Analogously, one obtains the following identities:
\begin{align}
  \lim_{\kappa \to \kappa_{0}} \langle 1,i|\hat{\CR}_{12}(\kappa)L^{\pm}_{1}(\kappa u)L^{\pm}_{2}(u)|j,1 \rangle
    &= l^{\pm}_{11}(\kappa_{0}u) \wt{s}^{\pm}_{ij}(u) \,, \label{eq:quasidet-1ij1} \\
  \lim_{\kappa \to \kappa_{0}} \langle i,1|\hat{\CR}_{12}(\kappa)L^{\pm}_{1}(\kappa u)L^{\pm}_{2}(u)|1,j \rangle
    &= (-1)^{\ol{1}(\ol{i}+\ol{j})} \cdot l^{\pm}_{11}(\kappa_{0}u) \wt{s}^{\pm}_{ij}(u) \,, \label{eq:quasidet-i11j} \\
  \lim_{\kappa \to \kappa_{0}} \langle i,1|\hat{\CR}_{12}(\kappa)L^{\pm}_{1}(\kappa u)L^{\pm}_{2}(u)|j,1 \rangle
    &= -(-1)^{\ol{1}(\ol{1}+\ol{i})} q^{(\varepsilon_{1},\varepsilon_{1})} \cdot
    l^{\pm}_{11}(\kappa_{0}u) \wt{s}^{\pm}_{ij}(u) \,. \label{eq:quasidet-i1j1}
\end{align}
These can be easily deduced from~\eqref{eq:quasidet-1i1j} by utilizing the following equalities:
\begin{equation}\label{eq:R-hat-kappa0-prop}
\begin{split}
  \lim_{\kappa \to \kappa_{0}} \langle i,1 | \hat{\CR}(\kappa)
  &= (-1)^{\ol{1}\,\ol{i}} \langle 1,i | - (-1)^{\ol{1}} q^{(\varepsilon_{1},\varepsilon_{1})} \langle i,1 |
   = -(-1)^{\ol{1}(\ol{1}+\ol{i})} q^{(\varepsilon_{1},\varepsilon_{1})}
     \lim_{\kappa \to \kappa_{0}} \langle 1,i | \hat{\CR}(\kappa) \,, \\
  \lim_{\kappa \to \kappa_{0}} \hat{\CR}(\kappa) | j,1 \rangle
  &= (-1)^{\ol{1}\,\ol{j}} | 1,j \rangle - (-1)^{\ol{1}} q^{(\varepsilon_{1},\varepsilon_{1})} | j,1 \rangle
   = -(-1)^{\ol{1}(\ol{1}+\ol{j})} q^{(\varepsilon_{1},\varepsilon_{1})}
     \lim_{\kappa \to \kappa_{0}} \hat{\CR}(\kappa) | 1,j \rangle \,,
\end{split}
\end{equation}
due to~\eqref{eq:calR-hat-explicit} combined with the identities
\begin{equation}\label{eq:kappa0-fractions}
  \frac{(q-q^{-1})\kappa_{0}}{\kappa_{0}-1} = -(-1)^{\ol{1}} q^{-(\varepsilon_{1},\varepsilon_{1})} \,,\qquad
  \frac{q-q^{-1}}{\kappa_{0}-1} = -(-1)^{\ol{1}} q^{(\varepsilon_{1},\varepsilon_{1})} \,.
\end{equation}
For convenience, let
\begin{equation}\label{eq:gamma-ij-def}
  \wt{\gamma}^{\pm}_{ij}(u) \coloneqq l^{\pm}_{11}(\kappa_{0}u) \wt{s}^{\pm}_{ij}(u) \qquad \forall\, 1 < i,j < 1' \,,
\end{equation}
and define the matrices
\begin{equation}\label{eq:Gamma-def}
  \wt{\Gamma}^{\pm}(u) \coloneqq \sum_{1<i,j<1'} \wt{\gamma}^{\pm}_{ij}(u) \otimes E_{ij}\,.
\end{equation}

We also note from~\eqref{eq:calR-hat-explicit} that
\begin{equation}\label{eq:quasidet-11}
\begin{split}
  \langle 1,1|\hat{\CR}(\kappa) &= (-1)^{\ol{1}} \cdot
  \frac{q^{(\varepsilon_{1},\varepsilon_{1})} \kappa - q^{-(\varepsilon_{1},\varepsilon_{1})}}
       {\kappa - 1} \cdot \langle 1,1| \,, \\
  \hat{\CR}(\kappa)|1,1 \rangle &= (-1)^{\ol{1}} \cdot
  \frac{q^{(\varepsilon_{1},\varepsilon_{1})} \kappa - q^{-(\varepsilon_{1},\varepsilon_{1})}}
       {\kappa - 1} \cdot |1,1 \rangle \,,
\end{split}
\end{equation}
and therefore
\begin{equation}\label{eq:quasidet-11-lim}
  \lim_{\kappa \to \kappa_{0}} \langle 1,1|\hat{\CR}(\kappa) = 0 \qquad\textrm{and}\qquad
  \lim_{\kappa \to \kappa_{0}} \hat{\CR}(\kappa)|1,1 \rangle = 0\,.
\end{equation}


\subsection{The rank-reduction homomorphism}\label{ssec:rank-reduction-hom}
\

To avoid confusion, we use the superscript $[s]$ to indicate the rank $s = \lfloor \frac{N}{2} \rfloor$ when necessary.
Specifically, let
\begin{equation*}
  V^{[s-1]} \coloneqq \Span_{\BC}\, \{v_{i} \,|\, 1 < i < 1'\} \subset V = V^{[s]}
\end{equation*}
be the subspace of codimension $2$ obtained by removing the basis vectors $v_{1}$ and $v_{1'}$. We denote the
Weyl vector~\eqref{eq:weyl-vector}, the parameter $\xi$, and the operators $\ol{R}(z), \hat{\CR}(z)$ from~\eqref{eq:R-affine-BCD},
\eqref{eq:calR-hat-explicit} associated with $V^{[s]}$ and $V^{[s-1]}$ by $\rho^{[s]}, \xi^{[s]}, \ol{R}^{[s]}(z), \hat{\CR}^{[s]}(z)$
and $\rho^{[s-1]}, \xi^{[s-1]}, \ol{R}^{[s-1]}(z), \hat{\CR}^{[s-1]}(z)$, respectively, which are related via
\begin{equation}\label{eq:parameter-shifts}
  \xi^{[s-1]} = \kappa_{0}^{-1} \xi^{[s]} \qquad \text{and} \qquad
  (\rho^{[s-1]}, \varepsilon_{i}) = (\rho^{[s]}, \varepsilon_{i}) \quad \text{for all} \quad 1 < i < 1' \,,
\end{equation}
where we used~\eqref{eq:weyl-vector} for the second identity.

\begin{Thm}\label{thm:rank-reduction}
Assume $s \geq 2$.

\smallskip
\noindent
(a) In $U(\ol{R}^{[s]}(z))$, the elements $s^{\pm}_{ij}[0]$ and matrices $S^{\pm}(u)$ satisfy the defining relations of
$U(\ol{R}^{[s-1]}(z))$:
\begin{gather*}
  s^{+}_{ij}[0] = s^{-}_{ji}[0] = 0 \quad (1 < j < i < 1') \,, \qquad s^{\pm}_{ii}[0] s^{\mp}_{ii}[0] = 1 \quad (1 < i < 1') \,,
    \\
  \ol{R}^{[s-1]}_{12}(u/v) S^{\pm}_{1}(u) S^{\pm}_{2}(v) = S^{\pm}_{2}(v) S^{\pm}_{1}(u) \ol{R}^{[s-1]}_{12}(u/v) \,,
    \\
  \ol{R}^{[s-1]}_{12}(q^{c}u/v) S^{+}_{1}(u) S^{-}_{2}(v) = S^{-}_{2}(v) S^{+}_{1}(u) \ol{R}^{[s-1]}_{12}(q^{-c}u/v) \,.
\end{gather*}
Therefore, the assignment
\begin{equation*}
  L^{\pm\,[s-1]}(u) \mapsto S^{\pm}(u) \,, \qquad  q^{\pm c} \mapsto q^{\pm c}
\end{equation*}
gives rise to a superalgebra homomorphism $U(\ol{R}^{[s-1]}(z)) \to U(\ol{R}^{[s]}(z))$.

\smallskip
\noindent
(b) Equivalently, in $\wt{U}(\ol{R}^{[s]}(z))$, the elements $\wt{s}^{\pm}_{ij}[0]$ and matrices $\wt{S}^{\pm}(u)$ satisfy
the defining relations of $\wt{U}(\ol{R}^{[s-1]}(z))$:
\begin{gather}
  \wt{s}^{+}_{ij}[0] = \wt{s}^{-}_{ji}[0] = 0 \quad (1 < j < i < 1') \,, \qquad
    \wt{s}^{\pm}_{ii}[0] \wt{s}^{\mp}_{ii}[0] = 1 \quad (1 < i < 1') \,, \label{eq:S-RLL-0} \\
  \ol{R}^{[s-1]}_{12}(u/v) \wt{S}^{\pm}_{1}(u) \wt{S}^{\pm}_{2}(v) =
    \wt{S}^{\pm}_{2}(v) \wt{S}^{\pm}_{1}(u) \ol{R}^{[s-1]}_{12}(u/v) \,, \label{eq:S-RLL-1} \\
  \ol{R}^{[s-1]}_{12}(q^{c}u/v) \wt{S}^{+}_{1}(u) \wt{S}^{-}_{2}(v) =
    \wt{S}^{-}_{2}(v) \wt{S}^{+}_{1}(u) \ol{R}^{[s-1]}_{12}(q^{-c}u/v) \,, \label{eq:S-RLL-2}
\end{gather}
and the assignment
\begin{equation*}
  L^{\pm\,[s-1]}(u) \mapsto \wt{S}^{\pm}(u) \,, \qquad  q^{\pm c} \mapsto q^{\pm c}
\end{equation*}
gives rise to a superalgebra homomorphism $\wt{U}(\ol{R}^{[s-1]}(z)) \to \wt{U}(\ol{R}^{[s]}(z))$.
\end{Thm}

The proof of this main result is presented in Section~\ref{sec:proof-rank-reduction}, where it crucially relies on the
following $10$-term matrix identity (in the opposite-sign case):
\begin{equation}\label{eq:fused-RLL-check}
\begin{split}
  & \check{R}_{23}(q^{c}\kappa^{-1} z) \check{R}_{34}(q^{c}z) \check{R}_{12}(q^{c}z) \check{R}_{23}(q^{c}\kappa z)
    \check{R}_{12}(\kappa) \check{R}_{34}(\kappa) L^{+}_{1}(\kappa u)L^{+}_{2}(u) L^{-}_{3}(\kappa v)L^{-}_{4}(v) \\
  &\!\!= L^{-}_{1}(v)L^{-}_{2}(\kappa v) L^{+}_{3}(u)L^{+}_{4}(\kappa u) \check{R}_{12}(\kappa) \check{R}_{34}(\kappa)
    \check{R}_{23}(q^{-c}\kappa z) \check{R}_{34}(q^{-c}z) \check{R}_{12}(q^{-c}z) \check{R}_{23}(q^{-c}\kappa^{-1} z)
\end{split}
\end{equation}
where $z = u/v$. Being of major importance, we shall now prove~\eqref{eq:fused-RLL-check} through a graphical calculus that
also exhibits the fusion procedure underlying it.

We represent matrix products by string diagrams read from bottom to top, where moving upwards from one vertical level to the next
corresponds to multiplying from right to left in the algebraic expression. Each vertical strand (directed upwards) represents a copy
of $V_{q}$ and carries a spectral parameter (such as $u, v, w, \dots$). At any vertical level, the strands are numbered $1, 2, \ldots$
from left to right, matching the leg-numbering notation:
\begin{itemize}[leftmargin=0.6cm]

\item
A hollow ($\circ$) or filled ($\bullet$) bead on the $a$-th strand (with spectral parameter $u$) represents $L^{+}_{a}(u)$ or
$L^{-}_{a}(u)$, respectively, while a horizontal line directed from right to left through a row of beads records the multiplication order
$L^{\nu_{1}}_{1}(u_{1}) L^{\nu_{2}}_{2}(u_{2}) \cdots$ in $\wt{U}(\ol{R}(z))$ (so that, just as in the vertical
direction, moving along the arrow corresponds to multiplying from right to left);

\item
A crossing of the $a$-th and $(a+1)$-th strands (with the right strand passing over the left) represents $\check{R}_{a,a+1}$
acting on the $a$-th and $(a+1)$-th tensor factors of $V_{q}^{\otimes r}$, with each spectral parameter carried along its strand
through the crossing (whereas the opposite crossing, with the left strand passing over the right, would rather represent
$\check{R}_{a,a+1}^{-1}$):
\begin{itemize}[leftmargin=0.5cm]

\item[--]
in view of~\eqref{eq:braided-RLL-2}, we only consider diagrams in which strands carrying filled beads ($\bullet$) pass over
strands carrying hollow beads ($\circ$);

\item[--]
if the two strands carry beads of the same color ($\circ, \circ$ or $\bullet, \bullet$) or no beads, the central element $q^{\pm c}$
does not appear, and the crossing (resp.\ opposite crossing) with strand parameters $u, v$ represents $\check{R}_{a,a+1}(u/v)$
(resp.\ $\check{R}_{a,a+1}(v/u)^{-1}$);

\item[--]
if the two strands carry beads of opposite colors ($\circ, \bullet$ or $\bullet, \circ$), the spectral parameter ratio in
$\check{R}_{a,a+1}$ or $\check{R}_{a,a+1}^{-1}$ is multiplied by $q^{c}$ when the crossing lies above the horizontal line,
and by $q^{-c}$ when the crossing lies below the horizontal line.

\end{itemize}

\end{itemize}
In this language,~\eqref{eq:braided-RLL-1}--\eqref{eq:braided-RLL-2} amount to sliding the horizontal line past a crossing, while
the braid relation~\eqref{eq:braid-rel} is the type III Reidemeister move, see Figure~\ref{fig:local-moves} (whereas the same-sign
relation \eqref{eq:braided-RLL-1} corresponds to the left diagram with both beads of the same color and $q^{\pm c}$ replaced by $1$):

\begin{figure}[H]
\centering
\begin{tikzpicture}[baseline={(0,1.25)}, scale=0.85, every node/.style={font=\small}]
  \begin{scope}[shift={(0,0)}]
    \draw[thick] (0,0) node[below] {$u$} -- (0,1.0);
    \draw[thick] (1.3,0) node[below] {$v$} -- (1.3,1.0);
    \draw[thick, ->] (0,1.0) -- (1.3,2.2) -- (1.3,2.7) node[above] {$u$};
    \draw[thick, preaction={-, draw=white, line width=5pt}, ->] (1.3,1.0) -- (0,2.2) -- (0,2.7) node[above] {$v$};
    \draw[thick, ->] (1.9,0.5) -- (-0.6,0.5);
    \filldraw[fill=white, thick] (0,0.5) circle (2.8pt) node[above left=1pt] {$L^{+}(u)$};
    \filldraw[fill=black, thick] (1.3,0.5) circle (2.8pt) node[above right=1pt] {$L^{-}(v)$};
    \node[left=4pt] at (0.3,1.6) {$\check{R}(q^{c}u/v)$};

    \node at (3.2,1.35) {$=$};

    \begin{scope}[shift={(4.4,0)}]
      \draw[thick, ->] (0,0) node[below] {$u$} -- (0,0.4) -- (1.3,1.6) -- (1.3,2.7) node[above] {$u$};
      \draw[thick, preaction={-, draw=white, line width=5pt}, ->] (1.3,0) node[below] {$v$} -- (1.3,0.4) -- (0,1.6) -- (0,2.7) node[above] {$v$};
      \draw[thick, ->] (1.9,2.1) -- (-0.6,2.1);
      \filldraw[fill=black, thick] (0,2.1) circle (2.8pt) node[above left=1pt] {$L^{-}(v)$};
      \filldraw[fill=white, thick] (1.3,2.1) circle (2.8pt) node[above right=1pt] {$L^{+}(u)$};
      \node[right=4pt] at (1.0,1.0) {$\check{R}(q^{-c}u/v)$};
    \end{scope}
  \end{scope}

  \begin{scope}[shift={(11,0)}]
    \draw[thick, ->] (0,0) node[below] {$u$} -- (0.9,0.85) -- (1.8,1.7) -- (1.8,2.7) node[above] {$u$};
    \draw[thick, preaction={draw=white, line width=5pt}, ->] (0.9,0) node[below] {$v$} -- (0,0.85) -- (0,1.7) -- (0.9,2.55) -- (0.9,2.7) node[above] {$v$};
    \draw[thick, preaction={draw=white, line width=5pt}, ->] (1.8,0) node[below] {$w$} -- (1.8,0.85) -- (0.9,1.7) -- (0,2.55) -- (0,2.7) node[above] {$w$};

    \node at (2.5,1.35) {$=$};

    \begin{scope}[shift={(3.2,0)}]
      \draw[thick, ->] (0,0) node[below] {$u$} -- (0,0.85) -- (0.9,1.7) -- (1.8,2.55) -- (1.8,2.7) node[above] {$u$};
      \draw[thick, preaction={draw=white, line width=5pt}, ->] (0.9,0) node[below] {$v$} -- (1.8,0.85) -- (1.8,1.7) -- (0.9,2.55) -- (0.9,2.7) node[above] {$v$};
      \draw[thick, preaction={draw=white, line width=5pt}, ->] (1.8,0) node[below] {$w$} -- (0.9,0.85) -- (0,1.7) -- (0,2.7) node[above] {$w$};
    \end{scope}
  \end{scope}
\end{tikzpicture}
\caption{Graphical representation of~\eqref{eq:braided-RLL-2} (left) and~\eqref{eq:braid-rel} (right).}
\label{fig:local-moves}
\end{figure}
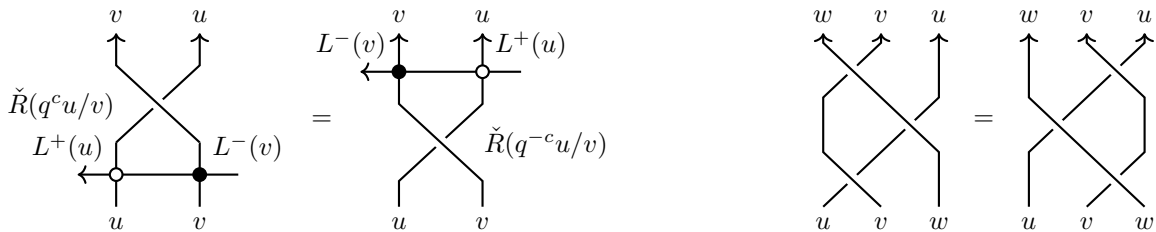

Consider now the $4$-strand half-twist, i.e.\ the braid that reverses the order of the four strands by passing each strand
over all strands to its left. One can verify by repeated applications of~\eqref{eq:braid-rel} that its two decompositions
$\check{R}_{23}\check{R}_{34}\check{R}_{12}\check{R}_{23}\check{R}_{12}\check{R}_{34}$ and
$\check{R}_{12}\check{R}_{34}\check{R}_{23}\check{R}_{34}\check{R}_{12}\check{R}_{23}$ coincide,
see Figure~\ref{fig:half-twist-decomp}:

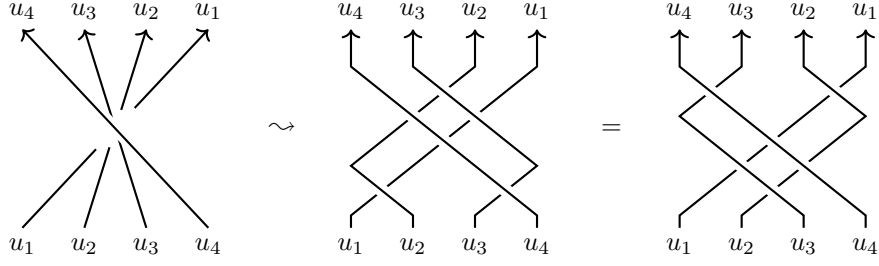
\begin{figure}[H]
\centering
\begin{tikzpicture}[baseline={(0,1.6)}, scale=0.82, every node/.style={font=\small}]
  \begin{scope}[shift={(0,0)}]
    \draw[thick, ->] (0,0) node[below] {$u_{1}$} -- (3,3.2) node[above] {$u_{1}$};
    \fill[white] (1.5,1.6) circle (13pt);
    \draw[thick, ->] (1,0) node[below] {$u_{2}$} -- (2,3.2) node[above] {$u_{2}$};
    \fill[white] (1.5,1.6) circle (9pt);
    \draw[thick, ->] (2,0) node[below] {$u_{3}$} -- (1,3.2) node[above] {$u_{3}$};
    \fill[white] (1.5,1.6) circle (5pt);
    \draw[thick, preaction={draw=white, line width=4.5pt}, ->] (3,0) node[below] {$u_{4}$} -- (0,3.2) node[above] {$u_{4}$};
  \end{scope}

  \node at (4.2,1.6) {$\leadsto$};

  \begin{scope}[shift={(5.3,0)}]
    \draw[thick, ->] (0,0) node[below] {$u_{1}$} -- (0,0.2) -- (1,1.0) -- (3,2.6) -- (3,3.2) node[above] {$u_{1}$};
    \draw[thick, preaction={draw=white, line width=5pt}, ->] (1,0) node[below] {$u_{2}$} -- (1,0.2) -- (0,1.0) -- (2,2.6) -- (2,3.2) node[above] {$u_{2}$};
    \draw[thick, preaction={draw=white, line width=5pt}, ->] (2,0) node[below] {$u_{3}$} -- (2,0.2) -- (3,1.0) -- (1,2.6) -- (1,3.2) node[above] {$u_{3}$};
    \draw[thick, preaction={draw=white, line width=5pt}, ->] (3,0) node[below] {$u_{4}$} -- (3,0.2) -- (2,1.0) -- (0,2.6) -- (0,3.2) node[above] {$u_{4}$};
  \end{scope}

  \node at (9.5,1.6) {$=$};

  \begin{scope}[shift={(10.6,0)}]
    \draw[thick, ->] (0,0) node[below] {$u_{1}$} -- (0,0.2) -- (2,1.8) -- (3,2.6) -- (3,3.2) node[above] {$u_{1}$};
    \draw[thick, preaction={draw=white, line width=5pt}, ->] (1,0) node[below] {$u_{2}$} -- (1,0.2) -- (3,1.8) -- (2,2.6) -- (2,3.2) node[above] {$u_{2}$};
    \draw[thick, preaction={draw=white, line width=5pt}, ->] (2,0) node[below] {$u_{3}$} -- (2,0.2) -- (0,1.8) -- (1,2.6) -- (1,3.2) node[above] {$u_{3}$};
    \draw[thick, preaction={draw=white, line width=5pt}, ->] (3,0) node[below] {$u_{4}$} -- (3,0.2) -- (1,1.8) -- (0,2.6) -- (0,3.2) node[above] {$u_{4}$};
  \end{scope}
\end{tikzpicture}
\caption{The $4$-strand half-twist and its two decompositions.}
\label{fig:half-twist-decomp}
\end{figure}

Sliding $L^{+}_{1}(\kappa u)L^{+}_{2}(u)L^{-}_{3}(\kappa v)L^{-}_{4}(v)$ from the bottom to the top of the $4$-strand half-twist
via~\eqref{eq:braided-RLL-1}--\eqref{eq:braided-RLL-2} and applying the above identity (represented by
Figure~\ref{fig:half-twist-decomp}) yields~\eqref{eq:fused-RLL-check}, see Figure~\ref{fig:fused-RLL-derivation}:

\begin{figure}[H]
\centering
\begin{tikzpicture}[baseline={(0,1.9)}, scale=0.82, every node/.style={font=\small}]
  \begin{scope}[shift={(0,0)}]
    \draw[thick, ->] (0,0) node[below] {$\kappa u$} -- (0,0.8) -- (1,1.6) -- (3,3.2) -- (3,3.8) node[above] {$\kappa u$};
    \draw[thick, preaction={draw=white, line width=5pt}, ->] (1,0) node[below] {$u$} -- (1,0.8) -- (0,1.6) -- (2,3.2) -- (2,3.8) node[above] {$u$};
    \draw[thick, preaction={draw=white, line width=5pt}, ->] (2,0) node[below] {$\kappa v$} -- (2,0.8) -- (3,1.6) -- (1,3.2) -- (1,3.8) node[above] {$\kappa v$};
    \draw[thick, preaction={draw=white, line width=5pt}, ->] (3,0) node[below] {$v$} -- (3,0.8) -- (2,1.6) -- (0,3.2) -- (0,3.8) node[above] {$v$};

    \draw[thick, ->] (3.5,0.4) -- (-0.5,0.4);
    \filldraw[fill=white, thick] (0,0.4) circle (2.8pt);
    \filldraw[fill=white, thick] (1,0.4) circle (2.8pt);
    \filldraw[fill=black, thick] (2,0.4) circle (2.8pt);
    \filldraw[fill=black, thick] (3,0.4) circle (2.8pt);
  \end{scope}

  \node at (4.2,1.9) {$=$};

  \begin{scope}[shift={(5.3,0)}]
    \draw[thick, ->] (0,0) node[below] {$\kappa u$} -- (0,0.3) -- (1,1.1) -- (3,2.7) -- (3,3.8) node[above] {$\kappa u$};
    \draw[thick, preaction={draw=white, line width=5pt}, ->] (1,0) node[below] {$u$} -- (1,0.3) -- (0,1.1) -- (2,2.7) -- (2,3.8) node[above] {$u$};
    \draw[thick, preaction={draw=white, line width=5pt}, ->] (2,0) node[below] {$\kappa v$} -- (2,0.3) -- (3,1.1) -- (1,2.7) -- (1,3.8) node[above] {$\kappa v$};
    \draw[thick, preaction={draw=white, line width=5pt}, ->] (3,0) node[below] {$v$} -- (3,0.3) -- (2,1.1) -- (0,2.7) -- (0,3.8) node[above] {$v$};

    \draw[thick, ->] (3.5,3.3) -- (-0.5,3.3);
    \filldraw[fill=black, thick] (0,3.3) circle (2.8pt);
    \filldraw[fill=black, thick] (1,3.3) circle (2.8pt);
    \filldraw[fill=white, thick] (2,3.3) circle (2.8pt);
    \filldraw[fill=white, thick] (3,3.3) circle (2.8pt);
  \end{scope}

  \node at (9.5,1.9) {$=$};

  \begin{scope}[shift={(10.6,0)}]
    \draw[thick, ->] (0,0) node[below] {$\kappa u$} -- (0,0.3) -- (2,1.9) -- (3,2.7) -- (3,3.8) node[above] {$\kappa u$};
    \draw[thick, preaction={draw=white, line width=5pt}, ->] (1,0) node[below] {$u$} -- (1,0.3) -- (3,1.9) -- (2,2.7) -- (2,3.8) node[above] {$u$};
    \draw[thick, preaction={draw=white, line width=5pt}, ->] (2,0) node[below] {$\kappa v$} -- (2,0.3) -- (0,1.9) -- (1,2.7) -- (1,3.8) node[above] {$\kappa v$};
    \draw[thick, preaction={draw=white, line width=5pt}, ->] (3,0) node[below] {$v$} -- (3,0.3) -- (1,1.9) -- (0,2.7) -- (0,3.8) node[above] {$v$};

    \draw[thick, ->] (3.5,3.3) -- (-0.5,3.3);
    \filldraw[fill=black, thick] (0,3.3) circle (2.8pt);
    \filldraw[fill=black, thick] (1,3.3) circle (2.8pt);
    \filldraw[fill=white, thick] (2,3.3) circle (2.8pt);
    \filldraw[fill=white, thick] (3,3.3) circle (2.8pt);
  \end{scope}
\end{tikzpicture}
\caption{Graphical derivation of~\eqref{eq:fused-RLL-check}.}
\label{fig:fused-RLL-derivation}
\end{figure}
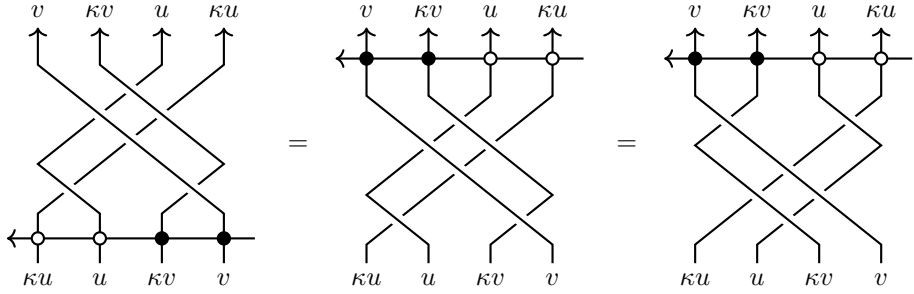

\begin{Rem}\label{rem:fusion-interpretation}
We note that grouping the four strands into the two pairs $(1,2)$ and $(3,4)$ reveals an important \emph{fusion structure}
for both sides of~\eqref{eq:fused-RLL-check}, see Figure~\ref{fig:fusion-decomp}:

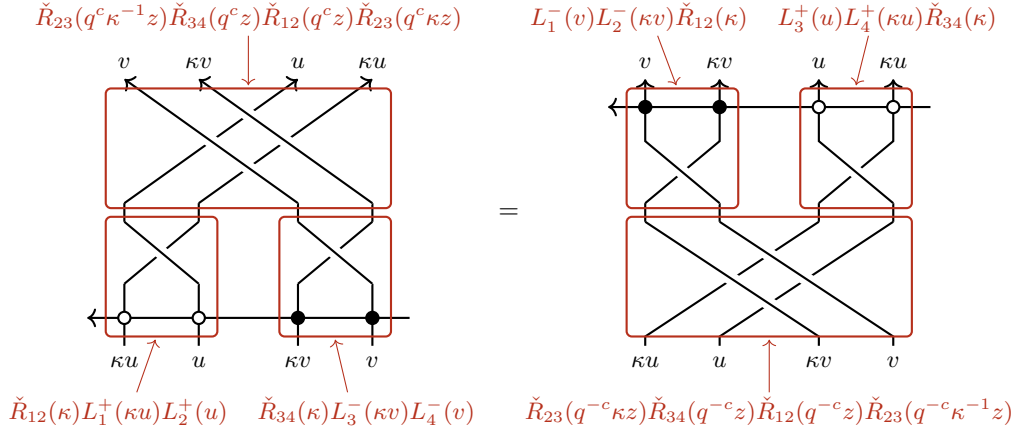
\begin{figure}[H]
\centering
\begin{tikzpicture}[baseline={(0,2.2)}, scale=0.82, every node/.style={font=\footnotesize}]
  \begin{scope}[shift={(0,0)}]
    \draw[thick] (0,0) node[below] {$\kappa u$} -- (0,1.0);
    \draw[thick] (1.2,0) node[below] {$u$} -- (1.2,1.0);
    \draw[thick] (2.8,0) node[below] {$\kappa v$} -- (2.8,1.0);
    \draw[thick] (4.0,0) node[below] {$v$} -- (4.0,1.0);

    \draw[thick, ->] (4.6,0.45) -- (-0.6,0.45);
    \filldraw[fill=white, thick] (0,0.45) circle (2.8pt);
    \filldraw[fill=white, thick] (1.2,0.45) circle (2.8pt);
    \filldraw[fill=black, thick] (2.8,0.45) circle (2.8pt);
    \filldraw[fill=black, thick] (4.0,0.45) circle (2.8pt);

    \draw[thick] (0,1.0) -- (1.2,2.0);
    \draw[thick, preaction={draw=white, line width=5pt}] (1.2,1.0) -- (0,2.0);

    \draw[thick] (2.8,1.0) -- (4.0,2.0);
    \draw[thick, preaction={draw=white, line width=5pt}] (4.0,1.0) -- (2.8,2.0);

    \foreach \x in {0, 1.2, 2.8, 4.0} {
      \draw[thick] (\x,2.0) -- (\x,2.3);
    }

    \draw[thick, ->] (0,2.3) -- (2.8,4.3) node[above] {$u$};
    \draw[thick, ->] (1.2,2.3) -- (4.0,4.3) node[above] {$\kappa u$};
    \draw[thick, preaction={draw=white, line width=5pt}, ->] (2.8,2.3) -- (0,4.3) node[above] {$v$};
    \draw[thick, preaction={draw=white, line width=5pt}, ->] (4.0,2.3) -- (1.2,4.3) node[above] {$\kappa v$};

    \draw[BrickRed, thick, rounded corners=2pt] (-0.3,0.15) rectangle (1.5,2.08);
    \node[BrickRed] at (-0.1,-1.05) {$\check{R}_{12}(\kappa)L^{+}_{1}(\kappa u)L^{+}_{2}(u)$};
    \draw[BrickRed, ->] (0.2,-0.75) -- (0.5,0.1);

    \draw[BrickRed, thick, rounded corners=2pt] (2.5,0.15) rectangle (4.3,2.08);
    \node[BrickRed] at (3.9,-1.05) {$\check{R}_{34}(\kappa)L^{-}_{3}(\kappa v)L^{-}_{4}(v)$};
    \draw[BrickRed, ->] (3.6,-0.75) -- (3.4,0.1);

    \draw[BrickRed, thick, rounded corners=2pt] (-0.3,2.22) rectangle (4.3,4.15);
    \node[BrickRed] at (2.0,5.3) {$\check{R}_{23}(q^{c}\kappa^{-1}z)\check{R}_{34}(q^{c}z)\check{R}_{12}(q^{c}z)\check{R}_{23}(q^{c}\kappa z)$};
    \draw[BrickRed, ->] (2.0,5.0) -- (2.0,4.2);
  \end{scope}

  \node[font=\small] at (6.2,2.15) {$=$};

  \begin{scope}[shift={(8.4,0)}]
    \draw[thick] (0,0) node[below] {$\kappa u$} -- (0,0.15) -- (2.8,2.0);
    \draw[thick] (1.2,0) node[below] {$u$} -- (1.2,0.15) -- (4.0,2.0);
    \draw[thick, preaction={draw=white, line width=5pt}] (2.8,0) node[below] {$\kappa v$} -- (2.8,0.15) -- (0,2.0);
    \draw[thick, preaction={draw=white, line width=5pt}] (4.0,0) node[below] {$v$} -- (4.0,0.15) -- (1.2,2.0);

    \foreach \x in {0, 1.2, 2.8, 4.0} {
      \draw[thick] (\x,2.0) -- (\x,2.3);
    }

    \draw[thick, ->] (0,2.3) -- (1.2,3.3) -- (1.2,4.3) node[above] {$\kappa v$};
    \draw[thick, preaction={draw=white, line width=5pt}, ->] (1.2,2.3) -- (0,3.3) -- (0,4.3) node[above] {$v$};

    \draw[thick, ->] (2.8,2.3) -- (4.0,3.3) -- (4.0,4.3) node[above] {$\kappa u$};
    \draw[thick, preaction={draw=white, line width=5pt}, ->] (4.0,2.3) -- (2.8,3.3) -- (2.8,4.3) node[above] {$u$};

    \draw[thick, ->] (4.6,3.85) -- (-0.6,3.85);
    \filldraw[fill=black, thick] (0,3.85) circle (2.8pt);
    \filldraw[fill=black, thick] (1.2,3.85) circle (2.8pt);
    \filldraw[fill=white, thick] (2.8,3.85) circle (2.8pt);
    \filldraw[fill=white, thick] (4.0,3.85) circle (2.8pt);

    \draw[BrickRed, thick, rounded corners=2pt] (-0.3,0.15) rectangle (4.3,2.08);
    \node[BrickRed] at (2.0,-1.05) {$\check{R}_{23}(q^{-c}\kappa z)\check{R}_{34}(q^{-c}z)\check{R}_{12}(q^{-c}z)\check{R}_{23}(q^{-c}\kappa^{-1}z)$};
    \draw[BrickRed, ->] (2.0,-0.75) -- (2.0,0.1);

    \draw[BrickRed, thick, rounded corners=2pt] (-0.3,2.22) rectangle (1.5,4.15);
    \node[BrickRed] at (-0.1,5.3) {$L^{-}_{1}(v)L^{-}_{2}(\kappa v)\check{R}_{12}(\kappa)$};
    \draw[BrickRed, ->] (0.2,5.0) -- (0.5,4.2);

    \draw[BrickRed, thick, rounded corners=2pt] (2.5,2.22) rectangle (4.3,4.15);
    \node[BrickRed] at (3.9,5.3) {$L^{+}_{3}(u)L^{+}_{4}(\kappa u)\check{R}_{34}(\kappa)$};
    \draw[BrickRed, ->] (3.6,5.0) -- (3.4,4.2);
  \end{scope}
\end{tikzpicture}
\caption{Fusion decomposition of~\eqref{eq:fused-RLL-check}.}
\label{fig:fusion-decomp}
\end{figure}

\noindent
According to~\eqref{eq:braided-RLL-1}, we have
\begin{equation*}
\begin{split}
  \check{R}_{12}(\kappa)L^{+}_{1}(\kappa u)L^{+}_{2}(u) &= L^{+}_{1}(u)L^{+}_{2}(\kappa u)\check{R}_{12}(\kappa) \,, \\
  \check{R}_{12}(\kappa)L^{-}_{1}(\kappa v)L^{-}_{2}(v) &= L^{-}_{1}(v)L^{-}_{2}(\kappa v)\check{R}_{12}(\kappa) \,,
\end{split}
\end{equation*}
while the blocks
  $\check{R}_{23}(q^{\pm c}\kappa^{\mp 1}z)\check{R}_{34}(q^{\pm c}z)\check{R}_{12}(q^{\pm c}z)\check{R}_{23}(q^{\pm c}\kappa^{\pm 1}z)$
braid the pair $(1,2)$ past the pair $(3,4)$. Thus,~\eqref{eq:fused-RLL-check} may be thought of as a fused counterpart
of~\eqref{eq:braided-RLL-2} for the pairs $(1,2)$ and $(3,4)$.
\end{Rem}


\medskip

\section{Proof of Theorem~\ref{thm:rank-reduction} and rational counterpart}\label{sec:proof-rank-reduction}

In this section, we prove Theorem~\ref{thm:rank-reduction} as well as derive its rational version (Theorem~\ref{thm:Y-rank-reduction}):
the rank reduction homomorphisms for orthosymplectic Yangians, established earlier in~\cite{ft,m}.

By Proposition~\ref{prop:affine-twisted-rtt} and Remarks~\ref{rem:twist-functoriality}, \ref{rem:quasidet-twist}, parts~(a) and~(b)
of Theorem~\ref{thm:rank-reduction} are equivalent. Thus it suffices to prove part (b). Throughout the rest of the proof, we work
in $\wt{U}(\ol{R}^{[s]}(z))$, fix indices $1 < i,j,k,l < 1'$, and set $z \coloneqq u/v$. We note that~\eqref{eq:S-RLL-0} is already
established in~\eqref{eq:s-zero-mode}.

We shall now establish~\eqref{eq:S-RLL-2}. Rewriting $\check{R}$ in terms of $\hat{\CR}$ via~\eqref{eq:calR-hat-def}
in~\eqref{eq:fused-RLL-check} (and canceling the common $\phi(\kappa)^{-2}$), and pairing both sides with $\langle k,1,1,l|$
and $|i,1,1,j\rangle$, we obtain
\begin{equation}\label{eq:fused-RLL-pm}
\begin{split}
  \phi(q^{c}\kappa z)^{-1} & \phi(q^{c}\kappa^{-1} z)^{-1} \phi(q^{c}z)^{-2} \cdot
    \langle k,1,1,l | \hat{\CR}_{23}(q^{c}\kappa^{-1} z) \hat{\CR}_{34}(q^{c}z) \hat{\CR}_{12}(q^{c}z) \hat{\CR}_{23}(q^{c}\kappa z) \\
  & \times \hat{\CR}_{12}(\kappa) \hat{\CR}_{34}(\kappa)
    L^{+}_{1}(\kappa u)L^{+}_{2}(u) L^{-}_{3}(\kappa v)L^{-}_{4}(v) | i,1,1,j \rangle \\
  = \phi(q^{-c}\kappa z)^{-1} & \phi(q^{-c}\kappa^{-1} z)^{-1} \phi(q^{-c}z)^{-2} \cdot
    \langle k,1,1,l |  L^{-}_{1}(v)L^{-}_{2}(\kappa v) L^{+}_{3}(u)L^{+}_{4}(\kappa u) \\
  &\times \hat{\CR}_{12}(\kappa) \hat{\CR}_{34}(\kappa)
    \hat{\CR}_{23}(q^{-c}\kappa z) \hat{\CR}_{34}(q^{-c}z) \hat{\CR}_{12}(q^{-c}z) \hat{\CR}_{23}(q^{-c}\kappa^{-1} z)
    | i,1,1,j \rangle \,.
\end{split}
\end{equation}
We shall now evaluate both sides of~\eqref{eq:fused-RLL-pm} in the limit $\kappa \to \kappa_{0}$, see~\eqref{eq:kappa-o}.


\subsection{Computation of the left-hand side of~\eqref{eq:fused-RLL-pm}}\label{ssec:proof-LHS}
\

\smallskip
\noindent
\emph{Step 1.}
Applying $\hat{\CR}_{23}(q^{c}\kappa^{-1} z)$ and then $\hat{\CR}_{34}(q^{c}z) \hat{\CR}_{12}(q^{c}z)$ to
$\langle k,1,1,l|$ via~\eqref{eq:calR-hat-explicit}, we obtain
\begin{equation*}
  \langle k,1,1,l| \hat{\CR}_{23}(q^{c}\kappa^{-1} z) = (-1)^{\ol{1}} \cdot
  \frac{q^{(\varepsilon_{1},\varepsilon_{1})}q^{c}\kappa^{-1}z - q^{-(\varepsilon_{1},\varepsilon_{1})}}
       {q^{c}\kappa^{-1}z - 1} \cdot \langle k,1,1,l|
\end{equation*}
and hence
\begin{equation}\label{eq:LHS-step1}
\begin{split}
  \langle k,1,1,l|& \hat{\CR}_{23}(q^{c}\kappa^{-1} z) \hat{\CR}_{34}(q^{c}z) \hat{\CR}_{12}(q^{c}z) \\
  =\, & (-1)^{\ol{1}} \cdot
   \frac{q^{(\varepsilon_{1},\varepsilon_{1})}q^{c}\kappa^{-1}z - q^{-(\varepsilon_{1},\varepsilon_{1})}}{q^{c}\kappa^{-1}z - 1} \\
  &\times \bigg\{
    (-1)^{\ol{1}(\ol{k}+\ol{l})} \langle 1,k,l,1|
    + (-1)^{\ol{1}\,\ol{l}} \frac{q-q^{-1}}{q^{c}z-1} \langle k,1,l,1| \\
  &\hspace{20pt}
    + (-1)^{\ol{1}\,\ol{k}} \frac{(q-q^{-1})q^{c}z}{q^{c}z-1} \langle 1,k,1,l|
    + \frac{(q-q^{-1})^{2}q^{c}z}{(q^{c}z-1)^{2}} \langle k,1,1,l|
  \bigg\} \,.
\end{split}
\end{equation}

\smallskip
\noindent
\emph{Step 2.}
Set
\begin{equation*}
  \sT_{\mathrm{L}}'(\kappa) \coloneqq L^{+}_{1}(\kappa u)L^{+}_{2}(u) L^{-}_{3}(\kappa v)L^{-}_{4}(v) |i,1,1,j \rangle \,, \qquad
  \sT_{\mathrm{L}}(\kappa) \coloneqq \hat{\CR}_{12}(\kappa) \hat{\CR}_{34}(\kappa) \sT_{\mathrm{L}}'(\kappa) \,.
\end{equation*}
Using~\eqref{eq:calR-hat-explicit}, we expand the action of $\hat{\CR}_{23}(q^{c}\kappa z)$ on each of the four covectors
in~\eqref{eq:LHS-step1}:
\begin{align}
  \langle k,1,l,1| \hat{\CR}_{23}(q^{c}\kappa z) \sT_{\mathrm{L}}(\kappa)
  &= \frac{(q-q^{-1})q^{c}\kappa z}{q^{c}\kappa z-1} \langle k,1,l,1| \sT_{\mathrm{L}}(\kappa)
   + (-1)^{\ol{1}\,\ol{l}} \langle k,l,1,1| \sT_{\mathrm{L}}(\kappa) \,, \label{eq:LHS-term-k1l1}\\
  \langle 1,k,1,l| \hat{\CR}_{23}(q^{c}\kappa z) \sT_{\mathrm{L}}(\kappa)
  &= \frac{q-q^{-1}}{q^{c}\kappa z-1} \langle 1,k,1,l| \sT_{\mathrm{L}}(\kappa)
   + (-1)^{\ol{1}\,\ol{k}} \langle 1,1,k,l| \sT_{\mathrm{L}}(\kappa) \,, \label{eq:LHS-term-1k1l}\\
  \langle k,1,1,l| \hat{\CR}_{23}(q^{c}\kappa z) \sT_{\mathrm{L}}(\kappa)
  &= (-1)^{\ol{1}} \frac{q^{(\varepsilon_{1},\varepsilon_{1})} q^{c}\kappa z - q^{-(\varepsilon_{1},\varepsilon_{1})}}
     {q^{c}\kappa z-1} \langle k,1,1,l| \sT_{\mathrm{L}}(\kappa) \,, \label{eq:LHS-term-k11l}
\end{align}
and finally
{\allowdisplaybreaks
\begin{equation}\label{eq:LHS-term-1kl1}
\begin{split}
  \langle 1,k,l,1|& \hat{\CR}_{23}(q^{c}\kappa z) \sT_{\mathrm{L}}(\kappa) \\
  = \Bigg\{ &
    \delta_{k \neq l,l'} (-1)^{\ol{k}\,\ol{l}} \langle 1,l,k,1| + \delta_{k=l} \delta_{k \neq k'} (-1)^{\ol{k}}
    \frac{q^{(\varepsilon_{k},\varepsilon_{k})}q^{c}\kappa z - q^{-(\varepsilon_{k},\varepsilon_{k})}}{q^{c}\kappa z-1} \langle 1,k,k,1| \\
  & + \delta_{k=l'} \delta_{k \neq k'} (-1)^{\ol{k}}
    \frac{q^{-(\varepsilon_{k},\varepsilon_{k})}q^{c}\kappa z - q^{(\varepsilon_{k},\varepsilon_{k})}\xi}
         {q^{c}\kappa z-\xi} \langle 1,k',k,1| \\
  & + \delta_{k=l} \delta_{k=k'} (-1)^{\ol{k}} \bigg( 1 +
    \frac{(q-q^{-1})(1-\xi)q^{c}\kappa z}{(q^{c}\kappa z-1)(q^{c}\kappa z-\xi)} \bigg) \langle 1,k,k,1| \\
  & + \delta_{k<l} \frac{(q-q^{-1})q^{c}\kappa z}{q^{c}\kappa z-1} \langle 1,k,l,1|
    + \delta_{k>l} \frac{q-q^{-1}}{q^{c}\kappa z-1} \langle 1,k,l,1| \\
  & - \delta_{k = l'} \frac{(q-q^{-1})q^{c}\kappa z}{q^{c}\kappa z-\xi} \sum_{1 < b < l} \vartheta_{l}\vartheta_{b}
    q^{(\rho,\varepsilon_{l}-\varepsilon_{b})}q^{-(\varepsilon_{l},\varepsilon_{l})/2}q^{(\varepsilon_{b},\varepsilon_{b})/2}
    \langle 1,b,b',1| \\
  & - \delta_{k = l'} \frac{(q-q^{-1})\xi}{q^{c}\kappa z-\xi}
      \sum_{l < b < 1'} \vartheta_{l}\vartheta_{b} q^{(\rho,\varepsilon_{l}-\varepsilon_{b})}
      q^{-(\varepsilon_{l},\varepsilon_{l})/2}q^{(\varepsilon_{b},\varepsilon_{b})/2} \langle 1,b,b',1| \\
  & - \delta_{k = l'} \frac{(q-q^{-1})q^{c}\kappa z}{q^{c}\kappa z-\xi}
      \vartheta_{l}\vartheta_{1} q^{(\rho,\varepsilon_{l}-\varepsilon_{1})}
      q^{-(\varepsilon_{l},\varepsilon_{l})/2}q^{(\varepsilon_{1},\varepsilon_{1})/2} \langle 1,1,1',1| \\
  & - \delta_{k = l'} \frac{(q-q^{-1})\xi}{q^{c}\kappa z-\xi}
      \vartheta_{l}\vartheta_{1'} q^{(\rho,\varepsilon_{l}-\varepsilon_{1'})}
      q^{-(\varepsilon_{l},\varepsilon_{l})/2}q^{(\varepsilon_{1},\varepsilon_{1})/2} \langle 1,1',1,1| \Bigg\}
      \sT_{\mathrm{L}}(\kappa) \,.
\end{split}
\end{equation}}

\smallskip
\noindent
\emph{Step 3.}
Before taking $\lim_{\kappa \to \kappa_{0}}$, let us collect the four ``boundary terms''
in~\eqref{eq:LHS-term-k1l1}, \eqref{eq:LHS-term-1k1l}, and~\eqref{eq:LHS-term-1kl1} involving $\langle 1,1|$
in either the first two or the last two tensor factors (i.e.\ involving $\langle 1,1,\ast,\ast|$ or $\langle \ast,\ast,1,1|$).
In the linear combination of~\eqref{eq:LHS-term-k1l1}--\eqref{eq:LHS-term-1kl1} inside the braces of~\eqref{eq:LHS-step1}
(i.e.\ before multiplying by the overall prefactor
  $(-1)^{\ol{1}} \frac{q^{(\varepsilon_{1},\varepsilon_{1})}q^{c}\kappa^{-1}z - q^{-(\varepsilon_{1},\varepsilon_{1})}}
   {q^{c}\kappa^{-1}z - 1}$),
the combined contribution of these four terms is given by
\begin{equation*}
\begin{split}
  \Sigma_{\mathrm{L}}(\kappa) &\coloneqq
  \frac{q-q^{-1}}{q^{c}z-1} \langle k,l,1,1| \sT_{\mathrm{L}}(\kappa)
  + \frac{(q-q^{-1})q^{c}z}{q^{c}z-1} \langle 1,1,k,l| \sT_{\mathrm{L}}(\kappa) \\
  &\quad - \delta_{k = l'} \bigg\{
    \frac{(q-q^{-1})q^{c}\kappa z}{q^{c}\kappa z-\xi}
    \vartheta_{l}\vartheta_{1} q^{(\rho,\varepsilon_{l}-\varepsilon_{1})}
    q^{-(\varepsilon_{l},\varepsilon_{l})/2}q^{(\varepsilon_{1},\varepsilon_{1})/2} \langle 1,1,1',1| \\
  &\hspace{55pt}
    + \frac{(q-q^{-1})\xi}{q^{c}\kappa z-\xi}
    \vartheta_{l}\vartheta_{1'} q^{(\rho,\varepsilon_{l}-\varepsilon_{1'})}
    q^{-(\varepsilon_{l},\varepsilon_{l})/2}q^{(\varepsilon_{1},\varepsilon_{1})/2} \langle 1,1',1,1|
  \bigg\} \sT_{\mathrm{L}}(\kappa) \,.
\end{split}
\end{equation*}
Recall that $\sT_{\mathrm{L}}(\kappa) = \hat{\CR}_{12}(\kappa) \hat{\CR}_{34}(\kappa) \sT_{\mathrm{L}}'(\kappa)$, so in each term
of $\Sigma_{\mathrm{L}}(\kappa)$, one of the factors $\hat{\CR}_{12}(\kappa), \hat{\CR}_{34}(\kappa)$ acts on $\langle 1,1|$,
while the other acts on $\langle k,l|$, $\langle 1',1|$, or $\langle 1,1'|$.
We claim that $\lim_{\kappa \to \kappa_{0}} \Sigma_{\mathrm{L}}(\kappa) = 0$:
\begin{enumerate}[label=(\roman*), leftmargin=0.6cm]

\item
If $\kappa_{0} \neq \xi$, then $\hat{\CR}(\kappa)$ is regular at $\kappa = \kappa_{0}$ (since $\kappa_{0} \neq 1$).
Thus, in each term of $\Sigma_{\mathrm{L}}(\kappa)$, the copy of $\hat{\CR}(\kappa)$ acting on $\langle 1,1|$ produces $0$ in
the limit $\kappa \to \kappa_{0}$ by~\eqref{eq:quasidet-11-lim}, while the other copy of $\hat{\CR}(\kappa)$ remains regular,
so $\lim_{\kappa \to \kappa_{0}} \Sigma_{\mathrm{L}}(\kappa) = 0$.

\item
If $\kappa_{0} = \xi$, then $\hat{\CR}(\kappa)$ has a simple pole at $\kappa = \kappa_{0}$. By~\eqref{eq:calR-hat-explicit},
$\langle k,l|\hat{\CR}(\kappa)$ is regular at $\kappa = \kappa_{0}$ for $k \neq l'$ (so the first two terms of
$\Sigma_{\mathrm{L}}(\kappa)$ vanish as $\kappa \to \kappa_{0}$ unless $k = l'$), while on $\langle r',r|$ ($1 \leq r \leq 1'$)
we have
\begin{equation*}
  \lim_{\kappa \to \kappa_{0}} (\kappa - \kappa_{0}) \langle r',r| \hat{\CR}(\kappa)
  = -(q-q^{-1})\kappa_{0} \sum_{1 \leq b \leq 1'} \vartheta_{r}\vartheta_{b} q^{(\rho,\varepsilon_{r}-\varepsilon_{b})}
    q^{-(\varepsilon_{r},\varepsilon_{r})/2}q^{(\varepsilon_{b},\varepsilon_{b})/2} \langle b,b'| \,,
\end{equation*}
where the $b \neq r$ summands come from the last two sums in~\eqref{eq:calR-hat-explicit} and the $b = r$ summand comes from
the $E_{r'r} \otimes E_{rr'}$ if $r\ne r'$ or $E_{rr} \otimes E_{rr}$ if $r = r'$ terms of~\eqref{eq:calR-hat-explicit}.
Combined with~\eqref{eq:quasidet-11}, which gives
\begin{equation*}
    \lim_{\kappa \to \kappa_{0}} (\kappa - \kappa_{0})^{-1} \langle 1,1|\hat{\CR}(\kappa)
    = \frac{(-1)^{\ol{1}}q^{(\varepsilon_{1},\varepsilon_{1})}}{\kappa_{0}-1}\langle 1,1|
    = -\frac{1}{(q-q^{-1})\kappa_{0}}\langle 1,1| \,,
\end{equation*}
we obtain for any $1 \leq r \leq 1'$:
\begin{align*}
    \lim_{\kappa \to \kappa_{0}} \langle r',r,1,1| \hat{\CR}_{12}(\kappa)\hat{\CR}_{34}(\kappa)
    & = \sum_{1 \leq b \leq 1'} \vartheta_{r}\vartheta_{b} q^{(\rho,\varepsilon_{r}-\varepsilon_{b})}
      q^{-(\varepsilon_{r},\varepsilon_{r})/2}q^{(\varepsilon_{b},\varepsilon_{b})/2} \langle b,b',1,1| \,, \\
      \lim_{\kappa \to \kappa_{0}} \langle 1,1,r',r| \hat{\CR}_{12}(\kappa)\hat{\CR}_{34}(\kappa)
    & = \sum_{1 \leq b \leq 1'} \vartheta_{r}\vartheta_{b} q^{(\rho,\varepsilon_{r}-\varepsilon_{b})}
      q^{-(\varepsilon_{r},\varepsilon_{r})/2}q^{(\varepsilon_{b},\varepsilon_{b})/2} \langle 1,1,b,b'| \,.
\end{align*}
Substituting these limits (for $r=l$, $r=1$, and $r=1'$) into $\Sigma_{\mathrm{L}}(\kappa)$ and using $\kappa_{0} = \xi$ gives
  \begin{align*}
    \lim_{\kappa \to \kappa_{0}} \Sigma_{\mathrm{L}}(\kappa)
    &= \delta_{k=l'} \frac{q-q^{-1}}{q^{c}z-1} \sum_{1 \leq b \leq 1'} \vartheta_{l}\vartheta_{b}
      q^{(\rho,\varepsilon_{l}-\varepsilon_{b})}q^{-(\varepsilon_{l},\varepsilon_{l})/2}q^{(\varepsilon_{b},\varepsilon_{b})/2}
      \langle b,b',1,1| \sT_{\mathrm{L}}'(\kappa_{0}) \\
    &\quad + \delta_{k=l'} \frac{(q-q^{-1})q^{c}z}{q^{c}z-1} \sum_{1 \leq b \leq 1'} \vartheta_{l}\vartheta_{b}
      q^{(\rho,\varepsilon_{l}-\varepsilon_{b})}q^{-(\varepsilon_{l},\varepsilon_{l})/2}q^{(\varepsilon_{b},\varepsilon_{b})/2}
      \langle 1,1,b,b'| \sT_{\mathrm{L}}'(\kappa_{0}) \\
    &\quad - \delta_{k=l'} \frac{(q-q^{-1})q^{c}z}{q^{c}z-1} \vartheta_{l}\vartheta_{1} q^{(\rho,\varepsilon_{l}-\varepsilon_{1})}
      q^{-(\varepsilon_{l},\varepsilon_{l})/2}q^{(\varepsilon_{1},\varepsilon_{1})/2} \\
    &\hspace{75pt} \times
      \sum_{1 \leq b \leq 1'} \vartheta_{1}\vartheta_{b} q^{(\rho,\varepsilon_{1}-\varepsilon_{b})}
      q^{-(\varepsilon_{1},\varepsilon_{1})/2}q^{(\varepsilon_{b},\varepsilon_{b})/2}
      \langle 1,1,b,b'| \sT_{\mathrm{L}}'(\kappa_{0}) \\
    &\quad - \delta_{k=l'} \frac{q-q^{-1}}{q^{c}z-1} \vartheta_{l}\vartheta_{1'} q^{(\rho,\varepsilon_{l}-\varepsilon_{1'})}
      q^{-(\varepsilon_{l},\varepsilon_{l})/2}q^{(\varepsilon_{1},\varepsilon_{1})/2} \\
    &\hspace{75pt} \times
      \sum_{1 \leq b \leq 1'} \vartheta_{1'}\vartheta_{b} q^{(\rho,\varepsilon_{1'}-\varepsilon_{b})}
      q^{-(\varepsilon_{1},\varepsilon_{1})/2}q^{(\varepsilon_{b},\varepsilon_{b})/2}
      \langle b,b',1,1| \sT_{\mathrm{L}}'(\kappa_{0}) \\
    &= 0 \,.
  \end{align*}
\end{enumerate}

\smallskip
\noindent
\emph{Step 4.}
For any $1 < a,b,c,d < 1'$, let us set
\begin{equation}\label{eq:chi-sign-def}
  \chi_{ab,cd} \coloneqq
  (-1)^{(\ol{1}+\ol{a})(\ol{1}+\ol{b}) + (\ol{1}+\ol{c})(\ol{1}+\ol{d}) + (\ol{1}+\ol{c})(\ol{1}+\ol{b})} \,.
\end{equation}
Since
  $\sT_{\mathrm{L}}(\kappa) = \big(\hat{\CR}_{12}(\kappa) L^{+}_{1}(\kappa u)L^{+}_{2}(u)\big)
   \big(\hat{\CR}_{34}(\kappa) L^{-}_{3}(\kappa v)L^{-}_{4}(v)\big) |i,1,1,j\rangle$,
applying~\eqref{eq:twisted-4-matrix-element} to the remaining terms of~\eqref{eq:LHS-term-k1l1}--\eqref{eq:LHS-term-1kl1} and
evaluating the $\kappa\to \kappa_0$ limit of the resulting factors via~\eqref{eq:quasidet-1i1j}--\eqref{eq:quasidet-i1j1}, we obtain:
\begin{align*}
  \lim_{\kappa \to \kappa_{0}} \langle k,1,l,1| \hat{\CR}_{23}(q^{c}\kappa z) \sT_{\mathrm{L}}(\kappa)
  &\equiv \frac{(q-q^{-1})q^{c}\kappa_{0} z}{q^{c}\kappa_{0} z-1}
    \chi_{kl,ij} \Big(-(-1)^{\ol{1}(\ol{1}+\ol{k})} q^{(\varepsilon_{1},\varepsilon_{1})}\Big)
    (-1)^{\ol{1}(\ol{l}+\ol{j})} \wt{\gamma}^{+}_{ki}(u) \wt{\gamma}^{-}_{lj}(v) \,, \\
  \lim_{\kappa \to \kappa_{0}} \langle 1,k,1,l| \hat{\CR}_{23}(q^{c}\kappa z) \sT_{\mathrm{L}}(\kappa)
  &\equiv \frac{q-q^{-1}}{q^{c}\kappa_{0} z-1}
    \chi_{kl,ij} \Big(-(-1)^{\ol{1}(\ol{1}+\ol{j})} q^{-(\varepsilon_{1},\varepsilon_{1})}\Big)
    \wt{\gamma}^{+}_{ki}(u) \wt{\gamma}^{-}_{lj}(v) \,, \\
  \lim_{\kappa \to \kappa_{0}} \langle k,1,1,l| \hat{\CR}_{23}(q^{c}\kappa z) \sT_{\mathrm{L}}(\kappa)
  &= (-1)^{\ol{1}} \frac{q^{(\varepsilon_{1},\varepsilon_{1})} q^{c}\kappa_{0} z - q^{-(\varepsilon_{1},\varepsilon_{1})}}
    {q^{c}\kappa_{0} z-1} \chi_{kl,ij} (-1)^{\ol{1}(\ol{k}+\ol{j})} \wt{\gamma}^{+}_{ki}(u) \wt{\gamma}^{-}_{lj}(v) \,,
\end{align*}
and finally
{\allowdisplaybreaks
\begin{align*}
  \lim_{\kappa \to \kappa_{0}} & \langle 1,k,l,1| \hat{\CR}_{23}(q^{c}\kappa z) \sT_{\mathrm{L}}(\kappa) \\
  &\equiv \delta_{k \neq l,l'} (-1)^{\ol{k}\,\ol{l}}
    \chi_{lk,ij} (-1)^{\ol{1}(\ol{k}+\ol{j})} \wt{\gamma}^{+}_{li}(u) \wt{\gamma}^{-}_{kj}(v) \\
  &\quad + \delta_{k=l} \delta_{k \neq k'} (-1)^{\ol{k}}
    \frac{q^{(\varepsilon_{k},\varepsilon_{k})}q^{c}\kappa_{0} z - q^{-(\varepsilon_{k},\varepsilon_{k})}}{q^{c}\kappa_{0} z-1}
    \chi_{kl,ij} (-1)^{\ol{1}(\ol{l}+\ol{j})} \wt{\gamma}^{+}_{ki}(u) \wt{\gamma}^{-}_{kj}(v) \\
  &\quad + \delta_{k=l'} \delta_{k \neq k'} (-1)^{\ol{k}}
    \frac{q^{-(\varepsilon_{k},\varepsilon_{k})}q^{c}\kappa_{0} z - q^{(\varepsilon_{k},\varepsilon_{k})}\xi}{q^{c}\kappa_{0} z-\xi}
    \chi_{kl,ij} (-1)^{\ol{1}(\ol{l}+\ol{j})} \wt{\gamma}^{+}_{k'i}(u) \wt{\gamma}^{-}_{kj}(v) \\
  &\quad + \delta_{k=l} \delta_{k=k'} (-1)^{\ol{k}}
    \bigg( 1+ \frac{(q-q^{-1})(1-\xi)q^{c}\kappa_{0} z}{(q^{c}\kappa_{0} z-1)(q^{c}\kappa_{0} z-\xi)} \bigg)
    \chi_{kl,ij} (-1)^{\ol{1}(\ol{l}+\ol{j})} \wt{\gamma}^{+}_{ki}(u) \wt{\gamma}^{-}_{kj}(v) \\
  &\quad + \delta_{k<l} \frac{(q-q^{-1})q^{c}\kappa_{0} z}{q^{c}\kappa_{0} z-1}
    \chi_{kl,ij} (-1)^{\ol{1}(\ol{l}+\ol{j})} \wt{\gamma}^{+}_{ki}(u) \wt{\gamma}^{-}_{lj}(v) \\
  &\quad + \delta_{k>l} \frac{q-q^{-1}}{q^{c}\kappa_{0} z-1}
    \chi_{kl,ij} (-1)^{\ol{1}(\ol{l}+\ol{j})} \wt{\gamma}^{+}_{ki}(u) \wt{\gamma}^{-}_{lj}(v) \\
  &\quad - \delta_{k = l'} \frac{(q-q^{-1})q^{c}\kappa_{0} z}{q^{c}\kappa_{0} z-\xi}
    \sum_{1 < b < l} \vartheta_{l}\vartheta_{b} q^{(\rho,\varepsilon_{l}-\varepsilon_{b})}
    q^{-(\varepsilon_{l},\varepsilon_{l})/2}q^{(\varepsilon_{b},\varepsilon_{b})/2}
    \chi_{bb',ij} (-1)^{\ol{1}(\ol{b}+\ol{j})} \wt{\gamma}^{+}_{bi}(u) \wt{\gamma}^{-}_{b'j}(v) \\
  &\quad - \delta_{k = l'} \frac{(q-q^{-1})\xi}{q^{c}\kappa_{0} z-\xi}
    \sum_{l < b < 1'} \vartheta_{l}\vartheta_{b} q^{(\rho,\varepsilon_{l}-\varepsilon_{b})}
    q^{-(\varepsilon_{l},\varepsilon_{l})/2}q^{(\varepsilon_{b},\varepsilon_{b})/2}
    \chi_{bb',ij} (-1)^{\ol{1}(\ol{b}+\ol{j})} \wt{\gamma}^{+}_{bi}(u) \wt{\gamma}^{-}_{b'j}(v) \,,
\end{align*}}
where $\equiv$ denotes equality modulo the ``boundary terms'' already absorbed into $\Sigma_{\mathrm{L}}(\kappa)$ from Step~3.

\smallskip
\noindent
\emph{Step 5.}
Substituting the first three limits from Step~4 into the linear combination of $\langle k,1,l,1|$, $\langle 1,k,1,l|$,
and $\langle k,1,1,l|$ in~\eqref{eq:LHS-step1} yields
\begin{align*}
  \lim_{\kappa \to \kappa_{0}} & \bigg\{
    (-1)^{\ol{1}\,\ol{l}} \frac{q-q^{-1}}{q^{c}z-1} \langle k,1,l,1|
    + (-1)^{\ol{1}\,\ol{k}} \frac{(q-q^{-1})q^{c}z}{q^{c}z-1} \langle 1,k,1,l| \\
  &\hspace{190pt}
    + \frac{(q-q^{-1})^{2}q^{c}z}{(q^{c}z-1)^{2}} \langle k,1,1,l|
  \bigg\} \hat{\CR}_{23}(q^{c}\kappa z) \sT_{\mathrm{L}}(\kappa) \\
  &\equiv -\frac{(q-q^{-1})^{2}q^{c}z}{(q^{c}z-1)(q^{c}\kappa_{0} z-1)}
   \chi_{kl,ij} (-1)^{\ol{1}(\ol{1}+\ol{k}+\ol{j})} \cdot q^{-(\varepsilon_{1},\varepsilon_{1})}
   \wt{\gamma}^{+}_{ki}(u) \wt{\gamma}^{-}_{lj}(v) \,.
\end{align*}
Adding this to
  $(-1)^{\ol{1}(\ol{k}+\ol{l})} \lim_{\kappa \to \kappa_{0}} \langle 1,k,l,1| \hat{\CR}_{23}(q^{c}\kappa z) \sT_{\mathrm{L}}(\kappa)$
from Step~4 and multiplying by the overall prefactor
  $(-1)^{\ol{1}} \frac{q^{(\varepsilon_{1},\varepsilon_{1})}q^{c}\kappa^{-1}z - q^{-(\varepsilon_{1},\varepsilon_{1})}}
   {q^{c}\kappa^{-1}z - 1}$
(while using $\lim_{\kappa \to \kappa_{0}} \Sigma_{\mathrm{L}}(\kappa) = 0$ from Step~3), we~get:
{\allowdisplaybreaks
\begin{align*}
  (-1)^{\ol{1}} & \left(
    \frac{q^{(\varepsilon_{1},\varepsilon_{1})}q^{c}\kappa_{0}^{-1}z - q^{-(\varepsilon_{1},\varepsilon_{1})}}
         {q^{c}\kappa_{0}^{-1}z - 1}\right)^{-1}
  \lim_{\kappa \to \kappa_{0}} \langle k,1,1,l |
  \hat{\CR}_{23}(q^{c}\kappa^{-1} z) \hat{\CR}_{34}(q^{c}z) \hat{\CR}_{12}(q^{c}z) \\
  &\hspace{140pt} \times
  \hat{\CR}_{23}(q^{c}\kappa z) \hat{\CR}_{12}(\kappa) \hat{\CR}_{34}(\kappa)
  L^{+}_{1}(\kappa u)L^{+}_{2}(u) L^{-}_{3}(\kappa v)L^{-}_{4}(v)
  | i,1,1,j \rangle \\
  &= \delta_{k \neq l,l'} (-1)^{\ol{i}\,\ol{j} + \ol{k}\,\ol{i}} \wt{\gamma}^{+}_{li}(u) \wt{\gamma}^{-}_{kj}(v) \\
  &\quad + \delta_{k=l} \delta_{k \neq k'} (-1)^{\ol{k}}
    \frac{q^{(\varepsilon_{k},\varepsilon_{k})}q^{c}z - q^{-(\varepsilon_{k},\varepsilon_{k})}}{q^{c}z-1}
    (-1)^{\ol{k} + \ol{i}\,\ol{j} + \ol{i}\,\ol{l}} \wt{\gamma}^{+}_{ki}(u) \wt{\gamma}^{-}_{kj}(v) \\
  &\quad + \delta_{k=l'} \delta_{k \neq k'} (-1)^{\ol{k}}
    \frac{q^{-(\varepsilon_{k},\varepsilon_{k})}q^{c}\kappa_{0} z - q^{(\varepsilon_{k},\varepsilon_{k})}\xi}{q^{c}\kappa_{0} z-\xi}
    (-1)^{\ol{k} + \ol{i}\,\ol{j} + \ol{i}\,\ol{l}} \wt{\gamma}^{+}_{k'i}(u) \wt{\gamma}^{-}_{kj}(v) \\
  &\quad + \delta_{k=l} \delta_{k=k'} (-1)^{\ol{k}}
    \bigg( 1+ \frac{(q-q^{-1})(\kappa_{0}-\xi)q^{c}z}{(q^{c}z-1)(q^{c}\kappa_{0} z-\xi)} \bigg)
    (-1)^{\ol{k} + \ol{i}\,\ol{j} + \ol{i}\,\ol{l}} \wt{\gamma}^{+}_{ki}(u) \wt{\gamma}^{-}_{kj}(v) \\
  &\quad + \delta_{k<l} \frac{(q-q^{-1})q^{c}z}{q^{c}z-1}
    (-1)^{\ol{k}\,\ol{l} + \ol{i}\,\ol{j} + \ol{i}\,\ol{l}} \wt{\gamma}^{+}_{ki}(u) \wt{\gamma}^{-}_{lj}(v)
    + \delta_{k>l} \frac{q-q^{-1}}{q^{c}z-1}
    (-1)^{\ol{k}\,\ol{l} + \ol{i}\,\ol{j} + \ol{i}\,\ol{l}} \wt{\gamma}^{+}_{ki}(u) \wt{\gamma}^{-}_{lj}(v) \\
  &\quad - \delta_{k = l'} \frac{(q-q^{-1})q^{c}\kappa_{0} z}{q^{c}\kappa_{0} z-\xi}
    \sum_{1 < b < l} \vartheta_{l}\vartheta_{b} q^{(\rho,\varepsilon_{l}-\varepsilon_{b})}
    q^{-(\varepsilon_{l},\varepsilon_{l})/2}q^{(\varepsilon_{b},\varepsilon_{b})/2}
    (-1)^{\ol{b} + \ol{i}\,\ol{j} + \ol{i}\,\ol{b}} \wt{\gamma}^{+}_{bi}(u) \wt{\gamma}^{-}_{b'j}(v) \\
  &\quad - \delta_{k = l'} \frac{(q-q^{-1})\xi}{q^{c}\kappa_{0} z-\xi}
    \sum_{l < b < 1'} \vartheta_{l}\vartheta_{b} q^{(\rho,\varepsilon_{l}-\varepsilon_{b})}
    q^{-(\varepsilon_{l},\varepsilon_{l})/2}q^{(\varepsilon_{b},\varepsilon_{b})/2}
    (-1)^{\ol{b} + \ol{i}\,\ol{j} + \ol{i}\,\ol{b}} \wt{\gamma}^{+}_{bi}(u) \wt{\gamma}^{-}_{b'j}(v) \\
  &= \langle k,l | \hat{\CR}^{[s-1]}_{12}(q^{c}z) \wt{\Gamma}^{+}_{1}(u)\wt{\Gamma}^{-}_{2}(v) | i,j \rangle \,,
\end{align*}}
with the last equality due to~\eqref{eq:calR-hat-explicit},~\eqref{eq:Gamma-def}, and~\eqref{eq:parameter-shifts}
(we substitute $\xi = \kappa_{0}\xi^{[s-1]}$ and cancel~$\kappa_{0}$).


\subsection{Computation of the right-hand side of~\eqref{eq:fused-RLL-pm}}\label{ssec:proof-RHS}
\

\smallskip
\noindent
\emph{Step 1.}
Applying $\hat{\CR}_{23}(q^{-c}\kappa^{-1} z)$ and then $\hat{\CR}_{34}(q^{-c}z) \hat{\CR}_{12}(q^{-c}z)$
to $|i,1,1,j\rangle$ via~\eqref{eq:calR-hat-explicit}, we obtain
\begin{equation*}
  \hat{\CR}_{23}(q^{-c}\kappa^{-1} z) |i,1,1,j\rangle = (-1)^{\ol{1}} \cdot
  \frac{q^{(\varepsilon_{1},\varepsilon_{1})}q^{-c}\kappa^{-1}z - q^{-(\varepsilon_{1},\varepsilon_{1})}}
       {q^{-c}\kappa^{-1}z - 1} \cdot |i,1,1,j\rangle
\end{equation*}
and hence
\begin{equation}\label{eq:RHS-step1}
\begin{split}
  \hat{\CR}_{34}(q^{-c}z) \hat{\CR}_{12}(q^{-c}z) & \hat{\CR}_{23}(q^{-c}\kappa^{-1} z) |i,1,1,j\rangle \\
  =\, & (-1)^{\ol{1}} \cdot
   \frac{q^{(\varepsilon_{1},\varepsilon_{1})}q^{-c}\kappa^{-1}z - q^{-(\varepsilon_{1},\varepsilon_{1})}}{q^{-c}\kappa^{-1}z - 1} \\
  &\times \bigg\{
    (-1)^{\ol{1}(\ol{i}+\ol{j})} |1,i,j,1\rangle
    + (-1)^{\ol{1}\,\ol{j}} \frac{q-q^{-1}}{q^{-c}z-1} |i,1,j,1\rangle \\
  &\hspace{20pt}
    + (-1)^{\ol{1}\,\ol{i}} \frac{(q-q^{-1})q^{-c}z}{q^{-c}z-1} |1,i,1,j\rangle
    + \frac{(q-q^{-1})^{2}q^{-c}z}{(q^{-c}z-1)^{2}} |i,1,1,j\rangle
  \bigg\} \,.
\end{split}
\end{equation}

\smallskip
\noindent
\emph{Step 2.}
Set
\begin{equation*}
  \sT_{\mathrm{R}}'(\kappa) \coloneqq \langle k,1,1,l | L^{-}_{1}(v)L^{-}_{2}(\kappa v) L^{+}_{3}(u)L^{+}_{4}(\kappa u) \,, \qquad
  \sT_{\mathrm{R}}(\kappa) \coloneqq \sT_{\mathrm{R}}'(\kappa) \hat{\CR}_{12}(\kappa) \hat{\CR}_{34}(\kappa) \,.
\end{equation*}
Using~\eqref{eq:calR-hat-explicit}, we expand the action of $\hat{\CR}_{23}(q^{-c}\kappa z)$ on each of the four vectors
in~\eqref{eq:RHS-step1}:
\begin{align}
  \sT_{\mathrm{R}}(\kappa) \hat{\CR}_{23}(q^{-c}\kappa z) |i,1,j,1\rangle
  &= \frac{(q-q^{-1})q^{-c}\kappa z}{q^{-c}\kappa z-1} \sT_{\mathrm{R}}(\kappa) |i,1,j,1\rangle
   + (-1)^{\ol{1}\,\ol{j}} \sT_{\mathrm{R}}(\kappa) |i,j,1,1\rangle \,, \label{eq:RHS-term-i1j1}\\
  \sT_{\mathrm{R}}(\kappa) \hat{\CR}_{23}(q^{-c}\kappa z) |1,i,1,j\rangle
  &= \frac{q-q^{-1}}{q^{-c}\kappa z-1} \sT_{\mathrm{R}}(\kappa) |1,i,1,j\rangle
   + (-1)^{\ol{1}\,\ol{i}} \sT_{\mathrm{R}}(\kappa) |1,1,i,j\rangle \,, \label{eq:RHS-term-1i1j}\\
  \sT_{\mathrm{R}}(\kappa) \hat{\CR}_{23}(q^{-c}\kappa z) |i,1,1,j\rangle
  &= (-1)^{\ol{1}} \frac{q^{(\varepsilon_{1},\varepsilon_{1})} q^{-c}\kappa z - q^{-(\varepsilon_{1},\varepsilon_{1})}}{q^{-c}\kappa z-1}
  \sT_{\mathrm{R}}(\kappa) |i,1,1,j\rangle \,, \label{eq:RHS-term-i11j}
\end{align}
and finally
{\allowdisplaybreaks
\begin{equation}\label{eq:RHS-term-1ij1}
\begin{split}
  \sT_{\mathrm{R}}(\kappa) & \hat{\CR}_{23}(q^{-c}\kappa z) |1,i,j,1\rangle = \sT_{\mathrm{R}}(\kappa) \\
  \times \Bigg\{ &
    \delta_{i \neq j,j'} (-1)^{\ol{i}\,\ol{j}} |1,j,i,1\rangle
    + \delta_{i=j} \delta_{i \neq i'} (-1)^{\ol{i}}
    \frac{q^{(\varepsilon_{i},\varepsilon_{i})}q^{-c}\kappa z - q^{-(\varepsilon_{i},\varepsilon_{i})}}
         {q^{-c}\kappa z-1} |1,i,i,1\rangle \\
  & + \delta_{i=j'} \delta_{i \neq i'} (-1)^{\ol{i}}
    \frac{q^{-(\varepsilon_{i},\varepsilon_{i})}q^{-c}\kappa z - q^{(\varepsilon_{i},\varepsilon_{i})}\xi}
         {q^{-c}\kappa z-\xi} |1,i',i,1\rangle \\
  & + \delta_{i=j} \delta_{i=i'} (-1)^{\ol{i}} \bigg( 1 +
    \frac{(q-q^{-1})(1-\xi)q^{-c}\kappa z}
         {(q^{-c}\kappa z-1)(q^{-c}\kappa z-\xi)} \bigg) |1,i,i,1\rangle \\
  & + \delta_{i<j} \frac{(q-q^{-1})q^{-c}\kappa z}{q^{-c}\kappa z-1} |1,i,j,1\rangle
    + \delta_{i>j} \frac{q-q^{-1}}{q^{-c}\kappa z-1} |1,i,j,1\rangle \\
  & - \delta_{i = j'} \frac{(q-q^{-1})q^{-c}\kappa z}{q^{-c}\kappa z-\xi}
      \sum_{i < a < 1'} (-1)^{\ol{a}+\ol{i}} \vartheta_{a}\vartheta_{i} q^{(\rho,\varepsilon_{a}-\varepsilon_{i})}
        q^{-(\varepsilon_{a},\varepsilon_{a})/2}q^{(\varepsilon_{i},\varepsilon_{i})/2} |1,a',a,1\rangle \\
  & - \delta_{i = j'} \frac{(q-q^{-1})\xi}{q^{-c}\kappa z-\xi}
      \sum_{1 < a < i} (-1)^{\ol{a}+\ol{i}} \vartheta_{a}\vartheta_{i} q^{(\rho,\varepsilon_{a}-\varepsilon_{i})}
        q^{-(\varepsilon_{a},\varepsilon_{a})/2}q^{(\varepsilon_{i},\varepsilon_{i})/2} |1,a',a,1\rangle \\
  & - \delta_{i = j'} \frac{(q-q^{-1})q^{-c}\kappa z}{q^{-c}\kappa z-\xi}
      (-1)^{\ol{1}+\ol{i}} \vartheta_{1'}\vartheta_{i} q^{(\rho,\varepsilon_{1'}-\varepsilon_{i})}
        q^{-(\varepsilon_{1},\varepsilon_{1})/2}q^{(\varepsilon_{i},\varepsilon_{i})/2} |1,1,1',1\rangle \\
  & - \delta_{i = j'} \frac{(q-q^{-1})\xi}{q^{-c}\kappa z-\xi}
      (-1)^{\ol{1}+\ol{i}} \vartheta_{1}\vartheta_{i} q^{(\rho,\varepsilon_{1}-\varepsilon_{i})}
        q^{-(\varepsilon_{1},\varepsilon_{1})/2}q^{(\varepsilon_{i},\varepsilon_{i})/2} |1,1',1,1\rangle \Bigg\} \,.
\end{split}
\end{equation}}

\smallskip
\noindent
\emph{Step 3.}
Again, before taking $\lim_{\kappa \to \kappa_{0}}$, let us collect the four ``boundary terms''
in~\eqref{eq:RHS-term-i1j1}, \eqref{eq:RHS-term-1i1j}, and~\eqref{eq:RHS-term-1ij1} involving $|1,1\rangle$ in either
the first two or the last two tensor factors (i.e.\ involving $|1,1,\ast,\ast\rangle$ or $|\ast,\ast,1,1\rangle$).
In the linear combination of~\eqref{eq:RHS-term-i1j1}--\eqref{eq:RHS-term-1ij1} inside the braces of~\eqref{eq:RHS-step1}
(i.e.\ before multiplying by the overall prefactor
  $(-1)^{\ol{1}} \frac{q^{(\varepsilon_{1},\varepsilon_{1})}q^{-c}\kappa^{-1}z - q^{-(\varepsilon_{1},\varepsilon_{1})}}
   {q^{-c}\kappa^{-1}z - 1}$),
the combined contribution of these four terms is given by
\begin{equation*}
\begin{split}
  \Sigma_{\mathrm{R}}(\kappa) &\coloneqq
  \frac{q-q^{-1}}{q^{-c}z-1} \sT_{\mathrm{R}}(\kappa) |i,j,1,1\rangle
  + \frac{(q-q^{-1})q^{-c}z}{q^{-c}z-1} \sT_{\mathrm{R}}(\kappa) |1,1,i,j\rangle \\
  &\quad - \delta_{i = j'} (-1)^{\ol{1}+\ol{i}} \sT_{\mathrm{R}}(\kappa) \bigg\{
    \frac{(q-q^{-1})q^{-c}\kappa z}{q^{-c}\kappa z-\xi}
    \vartheta_{1'}\vartheta_{i} q^{(\rho,\varepsilon_{1'}-\varepsilon_{i})}
    q^{-(\varepsilon_{1},\varepsilon_{1})/2}q^{(\varepsilon_{i},\varepsilon_{i})/2} |1,1,1',1\rangle \\
  &\hspace{120pt}
    + \frac{(q-q^{-1})\xi}{q^{-c}\kappa z-\xi}
    \vartheta_{1}\vartheta_{i} q^{(\rho,\varepsilon_{1}-\varepsilon_{i})}
    q^{-(\varepsilon_{1},\varepsilon_{1})/2}q^{(\varepsilon_{i},\varepsilon_{i})/2} |1,1',1,1\rangle
  \bigg\} \,.
\end{split}
\end{equation*}
Recall that $\sT_{\mathrm{R}}(\kappa) = \sT_{\mathrm{R}}'(\kappa) \hat{\CR}_{12}(\kappa) \hat{\CR}_{34}(\kappa)$, so in each term
of $\Sigma_{\mathrm{R}}(\kappa)$, one of the factors $\hat{\CR}_{12}(\kappa), \hat{\CR}_{34}(\kappa)$ acts on $|1,1\rangle$,
while the other acts on $|i,j\rangle$, $|1',1\rangle$, or $|1,1'\rangle$. We claim that
$\lim_{\kappa \to \kappa_{0}} \Sigma_{\mathrm{R}}(\kappa) = 0$:
\begin{enumerate}[label=(\roman*), leftmargin=0.6cm]

\item
If $\kappa_{0} \neq \xi$, then $\hat{\CR}(\kappa)$ is regular at $\kappa = \kappa_{0}$ (since $\kappa_{0} \neq 1$).
Thus, in each term of $\Sigma_{\mathrm{R}}(\kappa)$, the copy of $\hat{\CR}(\kappa)$ acting on $|1,1\rangle$ vanishes as
$\kappa \to \kappa_{0}$ by~\eqref{eq:quasidet-11-lim} while the other copy of $\hat{\CR}(\kappa)$ remains regular, so
$\lim_{\kappa \to \kappa_{0}} \Sigma_{\mathrm{R}}(\kappa) = 0$.

\item
If $\kappa_{0} = \xi$, then $\hat{\CR}(\kappa)$ has a simple pole at $\kappa = \kappa_{0}$. By~\eqref{eq:calR-hat-explicit},
$\hat{\CR}(\kappa)|i,j\rangle$ is regular at $\kappa = \kappa_{0}$ for $i \neq j'$ (so the first two terms of
$\Sigma_{\mathrm{R}}(\kappa)$ vanish as $\kappa \to \kappa_{0}$ unless $i = j'$), while on $|r,r'\rangle$ ($1 \leq r \leq 1'$)
we have
\begin{equation*}
  \lim_{\kappa \to \kappa_{0}}  (\kappa - \kappa_{0}) \hat{\CR}(\kappa) |r,r'\rangle =
  -(q-q^{-1})\kappa_{0} \sum_{1 \leq a \leq 1'} (-1)^{\ol{a}+\ol{r}} \vartheta_{a}\vartheta_{r}
    q^{(\rho,\varepsilon_{a}-\varepsilon_{r})}q^{-(\varepsilon_{a},\varepsilon_{a})/2}q^{(\varepsilon_{r},\varepsilon_{r})/2}
    |a',a\rangle \,,
\end{equation*}
where the $a \neq r$ summands come from the last two sums in~\eqref{eq:calR-hat-explicit} and the $a = r$ summand comes
from the $E_{r'r} \otimes E_{rr'}$ if $r\ne r'$ or $E_{rr} \otimes E_{rr}$ if $r = r'$ terms of~\eqref{eq:calR-hat-explicit}.
Combined with~\eqref{eq:quasidet-11}, which gives
\begin{equation*}
  \lim_{\kappa \to \kappa_{0}} (\kappa - \kappa_{0})^{-1} \hat{\CR}(\kappa)|1,1\rangle
  = \frac{(-1)^{\ol{1}}q^{(\varepsilon_{1},\varepsilon_{1})}}{\kappa_{0}-1}|1,1\rangle
  = -\frac{1}{(q-q^{-1})\kappa_{0}}|1,1\rangle \,,
\end{equation*}
we obtain for any $1 \leq r \leq 1'$:
\begin{align*}
  \lim_{\kappa \to \kappa_{0}} \hat{\CR}_{12}(\kappa)\hat{\CR}_{34}(\kappa) |r,r',1,1\rangle
    &= \sum_{1 \leq a \leq 1'} (-1)^{\ol{a}+\ol{r}} \vartheta_{a}\vartheta_{r} q^{(\rho,\varepsilon_{a}-\varepsilon_{r})}
    q^{-(\varepsilon_{a},\varepsilon_{a})/2}q^{(\varepsilon_{r},\varepsilon_{r})/2} |a',a,1,1\rangle \,, \\
  \lim_{\kappa \to \kappa_{0}} \hat{\CR}_{12}(\kappa)\hat{\CR}_{34}(\kappa) |1,1,r,r'\rangle
    &= \sum_{1 \leq a \leq 1'} (-1)^{\ol{a}+\ol{r}} \vartheta_{a}\vartheta_{r} q^{(\rho,\varepsilon_{a}-\varepsilon_{r})}
    q^{-(\varepsilon_{a},\varepsilon_{a})/2}q^{(\varepsilon_{r},\varepsilon_{r})/2} |1,1,a',a\rangle \,.
\end{align*}
Substituting these limits (for $r=i$, $r=1'$, and $r=1$) into $\Sigma_{\mathrm{R}}(\kappa)$ and using $\kappa_{0} = \xi$ gives
\begin{align*}
  \lim_{\kappa \to \kappa_{0}} & \Sigma_{\mathrm{R}}(\kappa) \\
  =\, & \delta_{i=j'} \frac{q-q^{-1}}{q^{-c}z-1} \sum_{1 \leq a \leq 1'} (-1)^{\ol{a}+\ol{i}} \vartheta_{a}\vartheta_{i}
    q^{(\rho,\varepsilon_{a}-\varepsilon_{i})}q^{-(\varepsilon_{a},\varepsilon_{a})/2}q^{(\varepsilon_{i},\varepsilon_{i})/2}
    \sT_{\mathrm{R}}'(\kappa_{0}) |a',a,1,1\rangle \\
  & + \delta_{i=j'} \frac{(q-q^{-1})q^{-c}z}{q^{-c}z-1} \sum_{1 \leq a \leq 1'} (-1)^{\ol{a}+\ol{i}} \vartheta_{a}\vartheta_{i}
    q^{(\rho,\varepsilon_{a}-\varepsilon_{i})}q^{-(\varepsilon_{a},\varepsilon_{a})/2}q^{(\varepsilon_{i},\varepsilon_{i})/2}
    \sT_{\mathrm{R}}'(\kappa_{0}) |1,1,a',a\rangle \\
  & - \delta_{i=j'} (-1)^{\ol{1}+\ol{i}} \frac{(q-q^{-1})q^{-c}z}{q^{-c}z-1} \vartheta_{1'}\vartheta_{i}
    q^{(\rho,\varepsilon_{1'}-\varepsilon_{i})}q^{-(\varepsilon_{1},\varepsilon_{1})/2}q^{(\varepsilon_{i},\varepsilon_{i})/2} \\
  &\hspace{75pt} \times
      \sum_{1 \leq a \leq 1'} (-1)^{\ol{a}+\ol{1}} \vartheta_{a}\vartheta_{1'} q^{(\rho,\varepsilon_{a}-\varepsilon_{1'})}
      q^{-(\varepsilon_{a},\varepsilon_{a})/2}q^{(\varepsilon_{1},\varepsilon_{1})/2} \sT_{\mathrm{R}}'(\kappa_{0}) |1,1,a',a\rangle \\
  & - \delta_{i=j'} (-1)^{\ol{1}+\ol{i}} \frac{q-q^{-1}}{q^{-c}z-1} \vartheta_{1}\vartheta_{i}
    q^{(\rho,\varepsilon_{1}-\varepsilon_{i})}q^{-(\varepsilon_{1},\varepsilon_{1})/2}q^{(\varepsilon_{i},\varepsilon_{i})/2} \\
  &\hspace{75pt} \times
    \sum_{1 \leq a \leq 1'} (-1)^{\ol{a}+\ol{1}} \vartheta_{a}\vartheta_{1} q^{(\rho,\varepsilon_{a}-\varepsilon_{1})}
    q^{-(\varepsilon_{a},\varepsilon_{a})/2}q^{(\varepsilon_{1},\varepsilon_{1})/2} \sT_{\mathrm{R}}'(\kappa_{0}) |a',a,1,1\rangle \\
  =\, & 0 \,.
\end{align*}
\end{enumerate}

\smallskip
\noindent
\emph{Step 4.}
Since
  $\sT_{\mathrm{R}}(\kappa) = \langle k,1,1,l| \big(L^{-}_{1}(v)L^{-}_{2}(\kappa v) \hat{\CR}_{12}(\kappa)\big)
   \big(L^{+}_{3}(u)L^{+}_{4}(\kappa u) \hat{\CR}_{34}(\kappa)\big)$,
applying~\eqref{eq:twisted-4-matrix-element} (with $\chi_{ab,cd}$ as in~\eqref{eq:chi-sign-def}) to the remaining terms
of~\eqref{eq:RHS-term-i1j1}--\eqref{eq:RHS-term-1ij1} and evaluating the $\kappa \to \kappa_{0}$ limits of the resulting factors
via~\eqref{eq:quasidet-1i1j}--\eqref{eq:quasidet-i1j1} (and~\eqref{eq:rescaled-braided-RLL-1}), we obtain:
\begin{align*}
  \lim_{\kappa \to \kappa_{0}} \sT_{\mathrm{R}}(\kappa) \hat{\CR}_{23}(q^{-c}\kappa z) |i,1,j,1\rangle
  &\equiv \frac{(q-q^{-1})q^{-c}\kappa_{0} z}{q^{-c}\kappa_{0} z-1}
   \chi_{kl,ij} \Big(-(-1)^{\ol{1}(\ol{1}+\ol{k})} q^{(\varepsilon_{1},\varepsilon_{1})}\Big)
   \wt{\gamma}^{-}_{ki}(v) \wt{\gamma}^{+}_{lj}(u) \,, \\
  \lim_{\kappa \to \kappa_{0}} \sT_{\mathrm{R}}(\kappa) \hat{\CR}_{23}(q^{-c}\kappa z) |1,i,1,j\rangle
  &\equiv \frac{q-q^{-1}}{q^{-c}\kappa_{0} z-1} \chi_{kl,ij} (-1)^{\ol{1}(\ol{k}+\ol{i})}
   \Big(-(-1)^{\ol{1}(\ol{1}+\ol{j})} q^{-(\varepsilon_{1},\varepsilon_{1})}\Big) \wt{\gamma}^{-}_{ki}(v) \wt{\gamma}^{+}_{lj}(u) \,, \\
  \lim_{\kappa \to \kappa_{0}} \sT_{\mathrm{R}}(\kappa) \hat{\CR}_{23}(q^{-c}\kappa z) |i,1,1,j\rangle
  &= (-1)^{\ol{1}} \frac{q^{(\varepsilon_{1},\varepsilon_{1})} q^{-c}\kappa_{0} z - q^{-(\varepsilon_{1},\varepsilon_{1})}}
    {q^{-c}\kappa_{0} z-1} \chi_{kl,ij} (-1)^{\ol{1}(\ol{k}+\ol{j})} \wt{\gamma}^{-}_{ki}(v) \wt{\gamma}^{+}_{lj}(u) \,,
\end{align*}
and finally
{\allowdisplaybreaks
\begin{align*}
  & \lim_{\kappa \to \kappa_{0}}  \sT_{\mathrm{R}}(\kappa) \hat{\CR}_{23}(q^{-c}\kappa z) |1,i,j,1\rangle \\
  &\equiv \delta_{i \neq j,j'} (-1)^{\ol{i}\,\ol{j}}
    \chi_{kl,ji} (-1)^{\ol{1}(\ol{k}+\ol{j})} \wt{\gamma}^{-}_{kj}(v) \wt{\gamma}^{+}_{li}(u) \\
  &\quad + \delta_{i=j} \delta_{i \neq i'} (-1)^{\ol{i}}
    \frac{q^{(\varepsilon_{i},\varepsilon_{i})}q^{-c}\kappa_{0} z - q^{-(\varepsilon_{i},\varepsilon_{i})}}{q^{-c}\kappa_{0} z-1}
    \chi_{kl,ij} (-1)^{\ol{1}(\ol{k}+\ol{i})} \wt{\gamma}^{-}_{ki}(v) \wt{\gamma}^{+}_{li}(u) \\
  &\quad + \delta_{i=j'} \delta_{i \neq i'} (-1)^{\ol{i}}
    \frac{q^{-(\varepsilon_{i},\varepsilon_{i})}q^{-c}\kappa_{0} z - q^{(\varepsilon_{i},\varepsilon_{i})}\xi}{q^{-c}\kappa_{0} z-\xi}
    \chi_{kl,ij} (-1)^{\ol{1}(\ol{k}+\ol{i})} \wt{\gamma}^{-}_{ki'}(v) \wt{\gamma}^{+}_{li}(u) \\
  &\quad + \delta_{i=j} \delta_{i=i'} (-1)^{\ol{i}}
    \bigg( 1+ \frac{(q-q^{-1})(1-\xi)q^{-c}\kappa_{0} z}{(q^{-c}\kappa_{0} z-1)(q^{-c}\kappa_{0} z-\xi)} \bigg)
    \chi_{kl,ij} (-1)^{\ol{1}(\ol{k}+\ol{i})} \wt{\gamma}^{-}_{ki}(v) \wt{\gamma}^{+}_{li}(u) \\
  &\quad + \delta_{i<j} \frac{(q-q^{-1})q^{-c}\kappa_{0} z}{q^{-c}\kappa_{0} z-1}
    \chi_{kl,ij} (-1)^{\ol{1}(\ol{k}+\ol{i})} \wt{\gamma}^{-}_{ki}(v) \wt{\gamma}^{+}_{lj}(u) \\
  &\quad + \delta_{i>j} \frac{q-q^{-1}}{q^{-c}\kappa_{0} z-1}
    \chi_{kl,ij} (-1)^{\ol{1}(\ol{k}+\ol{i})} \wt{\gamma}^{-}_{ki}(v) \wt{\gamma}^{+}_{lj}(u) \\
  &\quad - \delta_{i = j'} \frac{(q-q^{-1})q^{-c}\kappa_{0} z}{q^{-c}\kappa_{0} z-\xi}
    \sum_{i < a < 1'} (-1)^{\ol{a}+\ol{i}} \vartheta_{a}\vartheta_{i} q^{(\rho,\varepsilon_{a}-\varepsilon_{i})}
    q^{-(\varepsilon_{a},\varepsilon_{a})/2}q^{(\varepsilon_{i},\varepsilon_{i})/2}
    \chi_{kl,a'a} (-1)^{\ol{1}(\ol{k}+\ol{a})} \wt{\gamma}^{-}_{ka'}(v) \wt{\gamma}^{+}_{la}(u) \\
  &\quad - \delta_{i = j'} \frac{(q-q^{-1})\xi}{q^{-c}\kappa_{0} z-\xi}
    \sum_{1 < a < i} (-1)^{\ol{a}+\ol{i}} \vartheta_{a}\vartheta_{i} q^{(\rho,\varepsilon_{a}-\varepsilon_{i})}
    q^{-(\varepsilon_{a},\varepsilon_{a})/2}q^{(\varepsilon_{i},\varepsilon_{i})/2}
    \chi_{kl,a'a} (-1)^{\ol{1}(\ol{k}+\ol{a})} \wt{\gamma}^{-}_{ka'}(v) \wt{\gamma}^{+}_{la}(u) \,,
\end{align*}}
where $\equiv$ denotes equality modulo the ``boundary terms'' already absorbed into $\Sigma_{\mathrm{R}}(\kappa)$ from Step~3.

\smallskip
\noindent
\emph{Step 5.}
Substituting the first three limits from Step~4 into the linear combination of $|i,1,j,1\rangle$, $|1,i,1,j\rangle$,
and $|i,1,1,j\rangle$ in~\eqref{eq:RHS-step1} yields
\begin{align*}
  \lim_{\kappa \to \kappa_{0}} & \sT_{\mathrm{R}}(\kappa) \hat{\CR}_{23}(q^{-c}\kappa z) \bigg\{
    (-1)^{\ol{1}\,\ol{j}} \frac{q-q^{-1}}{q^{-c}z-1} |i,1,j,1\rangle
    + (-1)^{\ol{1}\,\ol{i}} \frac{(q-q^{-1})q^{-c}z}{q^{-c}z-1} |1,i,1,j\rangle \\
  &\hspace{230pt}
    + \frac{(q-q^{-1})^{2}q^{-c}z}{(q^{-c}z-1)^{2}} |i,1,1,j\rangle \bigg\} \\
  &\equiv -\frac{(q-q^{-1})^{2}q^{-c}z}{(q^{-c}z-1)(q^{-c}\kappa_{0} z-1)}
  \chi_{kl,ij} (-1)^{\ol{1}(\ol{1}+\ol{k}+\ol{j})} \cdot q^{-(\varepsilon_{1},\varepsilon_{1})}
  \wt{\gamma}^{-}_{ki}(v) \wt{\gamma}^{+}_{lj}(u) \,.
\end{align*}
Adding this to
  $(-1)^{\ol{1}(\ol{i}+\ol{j})} \lim_{\kappa \to \kappa_{0}} \sT_{\mathrm{R}}(\kappa) \hat{\CR}_{23}(q^{-c}\kappa z) |1,i,j,1\rangle$
from Step~4  and multiplying by the overall prefactor
   $(-1)^{\ol{1}} \cdot
   \frac{q^{(\varepsilon_{1},\varepsilon_{1})}q^{-c}\kappa^{-1}z - q^{-(\varepsilon_{1},\varepsilon_{1})}}{q^{-c}\kappa^{-1}z - 1}$
(while also using $\lim_{\kappa \to \kappa_{0}} \Sigma_{\mathrm{R}}(\kappa) = 0$ from Step~3), we obtain:
{\allowdisplaybreaks
\begin{align*}
  &(-1)^{\ol{1}}
  \left(\frac{q^{(\varepsilon_{1},\varepsilon_{1})}q^{-c}\kappa_{0}^{-1}z - q^{-(\varepsilon_{1},\varepsilon_{1})}}
             {q^{-c}\kappa_{0}^{-1}z - 1}\right)^{-1}
  \lim_{\kappa \to \kappa_{0}} \langle k,1,1,l |
  L^{-}_{1}(v)L^{-}_{2}(\kappa v) L^{+}_{3}(u)L^{+}_{4}(\kappa u)
  \hat{\CR}_{12}(\kappa) \hat{\CR}_{34}(\kappa) \\
  &\hspace{175pt} \times
  \hat{\CR}_{23}(q^{-c}\kappa z) \hat{\CR}_{34}(q^{-c}z) \hat{\CR}_{12}(q^{-c}z) \hat{\CR}_{23}(q^{-c}\kappa^{-1} z)
  | i,1,1,j \rangle \\
  &= \delta_{i \neq j,j'} (-1)^{\ol{k}\,\ol{l} + \ol{j}\,\ol{l}} \wt{\gamma}^{-}_{kj}(v) \wt{\gamma}^{+}_{li}(u) \\
  &\quad + \delta_{i=j} \delta_{i \neq i'} (-1)^{\ol{i}}
    \frac{q^{(\varepsilon_{i},\varepsilon_{i})}q^{-c}z - q^{-(\varepsilon_{i},\varepsilon_{i})}}{q^{-c}z-1}
    (-1)^{\ol{k}\,\ol{l} + \ol{i} + \ol{i}\,\ol{l}} \wt{\gamma}^{-}_{ki}(v) \wt{\gamma}^{+}_{li}(u) \\
  &\quad + \delta_{i=j'} \delta_{i \neq i'} (-1)^{\ol{i}}
    \frac{q^{-(\varepsilon_{i},\varepsilon_{i})}q^{-c}\kappa_{0} z - q^{(\varepsilon_{i},\varepsilon_{i})}\xi}{q^{-c}\kappa_{0} z-\xi}
    (-1)^{\ol{k}\,\ol{l} + \ol{i} + \ol{i}\,\ol{l}} \wt{\gamma}^{-}_{ki'}(v) \wt{\gamma}^{+}_{li}(u) \\
  &\quad + \delta_{i=j} \delta_{i=i'} (-1)^{\ol{i}} \bigg(
      1+ \frac{(q-q^{-1})(\kappa_{0}-\xi)q^{-c}z}{(q^{-c}z-1)(q^{-c}\kappa_{0} z-\xi)}
    \bigg)
    (-1)^{\ol{k}\,\ol{l} + \ol{i} + \ol{i}\,\ol{l}} \wt{\gamma}^{-}_{ki}(v) \wt{\gamma}^{+}_{li}(u) \\
  &\quad + \delta_{i<j} \frac{(q-q^{-1})q^{-c}z}{q^{-c}z-1}
    (-1)^{\ol{k}\,\ol{l} + \ol{i}\,\ol{j} + \ol{i}\,\ol{l}} \wt{\gamma}^{-}_{ki}(v) \wt{\gamma}^{+}_{lj}(u)
    + \delta_{i>j} \frac{q-q^{-1}}{q^{-c}z-1}
    (-1)^{\ol{k}\,\ol{l} + \ol{i}\,\ol{j} + \ol{i}\,\ol{l}} \wt{\gamma}^{-}_{ki}(v) \wt{\gamma}^{+}_{lj}(u) \\
  &\quad - \delta_{i = j'} \frac{(q-q^{-1})q^{-c}\kappa_{0} z}{q^{-c}\kappa_{0} z-\xi}
    \sum_{i < a < 1'} (-1)^{\ol{a}+\ol{i}} \vartheta_{a}\vartheta_{i} q^{(\rho,\varepsilon_{a}-\varepsilon_{i})}
    q^{-(\varepsilon_{a},\varepsilon_{a})/2}q^{(\varepsilon_{i},\varepsilon_{i})/2}
    (-1)^{\ol{k}\,\ol{l} + \ol{a} + \ol{a}\,\ol{l}} \wt{\gamma}^{-}_{ka'}(v) \wt{\gamma}^{+}_{la}(u) \\
  &\quad - \delta_{i = j'} \frac{(q-q^{-1})\xi}{q^{-c}\kappa_{0} z-\xi}
    \sum_{1 < a < i} (-1)^{\ol{a}+\ol{i}} \vartheta_{a}\vartheta_{i} q^{(\rho,\varepsilon_{a}-\varepsilon_{i})}
    q^{-(\varepsilon_{a},\varepsilon_{a})/2}q^{(\varepsilon_{i},\varepsilon_{i})/2}
    (-1)^{\ol{k}\,\ol{l} + \ol{a} + \ol{a}\,\ol{l}} \wt{\gamma}^{-}_{ka'}(v) \wt{\gamma}^{+}_{la}(u) \\
  &= \langle k,l | \wt{\Gamma}^{-}_{1}(v)\wt{\Gamma}^{+}_{2}(u) \hat{\CR}^{[s-1]}_{12}(q^{-c}z) | i,j \rangle \,,
\end{align*}}
with the last equality due to~\eqref{eq:calR-hat-explicit},~\eqref{eq:Gamma-def}, and~\eqref{eq:parameter-shifts}
(we substitute $\xi = \kappa_{0}\xi^{[s-1]}$ and cancel~$\kappa_{0}$).


\subsection{Comparison of both sides and commutation with $l^{\pm}_{11}$}\label{ssec:proof-comparison}
\

\smallskip
Equating the $\kappa \to \kappa_{0}$ limits of the left-hand and right-hand sides of~\eqref{eq:fused-RLL-pm} obtained above yields
\begin{equation}\label{eq:Gamma-RLL-pm}
\begin{split}
  &\phi(q^{c}\kappa_{0} z)^{-1} \phi(q^{c}\kappa_{0}^{-1} z)^{-1} \phi(q^{c}z)^{-2} \cdot
  \frac{q^{(\varepsilon_{1},\varepsilon_{1})}q^{c}\kappa_{0}^{-1}z - q^{-(\varepsilon_{1},\varepsilon_{1})}}
       {q^{c}\kappa_{0}^{-1}z - 1} \\
  &\hspace{150pt} \times
   \langle k,l | \hat{\CR}^{[s-1]}_{12}(q^{c}z) \wt{\Gamma}^{+}_{1}(u)\wt{\Gamma}^{-}_{2}(v) | i,j \rangle \\
  &= \phi(q^{-c}\kappa_{0} z)^{-1} \phi(q^{-c}\kappa_{0}^{-1} z)^{-1} \phi(q^{-c}z)^{-2} \cdot
  \frac{q^{(\varepsilon_{1},\varepsilon_{1})}q^{-c}\kappa_{0}^{-1}z - q^{-(\varepsilon_{1},\varepsilon_{1})}}
       {q^{-c}\kappa_{0}^{-1}z - 1} \\
  &\hspace{150pt} \times
    \langle k,l | \wt{\Gamma}^{-}_{1}(v)\wt{\Gamma}^{+}_{2}(u) \hat{\CR}^{[s-1]}_{12}(q^{-c}z) | i,j \rangle \,.
\end{split}
\end{equation}
To factor out $l^{\pm}_{11}$-terms from $\wt{\Gamma}^{\pm}$ (see~\eqref{eq:gamma-ij-def}--\eqref{eq:Gamma-def}), consider the matrix identity
\begin{equation}\label{eq:aux-check-1}
\begin{split}
  &\check{R}_{12}(q^{c}z) \check{R}_{23}(q^{c}\kappa z) \check{R}_{12}(\kappa)
    L^{+}_{1}(\kappa u)L^{+}_{2}(u) L^{-}_{3}(v) \\
  &\quad = L^{-}_{1}(v) L^{+}_{2}(u)L^{+}_{3}(\kappa u)
    \check{R}_{23}(\kappa) \check{R}_{12}(q^{-c}\kappa z) \check{R}_{23}(q^{-c}z) \,,
\end{split}
\end{equation}
which follows from~\eqref{eq:braid-rel} and~\eqref{eq:braided-RLL-1}--\eqref{eq:braided-RLL-2}, see Figure~\ref{fig:aux-check-1}
for a graphical proof (akin to that of~\eqref{eq:fused-RLL-check}):

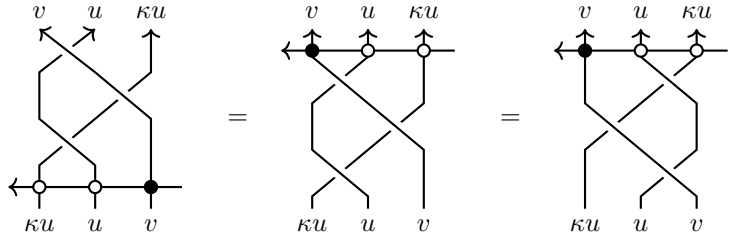
\begin{figure}[H]
\centering
\begin{tikzpicture}[baseline={(0,1.4)}, scale=0.82, every node/.style={font=\small}]
  \begin{scope}[shift={(0,0)}]
    \draw[thick, ->] (0,0) node[below] {$\kappa u$} -- (0,0.7) -- (0.9,1.45) -- (1.8,2.2) -- (1.8,2.9) node[above] {$\kappa u$};
    \draw[thick, preaction={-, draw=white, line width=5pt}, ->] (0.9,0) node[below] {$u$} -- (0.9,0.7) -- (0,1.45) -- (0,2.2) -- (0.9,2.9) node[above] {$u$};
    \draw[thick, preaction={-, draw=white, line width=5pt}, ->] (1.8,0) node[below] {$v$} -- (1.8,1.45) -- (0.9,2.2) -- (0,2.9) node[above] {$v$};

    \draw[thick, ->] (2.3,0.35) -- (-0.5,0.35);
    \filldraw[fill=white, thick] (0,0.35) circle (2.8pt);
    \filldraw[fill=white, thick] (0.9,0.35) circle (2.8pt);
    \filldraw[fill=black, thick] (1.8,0.35) circle (2.8pt);
  \end{scope}

  \node at (3.2,1.45) {$=$};

  \begin{scope}[shift={(4.4,0)}]
    \draw[thick, ->] (0,0) node[below] {$\kappa u$} -- (0,0.2) -- (0.9,0.95) -- (1.8,1.7) -- (1.8,2.9) node[above] {$\kappa u$};
    \draw[thick, preaction={-, draw=white, line width=5pt}, ->] (0.9,0) node[below] {$u$} -- (0.9,0.2) -- (0,0.95) -- (0,1.7) -- (0.9,2.45) -- (0.9,2.9) node[above] {$u$};
    \draw[thick, preaction={-, draw=white, line width=5pt}, ->] (1.8,0) node[below] {$v$} -- (1.8,0.95) -- (0.9,1.7) -- (0,2.45) -- (0,2.9) node[above] {$v$};

    \draw[thick, ->] (2.3,2.55) -- (-0.5,2.55);
    \filldraw[fill=black, thick] (0,2.55) circle (2.8pt);
    \filldraw[fill=white, thick] (0.9,2.55) circle (2.8pt);
    \filldraw[fill=white, thick] (1.8,2.55) circle (2.8pt);
  \end{scope}

  \node at (7.6,1.45) {$=$};

  \begin{scope}[shift={(8.8,0)}]
    \draw[thick, ->] (0,0) node[below] {$\kappa u$} -- (0,0.95) -- (0.9,1.7) -- (1.8,2.45) -- (1.8,2.9) node[above] {$\kappa u$};
    \draw[thick, preaction={-, draw=white, line width=5pt}, ->] (0.9,0) node[below] {$u$} -- (0.9,0.2) -- (1.8,0.95) -- (1.8,1.7) -- (0.9,2.45) -- (0.9,2.9) node[above] {$u$};
    \draw[thick, preaction={-, draw=white, line width=5pt}, ->] (1.8,0) node[below] {$v$} -- (1.8,0.2) -- (0.9,0.95) -- (0,1.7) -- (0,2.9) node[above] {$v$};

    \draw[thick, ->] (2.3,2.55) -- (-0.5,2.55);
    \filldraw[fill=black, thick] (0,2.55) circle (2.8pt);
    \filldraw[fill=white, thick] (0.9,2.55) circle (2.8pt);
    \filldraw[fill=white, thick] (1.8,2.55) circle (2.8pt);
  \end{scope}
\end{tikzpicture}
\caption{Graphical derivation of~\eqref{eq:aux-check-1}.}
\label{fig:aux-check-1}
\end{figure}
Rewriting $\check{R}$ in terms of $\hat{\CR}$ via~\eqref{eq:calR-hat-def} (and canceling the common $\phi(\kappa)^{-1}$),
and pairing both sides with $\langle 1,1,i|$ and $|j,1,1\rangle$, we obtain
\begin{equation}\label{eq:aux}
\begin{split}
  & \phi(q^{c}\kappa z)^{-1} \phi(q^{c}z)^{-1} \cdot
    \langle 1,1,i | \hat{\CR}_{12}(q^{c}z) \hat{\CR}_{23}(q^{c}\kappa z) \hat{\CR}_{12}(\kappa)
    L^{+}_{1}(\kappa u)L^{+}_{2}(u) L^{-}_{3}(v) | j,1,1 \rangle \\
  &= \phi(q^{-c}\kappa z)^{-1} \phi(q^{-c}z)^{-1} \cdot
    \langle 1,1,i | L^{-}_{1}(v) L^{+}_{2}(u)L^{+}_{3}(\kappa u)
    \hat{\CR}_{23}(\kappa) \hat{\CR}_{12}(q^{-c}\kappa z) \hat{\CR}_{23}(q^{-c}z) | j,1,1 \rangle \,.
\end{split}
\end{equation}
By~\eqref{eq:calR-hat-explicit}, the left-hand side of~\eqref{eq:aux} equals
\begin{equation*}
\begin{split}
  \phi(q^{c}\kappa z)^{-1} \phi(q^{c}z)^{-1} &\cdot
  (-1)^{\ol{1}} \cdot \frac{q^{(\varepsilon_{1},\varepsilon_{1})}q^{c}z - q^{-(\varepsilon_{1},\varepsilon_{1})}}{q^{c}z-1} \\
  &\times \bigg\{
    (-1)^{\ol{1}\,\ol{i}} \langle 1,i,1 | \hat{\CR}_{12}(\kappa) L^{+}_{1}(\kappa u)L^{+}_{2}(u) L^{-}_{3}(v) | j,1,1 \rangle \\
  &\hspace{20pt} + \frac{(q-q^{-1})q^{c}\kappa z}{q^{c}\kappa z - 1}
   \langle 1,1,i | \hat{\CR}_{12}(\kappa) L^{+}_{1}(\kappa u)L^{+}_{2}(u) L^{-}_{3}(v) | j,1,1 \rangle \bigg\} \,,
\end{split}
\end{equation*}
which in the limit $\kappa \to \kappa_{0}$ simplifies via~\eqref{eq:quasidet-1ij1} and~\eqref{eq:quasidet-11-lim} to
\begin{equation*}
  \phi(q^{c}\kappa_{0} z)^{-1} \phi(q^{c}z)^{-1} \cdot
  \frac{q^{(\varepsilon_{1},\varepsilon_{1})}q^{c}z - q^{-(\varepsilon_{1},\varepsilon_{1})}}{q^{c}z-1} \cdot
  \wt{\gamma}^{+}_{ij}(u) l^{-}_{11}(v) \,.
\end{equation*}
Similarly, the right-hand side of~\eqref{eq:aux} equals
\begin{equation*}
\begin{split}
  \phi(q^{-c}\kappa z)^{-1} \phi(q^{-c}z)^{-1} &\cdot
   (-1)^{\ol{1}} \cdot \frac{q^{(\varepsilon_{1},\varepsilon_{1})}q^{-c}z - q^{-(\varepsilon_{1},\varepsilon_{1})}}{q^{-c}z-1} \\
  &\times \bigg\{
    (-1)^{\ol{1}\,\ol{j}} \langle 1,1,i | L^{-}_{1}(v) L^{+}_{2}(u)L^{+}_{3}(\kappa u) \hat{\CR}_{23}(\kappa) | 1,j,1 \rangle \\
  &\hspace{20pt} + \frac{q-q^{-1}}{q^{-c}\kappa z - 1}
   \langle 1,1,i | L^{-}_{1}(v) L^{+}_{2}(u)L^{+}_{3}(\kappa u) \hat{\CR}_{23}(\kappa) | j,1,1 \rangle \bigg\} \,,
\end{split}
\end{equation*}
which in the limit $\kappa \to \kappa_{0}$ simplifies via~\eqref{eq:quasidet-i11j} and~\eqref{eq:quasidet-11-lim} to
\begin{equation*}
  \phi(q^{-c}\kappa_{0} z)^{-1} \phi(q^{-c}z)^{-1} \cdot
  \frac{q^{(\varepsilon_{1},\varepsilon_{1})}q^{-c}z - q^{-(\varepsilon_{1},\varepsilon_{1})}}{q^{-c}z-1} \cdot
  l^{-}_{11}(v) \wt{\gamma}^{+}_{ij}(u) \,.
\end{equation*}
Equating the two limits above and replacing $v$ by $\kappa_{0}v$ (so $z = u/v$ is replaced by $\kappa_{0}^{-1}z$) yields
\begin{equation}\label{eq:l11-gamma-comm-pm-1}
\begin{split}
  &\phi(q^{c}z)^{-1} \phi(q^{c}\kappa_{0}^{-1} z)^{-1} \cdot
  \frac{q^{(\varepsilon_{1},\varepsilon_{1})}q^{c}\kappa_{0}^{-1} z - q^{-(\varepsilon_{1},\varepsilon_{1})}}{q^{c}\kappa_{0}^{-1} z-1} \cdot
  \wt{\gamma}^{+}_{ij}(u) l^{-}_{11}(\kappa_{0} v) \\
  &= \phi(q^{-c}z)^{-1} \phi(q^{-c}\kappa_{0}^{-1} z)^{-1} \cdot
  \frac{q^{(\varepsilon_{1},\varepsilon_{1})}q^{-c}\kappa_{0}^{-1} z - q^{-(\varepsilon_{1},\varepsilon_{1})}}
       {q^{-c}\kappa_{0}^{-1} z-1} \cdot
  l^{-}_{11}(\kappa_{0} v) \wt{\gamma}^{+}_{ij}(u) \,.
\end{split}
\end{equation}

Analogously, consider the matrix identity
\begin{equation}\label{eq:aux-check-2}
\begin{split}
  &\check{R}_{23}(q^{c}\kappa^{-1} z) \check{R}_{12}(q^{c}z) \check{R}_{23}(\kappa)
    L^{+}_{1}(u) L^{-}_{2}(\kappa v)L^{-}_{3}(v) \\
  &\quad = L^{-}_{1}(v)L^{-}_{2}(\kappa v) L^{+}_{3}(u)
    \check{R}_{12}(\kappa) \check{R}_{23}(q^{-c}z) \check{R}_{12}(q^{-c}\kappa^{-1} z) \,,
\end{split}
\end{equation}
which follows from~\eqref{eq:braid-rel} and~\eqref{eq:braided-RLL-1}--\eqref{eq:braided-RLL-2}, see Figure~\ref{fig:aux-check-2}
for a graphical proof (cf.~\eqref{eq:fused-RLL-check} and~\eqref{eq:aux-check-1}):

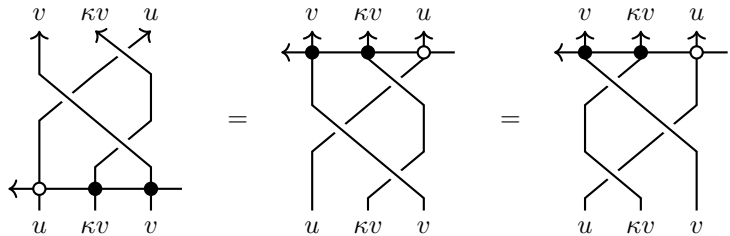
\begin{figure}[H]
\centering
\begin{tikzpicture}[baseline={(0,1.4)}, scale=0.82, every node/.style={font=\small}]
  \begin{scope}[shift={(0,0)}]
    \draw[thick, ->] (0,0) node[below] {$u$} -- (0,1.45) -- (0.9,2.2) -- (1.8,2.9) node[above] {$u$};
    \draw[thick, preaction={-, draw=white, line width=5pt}, ->] (0.9,0) node[below] {$\kappa v$} -- (0.9,0.7) -- (1.8,1.45) -- (1.8,2.2) -- (0.9,2.9) node[above] {$\kappa v$};
    \draw[thick, preaction={-, draw=white, line width=5pt}, ->] (1.8,0) node[below] {$v$} -- (1.8,0.7) -- (0.9,1.45) -- (0,2.2) -- (0,2.9) node[above] {$v$};

    \draw[thick, ->] (2.3,0.35) -- (-0.5,0.35);
    \filldraw[fill=white, thick] (0,0.35) circle (2.8pt);
    \filldraw[fill=black, thick] (0.9,0.35) circle (2.8pt);
    \filldraw[fill=black, thick] (1.8,0.35) circle (2.8pt);
  \end{scope}

  \node at (3.2,1.45) {$=$};

  \begin{scope}[shift={(4.4,0)}]
    \draw[thick, ->] (0,0) node[below] {$u$} -- (0,0.95) -- (0.9,1.7) -- (1.8,2.45) -- (1.8,2.9) node[above] {$u$};
    \draw[thick, preaction={-, draw=white, line width=5pt}, ->] (0.9,0) node[below] {$\kappa v$} -- (0.9,0.2) -- (1.8,0.95) -- (1.8,1.7) -- (0.9,2.45) -- (0.9,2.9) node[above] {$\kappa v$};
    \draw[thick, preaction={-, draw=white, line width=5pt}, ->] (1.8,0) node[below] {$v$} -- (1.8,0.2) -- (0.9,0.95) -- (0,1.7) -- (0,2.9) node[above] {$v$};

    \draw[thick, ->] (2.3,2.55) -- (-0.5,2.55);
    \filldraw[fill=black, thick] (0,2.55) circle (2.8pt);
    \filldraw[fill=black, thick] (0.9,2.55) circle (2.8pt);
    \filldraw[fill=white, thick] (1.8,2.55) circle (2.8pt);
  \end{scope}

  \node at (7.6,1.45) {$=$};

  \begin{scope}[shift={(8.8,0)}]
    \draw[thick, ->] (0,0) node[below] {$u$} -- (0,0.2) -- (0.9,0.95) -- (1.8,1.7) -- (1.8,2.9) node[above] {$u$};
    \draw[thick, preaction={-, draw=white, line width=5pt}, ->] (0.9,0) node[below] {$\kappa v$} -- (0.9,0.2) -- (0,0.95) -- (0,1.7) -- (0.9,2.45) -- (0.9,2.9) node[above] {$\kappa v$};
    \draw[thick, preaction={-, draw=white, line width=5pt}, ->] (1.8,0) node[below] {$v$} -- (1.8,0.95) -- (0.9,1.7) -- (0,2.45) -- (0,2.9) node[above] {$v$};

    \draw[thick, ->] (2.3,2.55) -- (-0.5,2.55);
    \filldraw[fill=black, thick] (0,2.55) circle (2.8pt);
    \filldraw[fill=black, thick] (0.9,2.55) circle (2.8pt);
    \filldraw[fill=white, thick] (1.8,2.55) circle (2.8pt);
  \end{scope}
\end{tikzpicture}
\caption{Graphical derivation of~\eqref{eq:aux-check-2}.}
\label{fig:aux-check-2}
\end{figure}
Rewriting $\check{R}$ in terms of $\hat{\CR}$ via~\eqref{eq:calR-hat-def} (and canceling the common $\phi(\kappa)^{-1}$),
and pairing both sides with $\langle i,1,1|$ and $|1,1,j\rangle$, we obtain
\begin{equation}\label{eq:aux-2-app}
\begin{split}
  &\phi(q^{c}\kappa^{-1} z)^{-1} \phi(q^{c}z)^{-1} \cdot
  \langle i,1,1 | \hat{\CR}_{23}(q^{c}\kappa^{-1} z) \hat{\CR}_{12}(q^{c}z) \hat{\CR}_{23}(\kappa)
  L^{+}_{1}(u) L^{-}_{2}(\kappa v)L^{-}_{3}(v) | 1,1,j \rangle \\
  &= \phi(q^{-c}\kappa^{-1} z)^{-1} \phi(q^{-c}z)^{-1} \cdot
  \langle i,1,1 | L^{-}_{1}(v)L^{-}_{2}(\kappa v) L^{+}_{3}(u)
  \hat{\CR}_{12}(\kappa) \hat{\CR}_{23}(q^{-c}z) \hat{\CR}_{12}(q^{-c}\kappa^{-1} z) | 1,1,j \rangle \,.
\end{split}
\end{equation}
By~\eqref{eq:calR-hat-explicit}, the left-hand side of~\eqref{eq:aux-2-app} equals
\begin{equation*}
\begin{split}
  \phi(q^{c}\kappa^{-1} z)^{-1} \phi(q^{c}z)^{-1} &\cdot
  (-1)^{\ol{1}} \cdot \frac{q^{(\varepsilon_{1},\varepsilon_{1})}q^{c}\kappa^{-1}z - q^{-(\varepsilon_{1},\varepsilon_{1})}}
    {q^{c}\kappa^{-1}z-1} \\
  &\times \bigg\{
    (-1)^{\ol{1}\,\ol{i}} \langle 1,i,1 | \hat{\CR}_{23}(\kappa) L^{+}_{1}(u) L^{-}_{2}(\kappa v)L^{-}_{3}(v) | 1,1,j \rangle \\
  &\hspace{20pt} + \frac{q-q^{-1}}{q^{c}z - 1}
   \langle i,1,1 | \hat{\CR}_{23}(\kappa) L^{+}_{1}(u) L^{-}_{2}(\kappa v)L^{-}_{3}(v) | 1,1,j \rangle \bigg\} \,,
\end{split}
\end{equation*}
which in the limit $\kappa \to \kappa_{0}$ simplifies via~\eqref{eq:quasidet-i11j} and~\eqref{eq:quasidet-11-lim} to
\begin{equation*}
  \phi(q^{c}\kappa_{0}^{-1} z)^{-1} \phi(q^{c}z)^{-1} \cdot
  \frac{q^{(\varepsilon_{1},\varepsilon_{1})}q^{c}\kappa_{0}^{-1}z - q^{-(\varepsilon_{1},\varepsilon_{1})}}
       {q^{c}\kappa_{0}^{-1}z-1} \cdot l^{+}_{11}(u) \wt{\gamma}^{-}_{ij}(v) \,.
\end{equation*}
Similarly, the right-hand side of~\eqref{eq:aux-2-app} equals
\begin{equation*}
\begin{split}
  \phi(q^{-c}\kappa^{-1} z)^{-1} \phi(q^{-c}z)^{-1} &\cdot
   (-1)^{\ol{1}} \cdot
   \frac{q^{(\varepsilon_{1},\varepsilon_{1})}q^{-c}\kappa^{-1}z - q^{-(\varepsilon_{1},\varepsilon_{1})}}{q^{-c}\kappa^{-1}z-1} \\
  &\times \bigg\{
    (-1)^{\ol{1}\,\ol{j}} \langle i,1,1 | L^{-}_{1}(v)L^{-}_{2}(\kappa v) L^{+}_{3}(u) \hat{\CR}_{12}(\kappa) | 1,j,1 \rangle \\
  &\hspace{20pt} + \frac{(q-q^{-1})q^{-c}z}{q^{-c}z - 1}
   \langle i,1,1 | L^{-}_{1}(v)L^{-}_{2}(\kappa v) L^{+}_{3}(u) \hat{\CR}_{12}(\kappa) | 1,1,j \rangle \bigg\} \,,
\end{split}
\end{equation*}
which in the limit $\kappa \to \kappa_{0}$ simplifies via~\eqref{eq:quasidet-i11j} and~\eqref{eq:quasidet-11-lim} to
\begin{equation*}
  \phi(q^{-c}\kappa_{0}^{-1} z)^{-1} \phi(q^{-c}z)^{-1} \cdot
  \frac{q^{(\varepsilon_{1},\varepsilon_{1})}q^{-c}\kappa_{0}^{-1}z - q^{-(\varepsilon_{1},\varepsilon_{1})}}
       {q^{-c}\kappa_{0}^{-1}z-1} \cdot \wt{\gamma}^{-}_{ij}(v) l^{+}_{11}(u) \,.
\end{equation*}
Equating the two limits above and replacing $u$ by $\kappa_{0}u$ (so $z = u/v$ is replaced by $\kappa_{0}z$) yields
\begin{equation}\label{eq:l11-gamma-comm-pm-2}
\begin{split}
  &\phi(q^{c}z)^{-1} \phi(q^{c}\kappa_{0} z)^{-1} \cdot
  \frac{q^{(\varepsilon_{1},\varepsilon_{1})}q^{c}z - q^{-(\varepsilon_{1},\varepsilon_{1})}}{q^{c}z-1} \cdot
  l^{+}_{11}(\kappa_{0} u) \wt{\gamma}^{-}_{ij}(v) \\
  &= \phi(q^{-c}z)^{-1} \phi(q^{-c}\kappa_{0} z)^{-1} \cdot
  \frac{q^{(\varepsilon_{1},\varepsilon_{1})}q^{-c}z - q^{-(\varepsilon_{1},\varepsilon_{1})}}{q^{-c}z-1} \cdot
  \wt{\gamma}^{-}_{ij}(v) l^{+}_{11}(\kappa_{0} u) \,.
\end{split}
\end{equation}

Finally, comparing the $E_{11} \otimes E_{11}$ components of both sides of~\eqref{eq:rescaled-braided-RLL-2}, we obtain
\begin{equation}\label{eq:l11-l11-comm-pm}
\begin{split}
  &\phi(q^{c}z)^{-1} \cdot
  \frac{q^{(\varepsilon_{1},\varepsilon_{1})}q^{c}z - q^{-(\varepsilon_{1},\varepsilon_{1})}}{q^{c}z - 1}
  l^{+}_{11}(u) l^{-}_{11}(v) \\
  &= \phi(q^{-c}z)^{-1} \cdot
  \frac{q^{(\varepsilon_{1},\varepsilon_{1})}q^{-c}z - q^{-(\varepsilon_{1},\varepsilon_{1})}}{q^{-c}z - 1}
  l^{-}_{11}(v) l^{+}_{11}(u) \,.
\end{split}
\end{equation}

Since $l^{\pm}_{11}(u)$ are invertible, combining~\eqref{eq:l11-gamma-comm-pm-1}--\eqref{eq:l11-l11-comm-pm} with
$\wt{\gamma}^{\pm}_{ij}(u) = l^{\pm}_{11}(\kappa_{0}u)\wt{s}^{\pm}_{ij}(u)$ from~\eqref{eq:gamma-ij-def}, we deduce
the commutation relations between $\wt{s}^{\pm}_{ij}(u)$ and $l^{\mp}_{11}(v)$:
\begin{equation}\label{eq:s-l11-comm-pm-1}
  \begin{split}
    &\phi(q^{c}\kappa_{0}^{-1} z)^{-1} \cdot
    \frac{q^{(\varepsilon_{1},\varepsilon_{1})}q^{c}\kappa_{0}^{-1} z - q^{-(\varepsilon_{1},\varepsilon_{1})}}
         {q^{c}\kappa_{0}^{-1} z-1} \cdot
    \frac{q^{c}z - 1}{q^{(\varepsilon_{1},\varepsilon_{1})}q^{c}z - q^{-(\varepsilon_{1},\varepsilon_{1})}} \cdot
    \wt{s}^{+}_{ij}(u) l^{-}_{11}(\kappa_{0} v) \\
    &\quad = \phi(q^{-c}\kappa_{0}^{-1} z)^{-1} \cdot
    \frac{q^{(\varepsilon_{1},\varepsilon_{1})}q^{-c}\kappa_{0}^{-1} z - q^{-(\varepsilon_{1},\varepsilon_{1})}}
         {q^{-c}\kappa_{0}^{-1} z-1} \cdot
    \frac{q^{-c}z - 1}{q^{(\varepsilon_{1},\varepsilon_{1})}q^{-c}z - q^{-(\varepsilon_{1},\varepsilon_{1})}} \cdot
    l^{-}_{11}(\kappa_{0} v) \wt{s}^{+}_{ij}(u)
  \end{split}
\end{equation}
and
\begin{equation}\label{eq:s-l11-comm-pm-2}
  \phi(q^{c}\kappa_{0} z)^{-1} \, l^{+}_{11}(\kappa_{0} u) \wt{s}^{-}_{ij}(v)
  = \phi(q^{-c}\kappa_{0} z)^{-1} \, \wt{s}^{-}_{ij}(v) l^{+}_{11}(\kappa_{0} u) \,.
\end{equation}
Using~\eqref{eq:l11-l11-comm-pm}--\eqref{eq:s-l11-comm-pm-2} to pull $l^{+}_{11}(\kappa_{0}u)$ and $l^{-}_{11}(\kappa_{0}v)$ out
to the left on both sides of~\eqref{eq:Gamma-RLL-pm} and then multiplying by $(l^{+}_{11}(\kappa_{0}u) l^{-}_{11}(\kappa_{0}v))^{-1}$
on the left, we obtain
\begin{equation*}
  \phi(q^{c}z)^{-1} \langle k,l | \hat{\CR}^{[s-1]}_{12}(q^{c}z) \wt{S}^{+}_{1}(u)\wt{S}^{-}_{2}(v) | i,j \rangle
  = \phi(q^{-c}z)^{-1} \langle k,l | \wt{S}^{-}_{1}(v)\wt{S}^{+}_{2}(u) \hat{\CR}^{[s-1]}_{12}(q^{-c}z) | i,j \rangle \,,
\end{equation*}
which is equivalent to~\eqref{eq:S-RLL-2}.

\medskip
The same-sign relation~\eqref{eq:S-RLL-1} is obtained via a line-by-line analogue of the above argument
(dropping all $\phi$-factors and setting $q^{\pm c} \mapsto 1$). To this end, we replace the equality~\eqref{eq:fused-RLL-pm} with
\begin{equation*}
\begin{split}
  &\langle k,1,1,l |
  \hat{\CR}_{23}(\kappa^{-1} z) \hat{\CR}_{34}(z) \hat{\CR}_{12}(z) \hat{\CR}_{23}(\kappa z) \hat{\CR}_{12}(\kappa) \hat{\CR}_{34}(\kappa)
  L^{\pm}_{1}(\kappa u)L^{\pm}_{2}(u) L^{\pm}_{3}(\kappa v)L^{\pm}_{4}(v)
  | i,1,1,j \rangle \\
  &= \langle k,1,1,l |
  L^{\pm}_{1}(v)L^{\pm}_{2}(\kappa v) L^{\pm}_{3}(u)L^{\pm}_{4}(\kappa u)
  \hat{\CR}_{12}(\kappa) \hat{\CR}_{34}(\kappa) \hat{\CR}_{23}(\kappa z) \hat{\CR}_{34}(z) \hat{\CR}_{12}(z) \hat{\CR}_{23}(\kappa^{-1} z)
  | i,1,1,j \rangle \,.
\end{split}
\end{equation*}
Then~\eqref{eq:S-RLL-1} follows by taking the $\kappa \to \kappa_{0}$ limits of both sides, while utilizing
\begin{equation*}
\begin{split}
  &\langle 1,1,i | \hat{\CR}_{12}(z) \hat{\CR}_{23}(\kappa z) \hat{\CR}_{12}(\kappa)
  L^{\pm}_{1}(\kappa u)L^{\pm}_{2}(u) L^{\pm}_{3}(v) | j,1,1 \rangle \\
  &= \langle 1,1,i | L^{\pm}_{1}(v) L^{\pm}_{2}(u)L^{\pm}_{3}(\kappa u)
  \hat{\CR}_{23}(\kappa) \hat{\CR}_{12}(\kappa z) \hat{\CR}_{23}(z) | j,1,1 \rangle
\end{split}
\end{equation*}
instead of~\eqref{eq:aux}
together with $l^{\pm}_{11}(u) l^{\pm}_{11}(v) = l^{\pm}_{11}(v) l^{\pm}_{11}(u)$ obtained by comparing the $E_{11} \otimes E_{11}$
components in~\eqref{eq:rescaled-braided-RLL-1}, so that~\eqref{eq:s-l11-comm-pm-1}--\eqref{eq:s-l11-comm-pm-2} is replaced by
$\wt{s}^{\pm}_{ij}(u) l^{\pm}_{11}(v) = l^{\pm}_{11}(v) \wt{s}^{\pm}_{ij}(u) \ \forall\, 1 < i,j < 1'$.

\medskip
This completes the proof of Theorem~\ref{thm:rank-reduction}.


\subsection{Rational counterpart}\label{ssec:yangian}
\

Retaining the setup and notations of Section~\ref{ssec:notations-setups}, the \emph{rational orthosymplectic $R$-matrix}
$R_{Y}(x) \in \BC(x) \otimes_{\BC} \End(V)^{\otimes 2}$ is given by the following explicit formula (see~\cite[(3.4)]{ft},
while the standard parity case appeared first in~\cite{aacfr}):
\begin{equation*}
  R_{Y}(x) \coloneqq \ID - \frac{1}{x} \sum_{i,j=1}^{N} (-1)^{\ol{j}} E_{ij} \otimes E_{ji}
  + \frac{1}{x - \varpi} \sum_{i,j=1}^{N} (-1)^{\ol{i}\,\ol{j}} \vartheta_{i}\vartheta_{j} E_{ij} \otimes E_{i'j'} \,,
\end{equation*}
where we set $\varpi \coloneqq \frac{m - n - 2}{2}$ in accordance with~\cite[(3.3)]{ft}.
The operator $R_{Y}(x)$ satisfies the parameter-dependent (additive) \emph{quantum Yang--Baxter equation}:
\begin{equation*}
  R_{Y,12}(x_1 - x_2) R_{Y,13}(x_1 - x_3) R_{Y,23}(x_2 - x_3)
  = R_{Y,23}(x_2 - x_3) R_{Y,13}(x_1 - x_3) R_{Y,12}(x_1 - x_2) \,.
\end{equation*}

Consider the free $\BC$-superalgebra
$\CT_{Y} \coloneqq \BC\big\langle t_{ij}^{(r)} \,\big|\, 1 \leq i,j \leq N\,, r \geq 1 \big\rangle$ equipped with the
$\BZ_{2}$-grading $\big|t_{ij}^{(r)}\big| = \ol{i} + \ol{j} = |E_{ij}|$. We organize these generators into generating series
\begin{equation*}
  t_{ij}(u) \coloneqq \delta_{ij} + \sum_{r \geq 1} t_{ij}^{(r)} u^{-r} \in \CT_{Y}[\![u^{-1}]\!]
\end{equation*}
and the matrix
$$
  T(u) \coloneqq \sum_{i,j=1}^{N} t_{ij}(u) \otimes E_{ij} \in \CT_{Y}[\![u^{-1}]\!] \otimes_{\BC} \End(V) \,.
$$

\begin{Def}\label{def:Y-R}
The \emph{extended Yangian of $\fosp(V)$} is the $\BC$-superalgebra $Y(R_{Y}(x))$ defined as the quotient of $\CT_{Y}$
by all matrix coefficients in $\CT_{Y}$ of the following matrix relation:
\begin{equation}\label{eq:Y-rel}
  R_{Y,12}(u - v) T_{1}(u) T_{2}(v) = T_{2}(v) T_{1}(u) R_{Y,12}(u - v) \,.
\end{equation}
\end{Def}

For $1 < i,j < 1'$, we define the \emph{quasideterminant} $s_{ij}(u) = \delta_{ij} + \sum_{r \geq 1} s_{ij}^{(r)} u^{-r}$ via
(cf.~\eqref{eq:quasidet-untwisted})
\begin{equation}\label{eq:Y-quasidet}
  s_{ij}(u) \coloneqq t_{ij}(u) - (-1)^{(\ol{1}+\ol{i})(\ol{1}+\ol{j})}
    t_{i1}(u) (t_{11}(u))^{-1} t_{1j}(u) \quad\in\quad Y(R_{Y}(x))[\![u^{-1}]\!] \,,
\end{equation}
and organize these series into the $(N-2) \times (N-2)$ matrix $S(u) \coloneqq \sum_{1<i,j<1'} s_{ij}(u) \otimes E_{ij}$.

\begin{Rem}\label{rem:FT-sign-comparison}
In~\cite[\S3.1]{ft}, the generating matrix is defined with an additional sign factor in each $(i,j)$-entry, corresponding to
$T(u) = \sum_{i,j=1}^{N} (-1)^{\ol{j}(\ol{i}+\ol{j})} \sft_{ij}(u) \otimes E_{ij}$, which amounts to the rescaling
$t_{ij}(u) = (-1)^{\ol{j}(\ol{i}+\ol{j})} \sft_{ij}(u)$. Under this rescaling, the
quasideterminant~\eqref{eq:Y-quasidet} takes the form $s_{ij}(u) = (-1)^{\ol{j}(\ol{i}+\ol{j})} \sfs_{ij}(u)$, where
$\sfs_{ij}(u) \coloneqq \sft_{ij}(u) - \sft_{i1}(u) (\sft_{11}(u))^{-1} \sft_{1j}(u)$ is the sign-free quasideterminant
(cf.~\cite[(3.5)]{m} for the standard parity case) and $S(u) = \sum_{1<i,j<1'} (-1)^{\ol{j}(\ol{i}+\ol{j})} \sfs_{ij}(u) \otimes E_{ij}$.
\end{Rem}

Following Section~\ref{ssec:rank-reduction-hom}, we use the superscripts $[s]$ and $[s-1]$ to denote objects associated with
$V^{[s]} = V$ and $V^{[s-1]} = \Span_{\BC}\, \{v_{i} \,|\, 1 < i < 1'\}$, respectively. In particular,
$\varpi^{[s-1]} = \varpi^{[s]} - (\varepsilon_{1},\varepsilon_{1})$, cf.~\cite[(3.42)]{ft}. The following provides a new proof
of~\cite[Theorem~3.47]{ft} (cf.~\cite[Theorem~3.1]{jlm1} and~\cite[Theorem~3.1]{m} for the non-super and standard parity cases):

\begin{Thm}\label{thm:Y-rank-reduction}
Assume $s \geq 2$. In $Y(R_{Y}^{[s]}(x))$, the matrix $S(u)$ satisfies the defining relation of $Y(R_{Y}^{[s-1]}(x))$:
\begin{equation}\label{eq:Y-S-rel}
  R_{Y,12}^{[s-1]}(u - v) S_{1}(u) S_{2}(v) = S_{2}(v) S_{1}(u) R_{Y,12}^{[s-1]}(u - v) \,.
\end{equation}
Therefore, the assignment $T^{[s-1]}(u) \mapsto S(u)$ gives rise to a superalgebra homomorphism
\begin{equation}\label{eq:rank-reduction-Yangian}
  Y(R_{Y}^{[s-1]}(x)) \to Y(R_{Y}^{[s]}(x)) \,.
\end{equation}
\end{Thm}

\medskip

\begin{Rem}\label{rem:generalized-rank-reduction}
To deduce Theorem~\ref{thm:Y-rank-reduction} from Theorem~\ref{thm:rank-reduction}, we note that the proof of the
same-sign relation~\eqref{eq:S-RLL-1} in Theorem~\ref{thm:rank-reduction} extends as~follows:
\begin{enumerate}[label=(\alph*), leftmargin=0.6cm]

\item
Let $\CY(\ol{R}(z))$ be the $\BC(q)$-superalgebra generated by the coefficients of a single matrix
$T(u) = \sum_{i,j=1}^{N} t_{ij}(u) \otimes E_{ij}$ with $t_{ij}(u) = \sum_{r \geq 0} t_{ij}^{(r)} u^{-r}$,
subject to $t_{ij}^{(0)} = \delta_{ij}$ and the defining relation
\begin{equation}\label{eq:CY-rel}
  \ol{R}_{12}(u/v) T_{1}(u) T_{2}(v) = T_{2}(v) T_{1}(u) \ol{R}_{12}(u/v) \,.
\end{equation}
The exact same proof of the same-sign relation~\eqref{eq:S-RLL-1} shows that the quasideterminant matrix
$S(u) = \sum_{1 < i,j < 1'} s_{ij}(u) \otimes E_{ij}$ with
$s_{ij}(u) \coloneqq t_{ij}(u) - (-1)^{(\ol{1}+\ol{i})(\ol{1}+\ol{j})} t_{i1}(u) t_{11}(u)^{-1} t_{1j}(u)$
satisfies~\eqref{eq:CY-rel} with $\ol{R}^{[s]}(z)$ replaced by $\ol{R}^{[s-1]}(z)$, yielding a rank-reduction homomorphism
\begin{equation}\label{eq:rank-recution-T}
  \CY(\ol{R}^{[s-1]}(z)) \to \CY(\ol{R}^{[s]}(z)) \,.
\end{equation}

\smallskip

\item
Consider the $\BC[q, q^{-1}]$-subalgebra $\CO \subset \BC(q)(z)$ given by
\begin{equation*}
  \CO \coloneqq \BC\left[q, q^{-1}, z, z^{-1}, \frac{q - q^{-1}}{z - q^{k}}, \frac{z - q^{l}}{z - q^{k}}
  \,\middle|\, k, l \in \BZ \right] \,,
\end{equation*}
which contains the matrix entries of $\ol{R}(z), \hat{\CR}(z)$ and all coefficients appearing in the
$\kappa \to \kappa_{0}$ limit process in Sections~\ref{ssec:quasideterminant}, \ref{ssec:proof-LHS}--\ref{ssec:proof-comparison}
(see~\eqref{eq:kappa0-fractions}), so that the derivation of~\eqref{eq:rank-recution-T} holds over~$\CO$.

The specialization $q \mapsto e^{\hbar/2}$, $z \mapsto e^{-\hbar x}$ defines a $\BC$-algebra homomorphism
\begin{equation}\label{eq:O-to-hbar-specialization}
  \CO \to \BC\left[x, (x - a)^{-1} \,\middle|\, a \in \sfrac{1}{2}\BZ \right][\![\hbar]\!] \subset \BC(x)[\![\hbar]\!] \,.
\end{equation}
Denoting the image of $\ol{R}(z)$ under~\eqref{eq:O-to-hbar-specialization} by
\begin{equation}\label{eq:R-hbar-def}
  \ol{R}_{\hbar}(x) \coloneqq \ol{R}(e^{-\hbar x})\big|_{q = e^{\hbar/2}} \,,
\end{equation}
let $\CA(\ol{R}_{\hbar}(x))$ be the $\BC[\![\hbar]\!]$-superalgebra generated by coefficients of $T(u)$
subject to $t_{ij}^{(0)} = \delta_{ij}$~and
\begin{equation}\label{eq:A-hbar-rel}
  \ol{R}_{\hbar,12}(u - v) T_{1}(u) T_{2}(v) = T_{2}(v) T_{1}(u) \ol{R}_{\hbar,12}(u - v) \,.
\end{equation}
By base change along~\eqref{eq:O-to-hbar-specialization}, part~(a) immediately yields the rank-reduction homomorphism
\begin{equation}\label{eq:rank-recution-A}
  \CA(\ol{R}_{\hbar}^{[s-1]}(x)) \to \CA(\ol{R}_{\hbar}^{[s]}(x)) \quad \mathrm{given\ by} \quad T^{[s-1]}(u) \mapsto S(u) \,.
\end{equation}

\end{enumerate}
\end{Rem}

\medskip

\begin{proof}[Proof of Theorem~\ref{thm:Y-rank-reduction}]
Under the specialization~\eqref{eq:O-to-hbar-specialization}--\eqref{eq:R-hbar-def}, whereas $\xi = q^{-m+n+2} = e^{-\hbar \varpi}$,
the $\hbar \to 0$ limit of $\ol{R}_{\hbar}(x)$ yields (cf.~\cite[Remark~6.5]{ht1}):
\begin{equation}\label{eq:Y-trig-to-rat-limit}
  \lim_{\hbar \to 0} \ol{R}_{\hbar}(x) = \frac{x}{x - 1} R_{Y}(x) \,.
\end{equation}
Consider the base change along the canonical quotient map $\BC[\![\hbar]\!] \to \BC[\![\hbar]\!]/(\hbar) \simeq \BC$.
By~\eqref{eq:Y-trig-to-rat-limit}, the defining relation~\eqref{eq:A-hbar-rel} of the $\BC[\![\hbar]\!]$-superalgebra
$\CA(\ol{R}_{\hbar}^{[s]}(x))$ specializes in the $\hbar \to 0$ limit
to~\eqref{eq:Y-rel} yielding a canonical $\BC$-superalgebra isomorphism
\begin{equation*}
  \CA(\ol{R}_{\hbar}^{[s]}(x)) \otimes_{\BC[\![\hbar]\!]} \BC \iso Y(R_{Y}^{[s]}(x)) \,.
\end{equation*}
Thus, the $\BC[\![\hbar]\!]$-superalgebra homomorphism $\CA(\ol{R}_{\hbar}^{[s-1]}(x)) \to \CA(\ol{R}_{\hbar}^{[s]}(x))$
from~\eqref{eq:rank-recution-A}, after the base change $- \otimes_{\BC[\![\hbar]\!]} \BC$ immediately yields the desired
homomorphism~\eqref{eq:rank-reduction-Yangian}.
\end{proof}


\medskip

\appendix
\section{Scalar prefactor in the universal $R$-matrix}\label{sec:prefactor}

In this Appendix, we compute the scalar prefactor $f(z)$ in~\eqref{eq:R-univ-eval} relating the evaluated universal $R$-matrix
$R(z)$ of $\UdqV$ to the unitary $R$-matrix $\ol{R}(z)$ of~\eqref{eq:R-affine-BCD}, thus refining~\cite[Theorem~6.3]{ht1}.


\subsection{Quantum affine orthosymplectic supergroups and evaluation modules}\label{ssec:eval-modules}
\

We briefly recall the Drinfeld--Jimbo realization of the quantum affine orthosymplectic supergroup $\UdqV$ and its evaluation
representation on $V_{q}[u,u^{-1}]$, cf.~\cite[\S6.1--6.2]{ht1}.
Recall the positive root system $\Phi^{+}$ of $\fosp(V)$ from Section~\ref{ssec:notations-setups}. The simple roots
$\alpha_{1}, \ldots, \alpha_{s} \in \Phi^{+}$ are given by $\alpha_{i} \coloneqq \varepsilon_{i} - \varepsilon_{i+1}$
for $1 \leq i < s$, together with
\begin{equation*}
  \alpha_{s} \coloneqq
  \begin{cases}
    \varepsilon_{s} & \text{if $N$ is odd ($B$ type)} \,, \\
    2\varepsilon_{s} & \text{if $N$ is even and } \ol{s} = \bbar{1} \text{ ($C$ type)} \,, \\
    \varepsilon_{s-1} + \varepsilon_{s} & \text{if $N$ is even and } \ol{s} = \bbar{0} \text{ ($D$ type)} \,.
  \end{cases}
\end{equation*}
Let $\theta \in \Phi^{+}$ be the highest root of $\fosp(V)$, that is (recalling that $\varepsilon_{2} = 0$ when $N = 3$)
\begin{equation*}
  \theta =
  \begin{cases}
    2\varepsilon_{1} & \text{if } \ol{1} = \bbar{1} \,, \\
    \varepsilon_{1} + \varepsilon_{2} & \text{if } \ol{1} = \bbar{0} \,.
  \end{cases}
\end{equation*}
Denoting by $\delta$ the null root of the (untwisted) affine Lie superalgebra $\widehat{\fosp}(V)$, the affine simple root
is given by $\alpha_{0} \coloneqq \delta - \theta$. The symmetric bilinear form $(\cdot,\cdot)$ on $\fh^{*}$ is extended to
$\fh^{*} \oplus \BC\delta$ via
$$
  (\delta, \fh^{*}) = 0 = (\delta, \delta) \,.
$$

We recall the following property of the Weyl vector $\rho$ from~\eqref{eq:weyl-vector} (see~\cite[Corollary~1.33]{cw}):
\begin{equation}\label{eq:rho-simple-roots}
  (2\rho, \alpha_{i}) = (\alpha_{i}, \alpha_{i}) \qquad \text{for all} \quad 1 \leq i \leq s \,,
\end{equation}
while the \emph{dual Coxeter number}
$$
  h^{\vee} \coloneqq \sfrac{1}{2}(\theta, \theta + 2\rho)
$$
is given independently of the parity sequence $\gamma_{V}$ by
\begin{equation*}
  2h^{\vee} = (\theta, \theta + 2\rho) = (\alpha_{0}, \alpha_{0}) - (2\rho, \alpha_{0}) = 2(m - n - 2) \,.
\end{equation*}
In particular, recalling $\xi = q^{-m+n+2}$ from~\eqref{eq:R-affine-BCD}, we have
\begin{equation}\label{eq:q-coxeter-xi}
  q^{-h^{\vee}} = \xi = q^{-m+n+2} \,.
\end{equation}

The \emph{quantum affine orthosymplectic supergroup} $\UdqV$ is the $\BC(q)$-superalgebra generated by
$e_{i}, f_{i}, q^{\pm h_{i}}$ ($0 \leq i \leq s$) and $q^{\pm d}$, with $\BZ_{2}$-grading $|e_{i}| = |f_{i}| = |\alpha_{i}|$
and $|q^{\pm h_{i}}| = |q^{\pm d}| = \bbar{0}$, subject to the Chevalley-type, standard Serre, and higher-order Serre relations,
see~\cite[\S6.1]{ht1}. In particular, the Chevalley-type relations are given by
\begin{gather*}
  q^{\pm h_{i}} q^{\mp h_{i}} = 1 \,,\quad q^{\pm d} q^{\mp d} = 1  \,,\quad
  q^{h_{i}} q^{h_{j}} = q^{h_{j}} q^{h_{i}} \,,\quad q^{d} q^{h_{i}} = q^{h_{i}} q^{d} \,, \\
  q^{h_{i}} e_{j} q^{-h_{i}} = q^{(\alpha_{i},\alpha_{j})} e_{j} \,, \qquad
  q^{h_{i}} f_{j} q^{-h_{i}} = q^{-(\alpha_{i},\alpha_{j})} f_{j} \,, \qquad
  [e_{i}, f_{j}] = \delta_{ij} \frac{q^{h_{i}} - q^{-h_{i}}}{q - q^{-1}} \,, \\
  q^{d} e_{i} q^{-d} = q^{\delta_{i 0}} e_{i} \,, \qquad
  q^{d} f_{i} q^{-d} = q^{-\delta_{i 0}} f_{i}
\end{gather*}
for all $0 \leq i, j \leq s$, where $[e_{i}, f_{j}] \coloneqq e_{i}f_{j} - (-1)^{|e_{i}||f_{j}|} f_{j}e_{i}$ denotes the supercommutator.
Its Hopf superalgebra structure is given by
\begin{equation}\label{eq:DJ-Hopf}
\begin{split}
  & \Delta(e_{i}) = q^{h_{i}} \otimes e_{i} + e_{i} \otimes 1 \,, \qquad
  \Delta(f_{i}) = 1 \otimes f_{i} + f_{i} \otimes q^{-h_{i}} \,, \qquad
  \Delta(q^{\pm h_{i}}) = q^{\pm h_{i}} \otimes q^{\pm h_{i}} \,, \\
  & S(e_{i}) = -q^{-h_{i}} e_{i} \,, \qquad
  S(f_{i}) = -f_{i} q^{h_{i}} \,, \qquad
  S(q^{\pm h_{i}}) = q^{\mp h_{i}} \,, \\
  & \epsilon(e_{i}) = 0 \,, \qquad \epsilon(f_{i}) = 0 \,, \qquad \epsilon(q^{\pm h_{i}}) = 1 \,, \\
  & \Delta(q^{\pm d}) = q^{\pm d} \otimes q^{\pm d} \,, \qquad
  S(q^{\pm d}) = q^{\mp d} \,, \qquad
  \epsilon(q^{\pm d}) = 1 \,.
\end{split}
\end{equation}
We also consider the positive and negative Borel subalgebras $U_{q}^{\geq}(\widehat{\fosp}(V))$ and $U_{q}^{\leq}(\widehat{\fosp}(V))$,
generated by $\{e_{i}, q^{\pm h_{i}}\}_{i=0}^{s} \cup \{q^{\pm d}\}$ and $\{f_{i}, q^{\pm h_{i}}\}_{i=0}^{s} \cup \{q^{\pm d}\}$,
respectively, as well as $U_{q}^{>}(\widehat{\fosp}(V))$ and $U_{q}^{<}(\widehat{\fosp}(V))$ generated by $\{e_{i}\}_{i=0}^{s}$ and
$\{f_{i}\}_{i=0}^{s}$, respectively.
Let $U_{q}(\wt{\fosp}(V))$ denote the subalgebra generated by $\{e_{i}, f_{i}, q^{\pm h_{i}}\}_{i=0}^{s}$ (without $q^{\pm d}$),
and let $U_{q}^{\geq}(\wt{\fosp}(V))$ and $U_{q}^{\leq}(\wt{\fosp}(V))$ denote its Borel subalgebras generated by
$\{e_{i}, q^{\pm h_{i}}\}_{i=0}^{s}$ and $\{f_{i}, q^{\pm h_{i}}\}_{i=0}^{s}$, respectively. Moreover, the subalgebra generated by
$\{e_{i}, f_{i}, q^{\pm h_{i}}\}_{i=1}^{s}$ is a Hopf subalgebra of $\UdqV$, which coincides with the
\emph{quantum orthosymplectic supergroup} $\uqV$, see~\cite[\S2.4]{ht1} and~\cite[\S4.3]{ht2}.

For an indeterminate $u$, the \emph{evaluation representation} $\varrho_{u}$ of $\UdqV$ on
$V_{q}[u,u^{-1}] \coloneqq \BC(q)[u,u^{-1}] \otimes_{\BC(q)} V_{q}$ is defined on $U_{q}(\wt{\fosp}(V))$ via
$\varrho_{u} \colon U_{q}(\wt{\fosp}(V)) \to \End_{\BC(q)}(V_{q})[u,u^{-1}]$
(cf.~\cite[Propositions~3.1,~6.1]{ht1}, with $X_{ij}$~of~\eqref{eq:X-ij-def}):
\begin{align}
  \varrho_{u}(q^{\pm h_{i}}) &=
    \sum_{a=1}^{N} q^{\pm(\alpha_{i}, \varepsilon_{a})} E_{aa} \qquad (0 \leq i \leq s) \,, \label{eq:eval-cartan} \\
  \varrho_{u}(e_{i}) &= X_{i,i+1} \,, \qquad
    \varrho_{u}(f_{i}) = (-1)^{\ol{i}} X_{i+1,i} \qquad (1 \leq i < s) \,, \label{eq:eval-finite-i} \\
  \varrho_{u}(e_{s}) &=
  \begin{cases}
    X_{s,s+1} & \text{(type $B$)} \,, \\
    E_{ss'} & \text{(type $C$)} \,, \\
    X_{s-1,s'} & \text{(type $D$)} \,,
  \end{cases}
  \qquad
  \varrho_{u}(f_{s}) =
  \begin{cases}
    (-1)^{\ol{s}} X_{s+1,s} & \text{(type $B$)} \,, \\
    -(q+q^{-1}) E_{s's} & \text{(type $C$)} \,, \\
    (-1)^{\ol{s-1}} X_{s',s-1} & \text{(type $D$)} \,,
  \end{cases} \label{eq:eval-finite-s} \\
  \varrho_{u}(e_{0}) &=
  \begin{cases}
    u E_{1'1} & \text{if } \ol{1} = \bbar{1} \,, \\
    u X_{1'2} & \text{if } \ol{1} = \bbar{0} \,,
  \end{cases}
  \qquad
  \varrho_{u}(f_{0}) =
  \begin{cases}
    -(q+q^{-1}) u^{-1} E_{11'} & \text{if } \ol{1} = \bbar{1} \,, \\
    u^{-1} X_{21'} & \text{if } \ol{1} = \bbar{0} \,,
  \end{cases} \label{eq:eval-affine-0}
\end{align}
while $q^{\pm d}$ acts $\BC(q)$-linearly on $V_{q}[u,u^{-1}]$ via $\varrho_{u}(q^{\pm d})(u^{k} w) = q^{\pm k} u^{k} w$
for all $k \in \BZ$ and $w \in V_{q}$.


\subsection{Prefactor for $R$-matrix}\label{ssec:determine-prefactor}
\

Let $\fR_{\aff} \in U_{q}^{\geq}(\widehat{\fosp}(V)) \hat{\otimes} U_{q}^{\leq}(\widehat{\fosp}(V))$ denote the universal
$R$-matrix of $\UdqV$ (where $\hat{\otimes}$ denotes the completed tensor product), and set $z \coloneqq u/v$.
Recall that $\fR_{\aff}$ factors as
\begin{equation}\label{eq:R-aff-factorization}
  \fR_{\aff} = \fR_{\aff,\rmu} \fR_{\aff,\rms} \qquad \text{with} \qquad
  \fR_{\aff,\rmu} \in U_{q}^{>}(\widehat{\fosp}(V)) \hat{\otimes} U_{q}^{<}(\widehat{\fosp}(V)) \,,
\end{equation}
where the semisimple part $\fR_{\aff,\rms} = \fR_{\rms} q^{-c \otimes d - d \otimes c}$ is the product of the semisimple part
$\fR_{\rms}$ of the universal $R$-matrix $\fR = \fR_{\rmu} \fR_{\rms}$ of $\uqV$ and the factor $q^{-c \otimes d - d \otimes c}$, with
$c \coloneqq h_{0} + h_{\theta}$ being the canonical central element corresponding to the null root $\delta = \alpha_{0} + \theta$
(where $h_{\theta} \coloneqq \sum_{i=1}^{s} k_{i} h_{i}$ for $\theta = \sum_{i=1}^{s} k_{i} \alpha_{i}$). Setting
$\wt{\fR}_{\aff} \coloneqq \fR_{\aff,\rmu} \fR_{\rms} \in U_{q}^{\geq}(\wt{\fosp}(V)) \hat{\otimes} U_{q}^{\leq}(\wt{\fosp}(V))$,
we thus have $\fR_{\aff} = \wt{\fR}_{\aff} q^{-c \otimes d - d \otimes c}$.
By~\cite[Theorem~6.3]{ht1}, evaluating $\fR_{\aff}$ under $\varrho_{u} \otimes \varrho_{v}$ yields
\begin{equation}\label{eq:R-univ-eval}
  R(z) \coloneqq (\varrho_{u} \otimes \varrho_{v})\fR_{\aff} = f(z) \ol{R}(z) \,,
\end{equation}
where $\ol{R}(z)$ is the unitary $R$-matrix~\eqref{eq:R-affine-BCD} and $f(z) \in \BC(q)[\![z]\!]^{\times}$ is an
invertible scalar power~series. We shall now determine this prefactor $f(z)$ explicitly.

Consider the supertranspose map $\st \colon \End_{\BC(q)}(V_{q}) \to \End_{\BC(q)}(V_{q})$ defined on the elementary matrices by
\begin{equation*}
  (E_{ij})^{\st} \coloneqq (-1)^{\ol{j}(\ol{i}+\ol{j})} E_{ji} \qquad \forall\, 1 \leq i,j \leq N \,.
\end{equation*}
Since the supertranspose map $\st$ and the antipode $S$ are both superalgebra anti-automorphisms
(see Section~\ref{ssec:notations-setups}), the composition
\begin{equation*}
  \pi_{u} \coloneqq \st \circ \varrho_{u} \circ S \colon U_{q}(\wt{\fosp}(V)) \to \End_{\BC(q)}(V_{q})[u,u^{-1}]
\end{equation*}
is also a superalgebra homomorphism, which extends to a representation of $\UdqV$ on $V_{q}[u,u^{-1}]$ by setting
$\pi_{u}(q^{\pm d}) \coloneqq \varrho_{u}(q^{\pm d})$.
Let $\st_{1} \coloneqq \st \otimes \Id$ denote the partial supertranspose on the first tensor factor of
$\End_{\BC(q)}(V_{q})^{\otimes 2}$.

\begin{Rem}\label{rem:cd-trivial}
Since $(\delta, \varepsilon_{a}) = 0$ for all $1 \leq a \leq N$, it follows from~\eqref{eq:eval-cartan} that
$\varrho_{u}(q^{\pm c}) = \ID$, and hence also $\pi_{u}(q^{\pm c}) = \ID$. In the $\hbar$-adic setup (with $q = e^{\hbar}$),
this corresponds to $c \mapsto 0$ under both $\varrho_{u}$ and $\pi_{u}$. Expanding
$q^{y} = e^{\hbar y} = \sum_{k=0}^{\infty} \frac{\hbar^{k}}{k!} y^{k}$ for $y = -(c \otimes d + d \otimes c)$, we obtain
\begin{equation}\label{eq:R-aff-series}
  \fR_{\aff} = \wt{\fR}_{\aff} + \sum_{k=1}^{\infty} \frac{\hbar^{k}}{k!} \wt{\fR}_{\aff} y^{k} \,.
\end{equation}
Every $k \geq 1$ summand in~\eqref{eq:R-aff-series} involves $c$ and thus vanishes upon applying any
representation morphism sending $c \mapsto 0$ (i.e.\ $q^{\pm c} \mapsto \ID$). For example,
\begin{equation}\label{eq:cd-trivial-2}
  (\varrho_{u} \otimes \varrho_{v})\fR_{\aff} = (\varrho_{u} \otimes \varrho_{v})\wt{\fR}_{\aff} \,, \qquad
  (\pi_{u} \otimes \varrho_{v})\fR_{\aff} = (\pi_{u} \otimes \varrho_{v})\wt{\fR}_{\aff} \,.
\end{equation}
More generally, any linear map (continuous with respect to the $\hbar$-adic topology) acts term-by-term on the infinite
sum~\eqref{eq:R-aff-series}. For instance, applying $S \otimes \Id$ to~\eqref{eq:R-aff-series} gives
\begin{equation*}
  (S \otimes \Id)\fR_{\aff}
  = (S \otimes \Id)\wt{\fR}_{\aff} + \sum_{k=1}^{\infty} \frac{\hbar^{k}}{k!} (S \otimes \Id)\big(\wt{\fR}_{\aff} y^{k}\big) \,.
\end{equation*}
Since $S(c) = -c$, every $k \geq 1$ summand on the right-hand side still involves $c$ and thus vanishes under
$\varrho_{u} \otimes \varrho_{v}$ and $\pi_{u} \otimes \varrho_{v}$, yielding
\begin{equation}\label{eq:cd-trivial-S}
\begin{split}
  (\varrho_{u} \otimes \varrho_{v})(S \otimes \Id)\fR_{\aff} &= (\varrho_{u} \otimes \varrho_{v})(S \otimes \Id)\wt{\fR}_{\aff} \,, \\
  (\pi_{u} \otimes \varrho_{v})(S \otimes \Id)\fR_{\aff} &= (\pi_{u} \otimes \varrho_{v})(S \otimes \Id)\wt{\fR}_{\aff} \,.
\end{split}
\end{equation}
\end{Rem}

Using the identity $(S \otimes \Id)\fR_{\aff} = \fR_{\aff}^{-1}$ (see~\cite[Proposition~A.34]{ht2}) together with the fact that
$\pi_{u} \otimes \varrho_{v}$ is a superalgebra homomorphism, we obtain
\begin{equation}\label{eq:crossing-1}
  (\pi_{u} \otimes \varrho_{v})(S \otimes \Id)\fR_{\aff} = (\pi_{u} \otimes \varrho_{v})(\fR_{\aff}^{-1})
  = \big((\pi_{u} \otimes \varrho_{v})\fR_{\aff}\big)^{-1} \,.
\end{equation}
On the other hand, unfolding the definition of $\pi_{u}$ and using~\eqref{eq:cd-trivial-2}--\eqref{eq:cd-trivial-S} directly yields
\begin{equation}\label{eq:crossing-2}
\begin{split}
  (\pi_{u} \otimes \varrho_{v})\fR_{\aff} &= (\pi_{u} \otimes \varrho_{v})\wt{\fR}_{\aff}
  = \big((\varrho_{u} \otimes \varrho_{v})(S \otimes \Id)\wt{\fR}_{\aff}\big)^{\st_{1}} \\
  &\qquad = \big((\varrho_{u} \otimes \varrho_{v})(S \otimes \Id)\fR_{\aff}\big)^{\st_{1}}
   = \big((\varrho_{u} \otimes \varrho_{v}) \fR_{\aff}^{-1} \big)^{\st_{1}} = \big(R(z)^{-1}\big)^{\st_{1}}
\end{split}
\end{equation}
as well as
\begin{equation}\label{eq:crossing-3}
  (\pi_{u} \otimes \varrho_{v})(S \otimes \Id)\fR_{\aff}
  = (\pi_{u} \otimes \varrho_{v})(S \otimes \Id)\wt{\fR}_{\aff}
  = \big((\varrho_{u} \otimes \varrho_{v})(S^{2} \otimes \Id)\wt{\fR}_{\aff}\big)^{\st_{1}} \,.
\end{equation}
Combining~\eqref{eq:crossing-1}--\eqref{eq:crossing-3}, we therefore get
\begin{equation}\label{eq:crossing-combined}
  \big((\varrho_{u} \otimes \varrho_{v})(S^{2} \otimes \Id)\wt{\fR}_{\aff}\big)^{\st_{1}}
  = \Big(\big(R(z)^{-1}\big)^{\st_{1}}\Big)^{-1} \,.
\end{equation}

Since the first tensor factor of $\wt{\fR}_{\aff}$ lies in $U_{q}^{\geq}(\wt{\fosp}(V))$, to evaluate the left-hand side
of~\eqref{eq:crossing-combined} it suffices to compute the square of the antipode on $\{e_{i}, q^{\pm h_{i}}\}_{i=0}^{s}$.
From~\eqref{eq:DJ-Hopf}, we have $S^{2}(q^{\pm h_{i}}) = q^{\pm h_{i}}$
and $S^{2}(e_{i}) = q^{-h_{i}}e_{i}q^{h_{i}} = q^{-(\alpha_{i},\alpha_{i})} e_{i}$ for all $0 \leq i \leq s$.
Combining this with~\eqref{eq:rho-simple-roots}--\eqref{eq:q-coxeter-xi} gives
\begin{equation*}
  S^{2}(e_{i}) =
  \begin{cases}
    q^{-(2\rho,\alpha_{i})} e_{i} & \mathrm{if} \ \ 1 \leq i \leq s \,, \\
    q^{-(2\rho,\alpha_{0})} \xi^{2} e_{0} & \mathrm{if} \ \ i = 0 \,.
  \end{cases}
\end{equation*}
Setting $q^{\pm 2\sfh_{\rho}} \coloneqq \sum_{a=1}^{N} q^{\pm(2\rho,\varepsilon_{a})} E_{aa} \in \End_{\BC(q)}(V_{q})$,
it follows from~\eqref{eq:eval-cartan}--\eqref{eq:eval-affine-0} that
\begin{equation*}
  \varrho_{u}\big(S^{2}(y)\big) = q^{-2\sfh_{\rho}} \varrho_{\xi^{2}u}(y) q^{2\sfh_{\rho}}
  \qquad \text{for all} \quad y \in U_{q}^{\geq}(\wt{\fosp}(V)) \,,
\end{equation*}
and hence (cf.~\eqref{eq:cd-trivial-2})
\begin{equation*}
  (\varrho_{u} \otimes \varrho_{v})(S^{2} \otimes \Id)\wt{\fR}_{\aff}
  = (q^{-2\sfh_{\rho}} \otimes \ID) R(\xi^{2}z) (q^{2\sfh_{\rho}} \otimes \ID) \,.
\end{equation*}
Substituting this into~\eqref{eq:crossing-combined} finally yields the following relation:
\begin{equation}\label{eq:crossing-master}
  \Big((q^{-2\sfh_{\rho}} \otimes \ID) R(\xi^{2}z) (q^{2\sfh_{\rho}} \otimes \ID)\Big)^{\st_{1}}
  = \Big(\big(R(z)^{-1}\big)^{\st_{1}}\Big)^{-1} \,.
\end{equation}

One can now derive the difference equation on the prefactor $f(z)$ from~\eqref{eq:R-univ-eval} by evaluating both sides
of~\eqref{eq:crossing-master} on the vector $v_{1'} \otimes v_{1} \in V_{q} \otimes V_{q}$:
\begin{itemize}[leftmargin=0.7cm]

\item [$\bullet$]
\emph{Action of the left-hand side of~\eqref{eq:crossing-master} on $v_{1'} \otimes v_{1}$.}

By direct computation using the explicit formula~\eqref{eq:R-affine-BCD}, we find
\begin{equation}\label{eq:crossing-LHS-eigenval}
  \Big((q^{-2\sfh_{\rho}} \otimes \ID) R(\xi^{2}z) (q^{2\sfh_{\rho}} \otimes \ID)\Big)^{\st_{1}} (v_{1'} \otimes v_{1})
  = f(\xi^{2}z) \frac{\xi^{2}z-1}{q\xi^{2}z - q^{-1}} \cdot
  \frac{q^{-(\varepsilon_{1},\varepsilon_{1})}\xi z - q^{(\varepsilon_{1},\varepsilon_{1})}}{\xi z - 1} \cdot
  v_{1'} \otimes v_{1} \,.
\end{equation}

\item [$\bullet$]
\emph{Action of the right-hand side of~\eqref{eq:crossing-master} on $v_{1'} \otimes v_{1}$.}

By unitarity~\eqref{eq:R-affine-unitary} and~\eqref{eq:R-univ-eval}, we have $R(z)^{-1} = f(z)^{-1} \ol{R}_{21}(z^{-1})$.
By direct computation using~\eqref{eq:R-affine-BCD}, one sees that $v_{1'} \otimes v_{1}$ is an eigenvector of the operator
$\big(R(z)^{-1}\big)^{\st_{1}}$:
\begin{equation*}
  \big(R(z)^{-1}\big)^{\st_{1}} (v_{1'} \otimes v_{1})
  = f(z)^{-1} \frac{z^{-1}-1}{qz^{-1} - q^{-1}} \cdot
  \frac{q^{-(\varepsilon_{1},\varepsilon_{1})}z^{-1} - q^{(\varepsilon_{1},\varepsilon_{1})}\xi}{z^{-1} - \xi}
  \cdot v_{1'} \otimes v_{1} \,.
\end{equation*}
Inverting this scalar eigenvalue immediately gives the action of $\big((R(z)^{-1})^{\st_{1}}\big)^{-1}$ on $v_{1'} \otimes v_{1}$:
\begin{equation}\label{eq:crossing-RHS-eigenval}
  \Big(\big(R(z)^{-1}\big)^{\st_{1}}\Big)^{-1} (v_{1'} \otimes v_{1})
  = f(z) \frac{q^{-1}z - q}{z-1} \cdot
  \frac{\xi z - 1}{q^{(\varepsilon_{1},\varepsilon_{1})}\xi z - q^{-(\varepsilon_{1},\varepsilon_{1})}} \cdot
  v_{1'} \otimes v_{1} \,.
\end{equation}

\end{itemize}
Equating the eigenvalues in~\eqref{eq:crossing-LHS-eigenval} and~\eqref{eq:crossing-RHS-eigenval},
we arrive at the difference equation
\begin{equation}\label{eq:f-diff-eq}
  \frac{f(z)}{f(\xi^{2}z)}
  = \frac{(1 - z)(1 - q^{-2}\xi z)(1 - q^{2}\xi z)(1 - \xi^{2}z)}{(1 - q^{-2}z)(1 - \xi z)^{2}(1 - q^{2}\xi^{2}z)} \,.
\end{equation}

Consider the $\BZ$-grading on $\UdqV$ induced by the conjugation action of $q^{d}$, under which $\deg(e_{i}) = \delta_{i 0}$,
$\deg(f_{i}) = -\delta_{i 0}$, and $\deg(q^{\pm h_{i}}) = \deg(q^{\pm d}) = 0$ ($0 \leq i \leq s$), together with the induced
$\BZ \times \BZ$-grading on $\UdqV \otimes \UdqV$. The degree-$0$ subalgebras of $U_{q}^{>}(\widehat{\fosp}(V))$ and
$U_{q}^{<}(\widehat{\fosp}(V))$ are the subalgebras $U_{q}^{>}(\fosp(V))$ and $U_{q}^{<}(\fosp(V))$ of $\uqV$ generated by
$\{e_{i}\}_{i=1}^{s}$ and $\{f_{i}\}_{i=1}^{s}$, respectively, whereas the direct sums of the homogeneous components of degrees
$k \geq 1$ and $-k \leq -1$ equal the two-sided ideals of $U_{q}^{>}(\widehat{\fosp}(V))$ and $U_{q}^{<}(\widehat{\fosp}(V))$
generated by $e_{0}$ and $f_{0}$, respectively. Since every term of $\fR_{\aff,\rmu}$ in~\eqref{eq:R-aff-factorization} has
$\BZ \times \BZ$-degree $(k, -k)$ for some $k \geq 0$, it decomposes as
$\fR_{\aff,\rmu} = \fR_{\rmu} + \sum_{k \geq 1} \fR_{\aff,\rmu}^{(k)}$, where
$\fR_{\rmu} \in U_{q}^{>}(\fosp(V)) \hat{\otimes} U_{q}^{<}(\fosp(V))$ has $\BZ \times \BZ$-degree $(0,0)$ and
$\fR_{\aff,\rmu}^{(k)}$ has $\BZ \times \BZ$-degree $(k, -k)$. By~\eqref{eq:eval-cartan}--\eqref{eq:eval-affine-0},
$(\varrho_{u} \otimes \varrho_{v})(\fR_{\aff,\rmu}^{(k)}) \in z^{k} \End_{\BC(q)}(V_{q})^{\otimes 2}$ vanishes at $z = u/v = 0$
for all $k \geq 1$, while the restriction of $\varrho_{u}$ to $\uqV \subset \UdqV$ is independent of $u$ and coincides with
the first fundamental representation $\varrho \colon \uqV \to \End_{\BC(q)}(V_{q})$ on $V_{q}$ from~\cite[Proposition~3.1]{ht1}.
Consequently, in view of~\eqref{eq:cd-trivial-2}, $R(0) = (\varrho \otimes \varrho)\fR$ is the evaluated universal $R$-matrix
of $\uqV$ on $V_{q} \otimes V_{q}$, which is given explicitly in~\cite[Remark~4.7]{ht1}~by
\begin{equation*}
\begin{split}
  (\varrho \otimes \varrho)\fR &= \sum_{i,j=1}^{N} q^{-(\varepsilon_{i},\varepsilon_{j})} E_{ii} \otimes E_{jj}
  - (q - q^{-1}) \sum_{i < j} (-1)^{\ol{j}} E_{ij} \otimes E_{ji} \\
  &\hspace{20pt}
  + (q - q^{-1}) \sum_{i < j} (-1)^{\ol{i}\,\ol{j}} \vartheta_{i}\vartheta_{j} q^{(\rho,\varepsilon_{i}-\varepsilon_{j})}
    q^{-(\varepsilon_{i},\varepsilon_{i})/2}q^{(\varepsilon_{j},\varepsilon_{j})/2} E_{ij} \otimes E_{i'j'} \,.
\end{split}
\end{equation*}
On the other hand, setting $z = 0$ in the right-hand side of~\eqref{eq:R-affine-BCD} yields the exact same expression
$(\varrho \otimes \varrho)\fR = R(0)$, while the left-hand side of~\eqref{eq:R-affine-BCD} evaluates at $z = 0$ to $q^{-1}\ol{R}(0)$.
Thus $R(0) = q^{-1}\ol{R}(0)$, which implies the initial condition
\begin{equation}\label{eq:initial}
  f(0) = q^{-1} \,.
\end{equation}
Solving~(\ref{eq:f-diff-eq},~\ref{eq:initial}) in $\BC(q)[\![z]\!]$ provides the following refinement of~\cite[Theorem~6.3]{ht1}:

\begin{Prop}\label{prop:R-univ-explicit}
The evaluation $R(z) = (\varrho_{u} \otimes \varrho_{v})\fR_{\aff}$ of the universal $R$-matrix $\fR_{\aff}$ of $\UdqV$
on $V_{q}[u,u^{-1}] \otimes V_{q}[v,v^{-1}]$ (with $z = u/v$) is given explicitly by
\begin{equation*}
  R(z) = f(z) \ol{R}(z) \,,
\end{equation*}
where $\ol{R}(z)$ is the unitary $R$-matrix from~\eqref{eq:R-affine-BCD} and
\begin{equation}\label{eq:f-explicit}
  f(z) \coloneqq q^{-1} \prod_{k=0}^{\infty}
    \frac{(1 - \xi^{2k}z)(1 - q^{-2}\xi^{2k+1}z)(1 - q^{2}\xi^{2k+1}z)(1 - \xi^{2k+2}z)}
         {(1 - q^{-2}\xi^{2k}z)(1 - \xi^{2k+1}z)^{2}(1 - q^{2}\xi^{2k+2}z)} \,.
\end{equation}
Explicitly, we have
{\allowdisplaybreaks
\begin{equation*}
\begin{split}
  R(z)
  &= f_{0}(z) \Bigg\{ \sum_{i,j = 1}^{N} q^{(\varepsilon_{i}, \varepsilon_{j})} E_{ii} \otimes E_{jj}
  + (q-q^{-1}) \sum_{i>j} (-1)^{\ol{j}} E_{ij} \otimes E_{ji} \\
  &\hspace{50pt}
  - (q-q^{-1}) \sum_{i>j} (-1)^{\ol{i}\,\ol{j}} \vartheta_{i}\vartheta_{j} q^{(\rho,\varepsilon_{i}-\varepsilon_{j})}
    q^{-(\varepsilon_{i},\varepsilon_{i})/2}q^{(\varepsilon_{j},\varepsilon_{j})/2} E_{ij} \otimes E_{i'j'}  \\
  &\hspace{50pt} + \frac{q-q^{-1}}{z-1} \sum_{i,j = 1}^{N} (-1)^{\ol{j}} E_{ij} \otimes E_{ji} \\
  &\hspace{50pt}
    - \frac{(q-q^{-1})\xi}{z-\xi} \sum_{i,j = 1}^{N} (-1)^{\ol{i}\,\ol{j}} \vartheta_{i}\vartheta_{j}
    q^{(\rho,\varepsilon_{i}-\varepsilon_{j})}q^{-(\varepsilon_{i},\varepsilon_{i})/2}q^{(\varepsilon_{j},\varepsilon_{j})/2}
    E_{ij} \otimes E_{i'j'} \Bigg\} \,,
\end{split}
\end{equation*}}
where $f_{0}(z) \in 1 + z\BC(q)[\![z]\!]$ is given by
\begin{equation*}
  f_{0}(z) \coloneqq \frac{z-1}{zq-q^{-1}} f(z) = \prod_{k=0}^{\infty}
  \frac{(1 - \xi^{2k}z)^{2}(1 - q^{-2}\xi^{2k+1}z)(1 - q^{2}\xi^{2k+1}z)}
       {(1 - q^{-2}\xi^{2k}z)(1 - q^{2}\xi^{2k}z)(1 - \xi^{2k+1}z)^{2}} \,.
\end{equation*}
\end{Prop}

\smallskip

\begin{Rem}\label{rem:jlm-comparison}
We note that Proposition~\ref{prop:R-univ-explicit} is compatible with the corresponding results of~\cite{jlm2, jlm3}
for the non-super $BCD$-types (where our spectral parameter $z$ is denoted by $u$, and we denote their parameter $\xi$
and scalar series $f(u),g(u)$ by $\xi_{\mathrm{JLM}}$ and $f_{\mathrm{JLM}}(u),g_{\mathrm{JLM}}(u)$, respectively):
\begin{itemize}[leftmargin=0.6cm]

\smallskip
\item
For types $B$ and $D$ ($n = 0$, $N = m$), the parameter $\xi_{\mathrm{JLM}} = q^{2-N}$ in~\cite{jlm3} coincides with our
$\xi = q^{-m+2}$, and the series $f_{\mathrm{JLM}}(z)$ in~\cite[(1.6)]{jlm3} is given by
\begin{equation}\label{eq:f-JLM}
  f_{\mathrm{JLM}}(z) = \prod_{k=0}^{\infty}
    \frac{(1 - \xi_{\mathrm{JLM}}^{2k}z)(1 - q^{-2}\xi_{\mathrm{JLM}}^{2k+1}z)(1 - q^{2}\xi_{\mathrm{JLM}}^{2k+1}z)
          (1 - \xi_{\mathrm{JLM}}^{2k+2}z)}
         {(1 - \xi_{\mathrm{JLM}}^{2k-1}z)(1 - \xi_{\mathrm{JLM}}^{2k+1}z)(1 - q^{2}\xi_{\mathrm{JLM}}^{2k}z)
          (1 - q^{-2}\xi_{\mathrm{JLM}}^{2k}z)} \,.
\end{equation}
In~\cite{jlm3}, the ratio between the evaluated universal $R$-matrix and the unitary $R$-matrix is given~by
(cf.~\cite[(1.7), (2.1), (3.1)]{jlm3})
\begin{equation}\label{eq:jlm-BD-comparison}
  g_{\mathrm{JLM}}(z) = (z - q^{-2})(z - \xi_{\mathrm{JLM}}) f_{\mathrm{JLM}}(z) = q^{-1}\xi f(z) \,,
\end{equation}
with $f(z)$ of~\eqref{eq:f-explicit}.

\smallskip
\item
For type $C$ ($m = 0$, $N = n$, with $n$ rather denoted by $2n$ in~\cite{jlm2}), our bilinear form~\eqref{eq:epsilon-form} on
$\fh^{*}$ satisfies $(\varepsilon_{i}, \varepsilon_{j}) = -\delta_{ij}$ ($1 \leq i, j \leq s$) since $\ol{i} = \bbar{1}$, which is
the negative of the bilinear form $(\epsilon_{i}, \epsilon_{j}) = \delta_{ij}$ in~\cite[\S1]{jlm2}. Accordingly, the quantum affine
algebra of~\cite[\S2.1]{jlm2} is identified with $\UdqV$ via $q_{\mathrm{JLM}} \mapsto q^{-1}$, under which their parameter
$\xi_{\mathrm{JLM}} = q_{\mathrm{JLM}}^{-n-2}$ coincides with our $\xi = q^{n+2}$. The series $f_{\mathrm{JLM}}(z)$ of~\cite[(1.6)]{jlm2}
is given by the same formula~\eqref{eq:f-JLM} with $q$ replaced by $q_{\mathrm{JLM}}$, and thus the ratio between the evaluated
universal $R$-matrix and the unitary $R$-matrix is given by (cf.~\cite[(1.7), (2.1), (3.1)]{jlm2})
\begin{equation}\label{eq:jlm-C-comparison}
  g_{\mathrm{JLM}}(z) = (z - q_{\mathrm{JLM}}^{-2})(z - \xi_{\mathrm{JLM}}) f_{\mathrm{JLM}}(z)
  = q\xi \frac{zq^{-1} - q}{zq - q^{-1}} f(z) \,,
\end{equation}
with $f(z)$ of~\eqref{eq:f-explicit}.
Here, the unitary $R$-matrix of~\cite[(3.1)]{jlm2} coincides with $\eta(z)^{-1}\ol{R}_{21}(z)$, where
$\eta(z) \coloneqq \frac{zq^{-1} - q}{zq - q^{-1}}$. Since $\eta(z)$ satisfies $\eta(z)\eta(z^{-1}) = 1$, the factor $\eta(z)$
appearing in the right-hand side of~\eqref{eq:jlm-C-comparison} simply reflects a different choice of the unitary normalization.

\end{itemize}
\end{Rem}

\begin{Rem}\label{rem:wlz-comparison}
We also note that Proposition~\ref{prop:R-univ-explicit} is compatible with the corresponding result of~\cite{wlz},
where the super $B$ type case (corresponding to $m$ being odd) is considered for the standard parity sequence
$\gamma_{V} = (\bbar{1}, \ldots, \bbar{1}, \bbar{0}, \ldots, \bbar{0})$. We note that our $m, n$ are rather denoted by
$2m+1, 2n$ in~\cite{wlz}; we shall also denote their spectral parameter $z$ and scalar series $f(z), g(z)$ by
$z_{\mathrm{WLZ}}$ and $\wt{f}(z_{\mathrm{WLZ}}), \wt{g}(z_{\mathrm{WLZ}})$, respectively.
While their parameter $q$ and bilinear form on $\fh^{*}$ coincide with ours, their parameters $z_{\mathrm{WLZ}}$ and
$\zeta$ are the inverses of our $z$ and $\xi$ (i.e.\ $z_{\mathrm{WLZ}} = z^{-1}$ and $\zeta = \xi^{-1}$),
under which their unitary $R$-matrix $\wt{R}(z_{\mathrm{WLZ}})$ in~\cite[(4.3)]{wlz} coincides with our
$\ol{R}(z) = \ol{R}(z_{\mathrm{WLZ}}^{-1})$, and the series $\wt{f}(z_{\mathrm{WLZ}})$ in~\cite[(4.4)]{wlz} is given by
(cf.~\eqref{eq:f-JLM})
\begin{equation*}
  \wt{f}(z_{\mathrm{WLZ}}) = \prod_{k=0}^{\infty}
    \frac{(1 - \zeta^{2k}z_{\mathrm{WLZ}})(1 - q^{-2}\zeta^{2k+1}z_{\mathrm{WLZ}})
          (1 - q^{2}\zeta^{2k+1}z_{\mathrm{WLZ}})(1 - \zeta^{2k+2}z_{\mathrm{WLZ}})}
         {(1 - \zeta^{2k-1}z_{\mathrm{WLZ}})(1 - \zeta^{2k+1}z_{\mathrm{WLZ}})
          (1 - q^{-2}\zeta^{2k}z_{\mathrm{WLZ}})(1 - q^{2}\zeta^{2k}z_{\mathrm{WLZ}})} \,.
\end{equation*}
In~\cite{wlz}, the ratio between the evaluated universal $R$-matrix and the unitary $R$-matrix is given~by
(cf.~\cite[(4.3)--(4.4)]{wlz})
\begin{equation*}
  \wt{g}(z_{\mathrm{WLZ}}) = (z_{\mathrm{WLZ}} - q^{2})(z_{\mathrm{WLZ}} - \zeta) \wt{f}(z_{\mathrm{WLZ}})
  = q^{3}\xi^{-1} f(z) \,,
\end{equation*}
with $f(z)$ of~\eqref{eq:f-explicit}.
\end{Rem}


\end{document}